\documentclass{amsart}
\usepackage[margin = 1.3 in]{geometry}
\usepackage{amsmath,amssymb,amsthm,pinlabel,tikz,hyperref,mathrsfs,color, thmtools}
\DeclareSymbolFontAlphabet{\amsmathbb}{AMSb}%
\usepackage{verbatim}
\usepackage{float}
\usepackage{caption}
\usepackage{subcaption}
\usepackage{enumitem}
\newcommand{\nc}{\newcommand}
\nc{\dmo}{\DeclareMathOperator}
\dmo{\ra}{\rightarrow}
\dmo{\Prob}{\mathbb{P}}
\dmo{\E}{\mathbb{E}}
\dmo{\N}{\mathbb{N}}
\dmo{\Z}{\mathbb{Z}}
\dmo{\Q}{\mathbb{Q}}
\dmo{\R}{\mathbb{R}}
\dmo{\C}{\mathcal{C}}
\dmo{\X}{\mathcal{X}}
\dmo{\U}{\mathcal{U}}
\dmo{\T}{\mathcal{T}}
\dmo{\F}{\mathcal{F}}
\dmo{\AC}{\mathcal{AC}}
\dmo{\w}{\omega}
\dmo{\MIN}{\mathcal{MIN}}
\dmo{\Mod}{Mod}
\dmo{\PMod}{PMod}
\dmo{\PMF}{\mathcal{PMF}}
\dmo{\Mat}{Mat}
\dmo{\supp}{supp}
\dmo{\UE}{\mathcal{UE}}
\dmo{\vol}{vol}
\dmo{\B}{B}
\dmo{\PB}{PB}
\dmo{\PR}{PSL(2,\mathbb{R})}
\dmo{\GL}{GL(k, \mathbb{C})}
\dmo{\SL}{SL(2, \mathbb{Z})}
\dmo{\Isom}{Isom}
\dmo{\RP}{\mathbb{R} \mathrm{P}}
\dmo{\I}{\mathcal{I}}
\dmo{\el}{\ell_{\C}}
\dmo{\NN}{\mathcal{N}}
\dmo{\rk}{rank}
\dmo{\tr}{tr}
\dmo{\llangle}{\langle\langle}
\dmo{\rrangle}{\rangle\rangle}
\dmo{\Unif}{Unif}
\dmo{\Out}{Out}
\dmo{\diam}{\operatorname{diam}}
\dmo{\Aut}{\operatorname{Aut}}
\dmo{\sumRho}{\mathscr{B}}
\dmo{\stopping}{\vartheta}
\dmo{\diffPivot}{\mathcal{P}}
\dmo{\diffEvenPivot}{\mathcal{Q}}
\dmo{\varGam}{\Upsilon}
\dmo{\prodSeq}{\Pi}
\dmo{\NSupp}{N_{supp}}
\dmo{\KSleep}{\mathnormal{K_{sleep}}}
\dmo{\Devi}{\upsilon}
\dmo{\DeviStr}{\iota}
\dmo{\DeviDisc}{\varrho}
\dmo{\sphere}{\mathcal{S}}
\dmo{\DeviUni}{\varsigma}
\dmo{\wVar}{\check{\w}}
\dmo{\muVar}{\eta}
\dmo{\sublinear}{\Delta}
\dmo{\axes}{\mathbf{Y}}
\dmo{\Stab}{\operatorname{Stab}}

\usetikzlibrary{decorations.markings, patterns}
\usetikzlibrary{decorations.pathmorphing}
\usetikzlibrary{calc, arrows, positioning}
\tikzset{->-/.style={decoration={
  markings,
  mark=at position #1 with {\arrow{>}}},postaction={decorate}}}

\nc{\nt}{\newtheorem}

\nt{theorem}{Theorem}

\newtheorem{thm}{{\bf Theorem}}[section]
\newtheorem{conv}[thm]{{\bf Convention}}
\newtheorem{lem}[thm]{{\bf Lemma}}
\newtheorem{cor}[thm]{{\bf Corollary}}
\newtheorem{prop}[thm]{{\bf Proposition}}
\newtheorem{fact}[thm]{Fact}

\newtheorem{claim}[thm]{Claim} 
\newtheorem{remark}[thm]{Remark}

\newtheorem{definition}[thm]{Definition}
\newtheorem{notation}[thm]{Notation}
\numberwithin{equation}{section}

\title[Differentiability of the escape rate]{Contracting isometries and differentiability of the escape rate}

\date{\today}
\author{Inhyeok Choi}

\address{%
		School of Mathematics, KIAS\\
		85 Hoegiro Dongdaemun-gu, Seoul, 02455, South Korea 
}
\email{%
        inhyeokchoi48@gmail.com
                }

\begin{document}

\begin{abstract}
Let $G$ be a countable group whose action on a metric space $X$ involves a contracting isometry. This setting naturally encompasses groups acting on Gromov hyperbolic spaces, Teichm{\"u}ller space, Culler-Vogtmann Outer space and CAT(0) spaces. We discuss continuity and differentiability of the escape rate of random walks on $G$. For relatively hyperbolic groups, CAT(-1) groups and CAT(0) cubical groups, we further discuss analyticity of the escape rate. Finally, assuming that the action of $G$ on $X$ is weakly properly discontinuous (WPD), we discuss continuity of the asymptotic entropy of random walks on $G$.


\noindent{\bf MSC classes:} 20F67, 30F60, 57M60, 60G50
\end{abstract}

\maketitle

\section{Introduction}\label{section:intro}

Let $G$ be a countable group acting on a geodesic metric space $(X, d)$ with a basepoint $o \in X$ and let $\mu$ be a probability measure on $G$. The \emph{first moment} and the \emph{(time-one) entropy} of $\mu$  are defined respectively by\[
L(\mu) := \sum_{g \in G} \mu(g) d(o, go), \quad H(\mu) := \sum_{g \in G} \mu(g) \log \mu(g).
\]
Since these quantities are subadditive with respect to convolution of measures, we can define two asymptotic quantities \[
l(\mu) := \lim_{n \rightarrow \infty} \frac{L(\mu^{\ast n})}{n}, \quad h(\mu) := \lim_{n \rightarrow \infty} \frac{H(\mu^{\ast n})}{n},
\]
called the \emph{escape rate} (or \emph{drift}) and the \emph{asymptotic} (or \emph{Avez}) \emph{entropy}. These quantities have been studied in connection with the analytic properties of the ambient group. For example, a countable group $G$ is nonamenable if and only if every admissible measure $\mu$ on $G$ has strictly positive escape rate $l(\mu) > 0$ \cite{kaimanovich1983entropy}. Moreover, for a probability measure $\mu$ on $G$ with finite  time-one entropy, its asymptotic entropy $H(\mu)$ vanishes if and only if the Poisson boundary for $(G, \mu)$ is trivial, i.e., there is no non-constant $\mu$-harmonic bounded function on $G$ (see \cite{MR324741}, \cite{derriennic1980quelques} and \cite{kaimanovich1983entropy}; see also \cite{karlsson2007linear} for the role of $L(\mu)$ in this problem).

Furthermore, these two quantities tell us how the random walk probes the $G$-orbit on $X$. When $L(\mu)$ is finite, almost every sample path $(Z_{n})_{n>0}$ escapes to infinity with rate $l(\mu)$, i.e., $l(\mu) = \lim_{n} d(o, Z_{n} o)/n$. Similarly, when $H(\mu)$ is finite, almost every sample path $(Z_{n})_{n >0}$ arrives at an $\operatorname{exp}(-n h(\mu))$-probable orbit point at time $n$, i.e., $h(\mu) = -\lim_{n} \log \mu^{\ast n} (Z_{n}) / n$. Finally, the volume growth $v := \limsup_{n} \ln(\# \{ g \in G : d(o, go) < n\}) / n$ tells us how many orbit points there are within a bounded distance. The Guivarc'h fundamental inequality (\cite{guivarch1980loi}, \cite{blachere2008asymptotic}) reads \[ h \le l v,\]and for word hyperbolic groups, the equality holds if and only if the limiting distributions for the $\mu$-random walk (the $\mu$-harmonic measure) and the orbit counting (the Patterson-Sullivan measure) are equivalent \cite{blachere2011harmonic}.

It is natural to ask  about the dependence of the quantities $h(\mu)$ and $l(\mu)$ on the underlying measure $\mu$. Indeed, Gou{\"e}zel, Math{\'e}us and Maucourant proved in \cite{gouezel2018entropy} that given a finite subset $S \subseteq G$ of a   non-virtually free word hyperbolic group, there exists a constant $C <1$ such that $h(\mu) \le Cl(\mu)v$ holds for all symmetric probability measures $\mu$ supported in $S$. One ingredient of their argument is the continuity of $h(\mu)$ and $l(\mu)$ with respect to the underlying measure $\mu$.

In this paper, we focus on group actions that involve contracting isometries (see Convention \ref{conv:main}). We say that a probability measure $\mu$ on $G$ is \emph{non-elementary} if the semigroup generated by the support of $\mu$ contains two independent contracting isometries of $X$. We say that a sequence of measures $\{\mu_{i}\}_{i>0}$ on $G$ \emph{converges simply} to $\mu$ if $\mu_{i}(g) \rightarrow \mu(g)$ for each $g\in G$. We recall the continuity of the escape rate:

\begin{thm}[{\cite[Proposition 5.15]{gouezel2022exponential}}]\label{thm:driftContinuity}
Let $G$ be a countable group acting on a metric space $(X, d)$ that involves two independent contracting isometries. Let $\mu$ be a non-elementary probability measure on $G$, and let $\{\mu_{i}\}_{i > 0}$ be a sequence of probability measures that converges simply to $\mu$ with $L(\mu_{i}) \rightarrow L(\mu)$. Then $l(\mu_{i})$ converges to $l(\mu)$.
\end{thm}

This result is originally due to Erschler and Kaimanovich for hyperbolic groups \cite{kaimanovich2013continuity} and due to Mathieu and Sisto for acylindrically hyperbolic groups and Teichm{\"u}ller space \cite{mathieu2020deviation}.   In \cite{masai2021drift}, Masai gave  another argument for Teichm{\"u}ller space. Gou{\"e}zel reproved this result for non-elementary isometry groups of Gromov hyperbolic spaces. His argument can be adapted to group actions involving contracting isometries. See \cite[Section 3.2, 3.3]{choi2022random1} for the necessary modification, where the alignment lemmata are presented).

Hugely motivated by the works of Gou{\"e}zel \cite{gouezel2022exponential}, Mathieu \cite{mathieu2015entropy} and Mathieu and Sisto \cite{mathieu2020deviation}, we aim to investigate higher regularity of the escape rate on various spaces. Given a signed measure $\eta$ on $G$, we define 
\[\begin{aligned}
\| \eta\|_{0} &:= \sum_{g \in G} |\eta(g)| \quad \textrm{(total variation)},\\
\| \eta\|_{1} &:= \sum_{g \in G} d(o, go) |\eta(g)| \quad \textrm{(first moment)},\\
\| \eta\|_{0, 1} &:= \|\eta\|_{0} + \|\eta\|_{1} = \sum_{g \in G} (d(o, go) + 1) |\eta(g)|.
\end{aligned}
\]

Our main results are as follows.

\begin{theorem}\label{thm:driftLipschitz}
Let $G$ be as in Convention \ref{conv:main} and let $\mu$ be a non-elementary probability measure on $G$ with finite first moment. Then there exists $C, \epsilon> 0$ such that \[
\big| l(\mu') - l(\mu) \big| \le C \| \mu' - \mu\|_{0, 1}
\]
for any probability measure $\mu'$ such that $\|\mu' - \mu\|_{0, 1} < \epsilon$.
\end{theorem}

\begin{theorem}\label{thm:driftDiff}
Let $G$ be as in Convention \ref{conv:main} and let $\mu$ be a non-elementary probability measure on $G$ with finite first moment, let $\eta$ be a signed measure that is absolutely continuous with respect to $\mu$ and such that $\|\eta\|_{0, 1} < \infty$, and let $\{\mu_{t} : t \in [-1, 1]\}$ be a family of probability measures such that \[
\| \mu_{t} - \mu - t \eta \|_{0, 1} = o(t).
\]
Then $l(\mu_{t})$ is differentiable at $t=0$ with the derivative \[
\lim_{n} \frac{1}{n} \sum_{i=1}^{n} \E_{\mu^{\ast i-1} \ast \eta \ast \mu^{\ast n-i}} d(o, go).
\]
Moreover, this derivative is continuous with respect to $\mu$ and $\eta$ in $\|\cdot\|_{0, 1}$-topology.
\end{theorem}

Analyticity of the escape rate of random walks on free groups and free products was discussed independently by Ledrappier and Gilch (see \cite{ledrappier2001free}, \cite{gilch2007rate}, \cite{ledrappier2012analyticity}). Ha{\"i}ssinsky, Mathieu and M{\"u}ller proved analyticity of the escape rate of random walks on surface groups by using automatic structures \cite{haissinsky2018renewal}. On Gromov hyperbolic groups, Lipschitz continuity and $C^{1}$-regularity of the escape rate is due to Ledrappier \cite{ledrappier2013regularity} and Mathieu \cite{mathieu2015entropy}, respectively. These results are generalized to acylindrically hyperbolic groups and Teichm{\"u}ller spaces by Mathieu and Sisto \cite{mathieu2020deviation}. Finally, Gou{\"e}zel established analyticity of the escape rate on hyperbolic groups in \cite{gouezel2017analyticity}. Theorem \ref{thm:driftLipschitz} and \ref{thm:driftDiff} generalize Mathieu-Sisto's differentiability of the escape rate to CAT(0) spaces and Culler-Vogtmann Outer space.

For certain examples, we have a different formula for the derivative of the escape rate.

\begin{theorem}\label{thm:driftDiffSqueeze}
Let $(X, G)$ be either: \begin{itemize}
\item Teichm{\"u}ller space (with the Teichm{\"u}ller metric or the Weil-Petersson metric) and the mapping class group;
\item Culler-Vogtmann Outer space $CV_{N}$ and $\Out(F_{N})$;
\item a CAT(-1) space and a countable group of its isometries;
\item a CAT(0) cube complex and a countable group of its isometries;
\item relatively hyperbolic group equipped with a Green metric;
\item the standard Cayley graph of a surface group, or;
\item the standard Cayley graph of a free product of nontrivial groups.
\end{itemize}
Let $\mu$ be a strongly non-elementary probability measure on $G$ with finite first moment, let $\eta$ be a signed measure on $G$ such that $\|\eta\|_{0, 1} < \infty$, and let $\{\mu_{t}: t \in [-1, 1]\}$ be a family of probability measures such that \[
\| \mu_{t} - \mu - t \eta \|_{0, 1} = o(t).
\]
Then $l(\mu_{t})$ is differentiable at $t=0$, with derivative \[
 \sigma_{1}(\mu, \eta) := \E_{\eta} d(o, go) -  \lim_{n \rightarrow \infty} \E_{\mu^{\ast n} \times \eta \times \mu^{\ast n}} \big[ d(o, g_{1}o) +d(o, g_{2}o) + d(o, g_{3}o) - d(o, g_{1}g_{2}g_{3}o)\big].
\]
Moreover, this derivative is continuous with respect to $\mu$ and $\eta$ in $\|\cdot\|_{0, 1}$-topology.
\end{theorem}

When the underlying space  exhibits a CAT(-1) behaviour near an axis of a contracting isometry, we can promote this $C^{1}$-regularity of the drift to  the $C^{\infty}$-regularity. For simplicity, we here present a result about $C^{2}$-regularity.

\begin{theorem}\label{thm:driftDiffSqueezeSecond}
Let $(X, G)$ be either: \begin{itemize}
\item a CAT(-1) space and its discrete isometry group;
\item a CAT(0) cube complex and its countable isometry group;
\item relatively hyperbolic group equipped with a Green metric;
\item the standard Cayley graph of a surface group, or;
\item the standard Cayley graph of a free product of nontrivial groups.
\end{itemize}
Then the function $\sigma_{1}(\mu, \eta)$ is differentiable in the following sense. Let $\mu$ be a strongly non-elementary probability measure on $G$ with finite first moment and let $\eta, \eta'$ be signed measures on $G$ such that $\|\mu\|_{0, 1}, \|\eta\|_{0, 1} < \infty$. Let $\{\mu_{t} : t \in [-1, 1]\}$ be a family of probability measures and $\{\eta_{t} : t \in [-1, 1]\}$ be a family of  balanced signed measures (i.e., $\sum_{g \in G} \eta_{t}(g) = 0$) such that \[
\|\mu_{t} - \mu - t \eta\|_{0, 1},\,\, \|\eta_{t} - \eta- t \eta'\|_{0, 1} = o(t).
\]
Then $\sigma_{1}(\mu_{t}, \eta_{t})$ is differentiable at $t = 0$, with derivative  \[\begin{aligned}
&\sigma_{2}(\mu, \eta, \eta') := \E_{\eta'} d(o, go)  
- \lim_{n\rightarrow \infty} \E_{\mu^{\ast n} \times \eta' \ast \mu^{\ast n}} \big[ d(o, g_{1}o) +d(o, g_{2}o) + d(o, g_{3}o) - d(o, g_{1}g_{2}g_{3}o)\big]\\
&- \sum_{i =1}^{\infty} \lim_{n \rightarrow \infty} \E_{\mu^{\ast n} \times \eta \times \left( \mu^{\ast (i-1)} \ast \eta \ast \mu^{\ast (n-i)}\right)} \big[ d(o, g_{1}o) +d(o, g_{2}o) + d(o, g_{3}o) - d(o, g_{1}g_{2}g_{3}o)\big] \\
&-\sum_{i =1}^{\infty} \lim_{n \rightarrow \infty} \E_{\left(\mu^{\ast (n-i)} \ast \eta \ast \mu^{\ast (i-1)}\right) \times \eta \times \mu^{\ast n}} \big[ d(o, g_{1}o) +d(o, g_{2}o) + d(o, g_{3}o) - d(o, g_{1}g_{2}g_{3}o)\big].
\end{aligned}
\]
Moreover, this derivative is continuous with respect to $\mu$, $\eta$ and $\eta'$ in $\|\cdot\|_{0, 1}$-topology.
\end{theorem}

As mentioned, analyticity of the escape rate on Gromov hyperbolic groups is due to Gou{\"e}zel \cite{gouezel2017analyticity}. Gou{\"e}zel's strategy is to utilize the contracting property of the Markov operator on a subset of the Busemann boundary of $G$. A crucial ingredient of his argument is that the values of a Busemann function $h$ on a bounded ball can determine the values of $h$ on a larger region (\cite[Lemma 3.4]{gouezel2017analyticity}). This phenomenon is expected when the underlying space is either a Gromov hyperbolic graph or a CAT(-1) space. It is not known whether such a strategy can be implemented for the group action on a general Gromov hyperbolic space.

Meanwhile, even restricted to the graph metric, our strategy requires much stronger (yet not global) rigidity of Busemann functions. Recall that for a hyperbolic group $G$, there is a continuous, (uniformly) finite-to-one surjection $\pi : \partial_{B} G \rightarrow \partial G$ from the Busemann boundary to the Gromov boundary. Theorem \ref{thm:driftDiffSqueezeSecond} applies to the Cayley graph of a hyperbolic group $G$ that contains a loxodromic element $g$ whose attracting point $g^{+} \in \partial G$ is a  one-to-one point of $\pi$, i.e., $\pi^{-1}(g^{+})$ consists of a single horofunction. This holds on the standard Cayley graph of a surface group, for example, but is not true in general. For example, on a ``ladder" $G = \Z \times \Z/2\Z$, the two ends of the group correspond to two boundary points on $\partial G$ and each of them has two preimages in $\partial_{B} G$. 

Another difference between Gou{\"e}zel's method and ours is the condition on the underlying measure $\mu$. Gou{\"e}zel's method depends on a spectral property of the Markov operator which requires $\mu$ to be admissible (i.e., the support of $\mu$ generates entire $G$) and have finite support (or at least with exponential tail). In contrast, we do not make use of spectral theory and  instead rely on the deviation inequality.  For this reason, our theorem applies to random walks that are not admissible (but non-elementary) nor finitely supported (but with finite first moment).

We next turn to the asymptotic entropy of $\mu$.  The following theorem describes continuity of the asymptotic entropy without moment condition.  To the best of the author's knowledge, this is new even on the free groups.

\begin{theorem}\label{thm:entropyContinuity}
Let $(X, G)$ be as in Convention \ref{conv:main} with a WPD contracting element $g \in G$, let $\mu$ be a probability measure with finite time-one entropy satisfying $\mu^{\ast n}(g^{m}) > 0$ for some $n, m > 0$, and let $\{\mu_{i}\}_{i > 0}$ be a sequence of probability measures that converges simply to $\mu$ with $H(\mu_{i}) \rightarrow H(\mu)$. Then $h(\mu_{i})$ converges to $h(\mu)$.
\end{theorem}

This continuity of the asymptotic entropy on hyperbolic groups was proven by Erschler and Kaimanovich \cite{kaimanovich2013continuity} when the underlying measure $\mu$ has finite first moment.  This was later generalized to measures with finite logarithmic moment by Gou{\"e}zel, Math{\'e}us and Maucourant \cite{gouezel2018entropy}. See \cite{masai2021drift} for the case of Teichm{\"u}ller space. While preparing the current manuscript, the author was informed by Anna Erschler and Joshua Frisch of their independent proof of Theroem \ref{thm:entropyContinuity}. 

For higher regularity (differentiability, analyticity, etc.) of the asymptotic entropy on hyperbolic groups, see \cite{gilch2011asymptotic}, \cite{ledrappier2012analyticity}, \cite{ledrappier2013regularity}, \cite{mathieu2015entropy}, \cite{gouezel2017analyticity} and \cite{mathieu2020deviation}. On hyperbolic groups and acylindrically hyperbolic groups, the known proofs of differentiability of the asymptotic entropy rely on the Martin boundary or the Green metric of $G$, and the Ancona inequality lies at the heart of the argument. Consequently, the underlying measures are subject to some moment condition (finite support or exponential tail). It seems very challenging to prove an analogue of Theorem \ref{thm:driftDiff} for asymptotic entropy, with finite entropy condition but without moment condition.

We prove Theorem \ref{thm:entropyContinuity} by combining Gou{\"e}zel's pivotal time technique in \cite{gouezel2022exponential} (modulo small modification as in \cite{choi2022random1}) and Chawla-Frisch-Forghani-Tiozzo's sublinear entropy growth of displacement \cite{chawla2022the-poisson}.

In the first half of Section \ref{section:prelim}, we explain notions related to contracting isometries and random walks  following  \cite{choi2022random1}.  We then summarize the geometric features of metric spaces possessing contracting isometries, including Teichm{\"u}ller space, CAT(0) spaces, CAT(0) cube   complexes, CAT(-1) spaces, Culler-Vogtmann Outer space and (relatively) hyperbolic groups. In Section \ref{section:pivoting}, we   explain Gou{\"e}zel's pivoting technique. In section \ref{section:diff}, we establish a single probabilistic estimate (Lemma \ref{lem:conditionedDist}) that leads to Theorem \ref{thm:driftLipschitz} and \ref{thm:driftDiff}. Once Lemma \ref{lem:conditionedDist} is given, the remaining argument for Theorem \ref{thm:driftLipschitz} and \ref{thm:driftDiff}   is purely probablistic and does not require much geometric information. In Section \ref{section:squeezing}, similarly, we single out a probabilistic estimate (Proposition \ref{prop:squeeze}) that leads to Theorem \ref{thm:driftDiffSqueeze}. In Section \ref{section:expSqueeze},  we prove a slightly more delicate estimate (Proposition \ref{prop:expSqueeze}) is  that leads to Theorem \ref{thm:driftDiffSqueezeSecond} is deduced. In Section \ref{section:entropy}, we recall Chawla-Forghani-Frisch-Tiozzo's sublinear growth of entropy of displacement \cite{chawla2022the-poisson} and a more complicated version of pivotal time construction \cite[Section 5C]{gouezel2022exponential}. By combining them, we establish Theorem \ref{thm:entropyContinuity}.

\subsection*{Acknowledgments}
The author thanks Kunal Chawla, Giulio Tiozzo and Joshua Frisch for explaining the usage of pivotal times for estimating asymptotic entropy. The author also thanks Ilya Gekhtman and Hidetoshi Masai for fruitful discussion. The author is grateful to the anonymous referee for suggesting many valuable comments that helped improved the exposition of this paper. The author is also grateful to the American Institute of Mathematics and the organizers and the participants of the workshop ``Random walks beyond hyperbolic groups'' in April 2022 for helpful and inspiring discussions. 

The author was supported by Samsung Science \& Technology Foundation (SSTF-BA1301-51), Mid-Career Researcher Program (RS-2023-00278510) through the National Research Foundation funded by the government of Korea, and by KIAS Individual Grants (SG091901, MG091901). This paper was revised while the author was at Cornell University. The author gratefully acknowledges the hospitality of the Department of Mathematics.

\section{Preliminaries}\label{section:prelim}

In this section we review geometric and probabilistic preliminaries. Subsection \ref{subsection:contracting}, \ref{subsection:RW} and \ref{subsection:align} are  summaries of the essence of \cite{choi2022random1}, regarding random walks and contracting isometries. These subsections are the only prerequisites  for Theorem \ref{thm:driftLipschitz}, \ref{thm:driftDiff} and \ref{thm:entropyContinuity}. The remaining subsections deal with individual spaces that appear in Theorem \ref{thm:driftDiffSqueeze} and \ref{thm:driftDiffSqueezeSecond}.

\subsection{Contracting isometries and BGIP isometries}\label{subsection:contracting}

Inspired   by many work including \cite{arzhantseva2015growth}, \cite{sisto2018contracting}, \cite{yang2019statistically}, \cite{kapovich2022random}, we focus on contracting isometries of metric spaces. The following definition is meant to  encompass Culler-Vogtmann Outer space  equipped with an asymmetric metric.

\begin{definition}[Metric space]\label{dfn:metric}
An \emph{(asymmetric) metric space} $(X, d)$ is a set $X$ equipped with a function $d : X \times X \rightarrow \mathbb{R}_{\ge 0}$ that satisfies the following: \begin{itemize}
\item (non-degeneracy) for each $x, y \in X$, $d(x, y) = 0$ if and only if $x = y$;
\item (triangle inequality) for each $x, y, z \in X$, $d(x, z) \le d(x, y) + d(y, z)$;
\item (local symmetry) for each $x \in X$, there exist $\epsilon, K>0$ such that $d(y, z) \le K d(z, y)$ holds for $y, z \in \left\{a \in X : \min\left(d(x, a), d(a, x)\right)<\epsilon\right\}$.
\end{itemize}
In this situation, we say that $d$ is a \emph{metric} on $X$.The metric $d$ is said to be \emph{symmetric} if $d(x, y) = d(y, x)$ holds for all $x, y \in X$. We define a symmetric metric called the \emph{symmetrization} of $d$ by \[
d^{sym}(x, y) := d(x, y) + d(y, x).
\]
We endow $(X, d)$ with the topology induced by $d^{sym}$.
\end{definition}

A \emph{path} is a map from an interval or a set of consecutive integers to the ambient space $X$. We say that two paths $\gamma$ and $\eta$ are \emph{$K$-fellow traveling} if there is a orientation-preserving reparametrization $\gamma'$ of $\gamma$ such that $d^{sym}(\gamma'(\tau), \eta(\tau)) < K$ for every input $\tau$. 

We define the Gromov product between $x$ and $y$ based at $z$ by \[
(x, y)_{z} := \frac{1}{2} [d(x, z) + d(z, y) - d(x, y)].
\]
When a path $\gamma : I \rightarrow X$ is understood, we denote by $\bar{\gamma}$ its reversal, i.e., the same path with orientation reversed.   Thus the beginning point of $\gamma$ becomes the ending point of $\bar{\gamma}$ and vice versa.

\begin{definition}[Quasigeodesics]\label{dfn:quasigeod}
A path $\gamma : I \rightarrow X$ from an interval or a set of consecutive integers $I$ is called a \emph{$K$-quasigeodesic} if \begin{equation}\label{eqn:quasiGeod}
\frac{1}{K} |t-s| - K \le d\big(\gamma(s), \gamma(t)\big) \le K|t-s| + K
\end{equation}
holds for all $s, t \in I$ such that $s<t$. If Inequality \ref{eqn:quasiGeod} holds for all $s, t \in I$, we say that $\gamma$ is a \emph{$K$-bi-quasigeodesic}. A path is called a \emph{geodesic} if $d\big(\gamma(s), \gamma(t)\big) = t-s$ holds for all  $s, t \in I$ such that $s < t$.

A metric space $X$ is said to be \emph{geodesic} if every ordered pair of points can be connected by a geodesic, i.e.,  for every $x, y \in X$ there exists a geodesic $\gamma : [a, b] \rightarrow X$ such that $\gamma(a) = x$ and $\gamma(b) = y$. We denote $\gamma$ by $[x, y]$.
\end{definition}

\begin{remark}\label{rem:asymmetric}
In general asymmetric metric spaces, the reversal of a geodesic may not be a geodesic. For an illuminating example due to Coulbois and Wiest, see \cite[Section 6]{francaviglia2011metric}.
\end{remark}

For a closed set $A \subseteq X$,  we denote by $\pi_{A}(\cdot)$ the closest point projection onto $A$. A $K$-quasigeodesic $\gamma$ is called a \emph{$K$-contracting axis} if \[
\diam(\pi_{\gamma}(B)) < K
\]
holds for every metric ball $B$ disjoint from the $K$-neighborhood of $\gamma$.

In metric spaces with a symmetric metric, contracting property is known to be equivalent to the \emph{bounded geodesic image property (BGIP)}: a $K$-bi-quasigeodesic $\gamma$ is called a $K$-BGIP axis if \[
\diam(\pi_{\gamma}(\eta)) < K
\]
for any geodesic $\eta$ that is disjoint from the $K$-neighborhood of $\gamma$. (Some authors call this equivalent notion   the contracting property.) In asymmetric metric spaces, these two properties are not equivalent in general. We stick to the language of BGIP axes, but they really mean contracting axes when the underlying metric is symmetric.

Fix a basepoint  $o\in X$. An isometry $g$ of $X$ is said to have \emph{$K$-BGIP} if its orbit $\{g^{i} o\}_{i \in \Z}$ is a $K$-BGIP axis. Two BGIP isometries $g, h \in \Isom(X)$ are said to be \emph{independent} if their orbits have bounded projections onto each other. A group $G \le \Isom(X)$ is said to be \emph{non-elementary} if it contains two independent BGIP isometries. \textbf{From now on, unless otherwise stated, we adopt}:

\begin{conv}\label{conv:main}
We fix a space $X$ with a possibly asymmetric metric $d$ and with a basepoint $o$. $G$ is a non-elementary countable isometry group of $X$, i.e., $G$ contains two independent BGIP isometries. When $X$ is assumed to be geodesic and $x, y \in X$, $[x, y]$ denotes an arbitrary geodesic connecting $x$ to $y$. 
\end{conv}

For genuine metric spaces with symmetric metrics, one can replace the above convention with: \begin{conv}\label{conv:weak}
We fix a metric space $(X, d)$ (where $d$ is symmetric) with a basepoint $o$. For each $x, y \in X$, $[x, y]$ denotes an arbitrary geodesic connecting $x$ to $y$. $G$ is a non-elementary countable isometry group of $X$, i.e., $G$ contains two independent contracting isometries.
\end{conv}

For each isometry $g$ of $X$, we define two norms \[
\|g\| := d(o, go), \quad \|g\|^{sym} := d(o, go) + d(go, o).
\]
Note that these norms are subadditive. Clearly, $\| \cdot \|^{sym} = 2 \| \cdot \|$ when $d$ is a genuine metric. On Outer space, $\| \cdot \|^{sym}$ is only bi-Lipschitz to $\| \cdot \|$ \cite[Theorem 24]{algom-kfir2012asymmetry}.

We say that a BGIP isometry $g$ of $X$  has the \emph{WPD (weak proper discontinuity) property}, or that $g$ is a WPD element, if for each $K\ge 0$ there exists $N > 0$ such that \[
\Stab_{K}(o, g^{N} o) := \big\{h \in G : d(o, ho) \le K, \, d\big(g^{N}o, hg^{N} o\big) \le K\big\}.
\]
is finite \cite[Section 3]{bestvina2002bounded}. For example, when the $G$-action on $X$ is properly discontinuous, every BGIP isometry of $X$ in $G$ is  WPD.

\subsection{Random walks}\label{subsection:RW}

Let $\{\mu_{i}\}_{i=1}^{n}$ be signed measures on $G$. By pushing forward the product measure $\mu_{1} \times \cdots \times \mu_{n}$ on $G^{n}$  via the composition map, we obtain the convolution measure $\mu_{1} \ast \cdots \ast\mu_{n}$ on $G$. When $\mu_{1} = \ldots = \mu_{n} = \mu$, we write by $\mu^{n}$ and $\mu^{\ast n}$ the $n$-product measure and the $n$-fold convolution measure of $\mu$, respectively.

Let $\mu$ be a probability measure on $G$. The (right) random walk on $G$ generated by $\mu$ is a Markov chain with the transition probability $p(x, y) = \mu(x^{-1}y)$.
\begin{notation}\label{notat:location}
Given a sequence of isometries $\mathbf{g} := (g_{i})_{i \in \Z}$, let \[
Z_{n} := \left\{ \begin{array}{cc} g_{1} \cdots g_{n} & n > 0 \\ id & n = 0 \\ g_{0}^{-1} \cdots g_{n+1}^{-1} & n < 0 \end{array}\right.
\]
\end{notation} Given the i.i.d. input $(g_{i})_{i \in \Z}$, the random variable $Z_{n}(g_{1}, \ldots, g_{n})$ models the $n$-th step position of the random walk starting from the identity. We can write the escape rate  of $\mu$ as \[
l(\mu) := \lim_{n \rightarrow \infty} \frac{1}{n} \E_{\mu^{\ast n}} \|Z_{n}\|.
\]
For $p > 0$, we define the $p$-th moment of  $\mu$ by \[
\E_{\mu} \|g\|^{p} = \sum_{g \in G} \mu(g) \, d(o, go)^{p}.
\]
The \emph{support} of $\mu$ is the set of elements $g \in G$ such that $\mu(g) > 0$. 

Let us now fix the metric space $X$ on which $G$ acts by isometries. A probability measure $\mu$ on $G$ is \emph{non-elementary} if the semigroup $\llangle \supp \mu \rrangle$ generated by the support of $\mu$ contains two independent BGIP isometries. We will say that a non-elementary probability measure $\mu$ is \emph{strongly non-elementary} in the following cases: 

\begin{itemize}
\item $G$ is the mapping class group  of an orientable surface $\Sigma_{g, n}$ of genus $g$ with $n$ punctures, where $3g - 3+n > 0$, and $X$ is the Teichm{\"u}ller space of $\Sigma_{g, n}$ with the Teichm{\"u}ller or the Weil-Petersson metric. In this case, since $\mu$ is assumed to be non-elementary, $\llangle \supp \mu \rrangle$ contains a pseudo-Anosov mapping class.
\item $G$ is the outer automorphism group $\Out(F_{N})$ of a free group of   rank $N \ge 3$, $X$ is the Culler-Vogtmann Outer space and $\llangle \supp \mu \rrangle$ contains a principal fully irreducible outer automorphism.
\item $G$ is a countable automorphism group of a CAT(0) cube complex and $\llangle \supp \mu \rrangle$ contains an automorphism that skewers a pair of strongly separated hyperplanes (see Fact \ref{fact:CAT(0)Bridge}).
\item $G$ is a countable isometry group of a CAT(-1) space or a relatively hyperbolic group and $\llangle \supp \mu \rrangle$ contains a loxodromic element.
\item $G = \langle a_{1}, \ldots, a_{g}, b_{1}, \ldots, b_{g} \, | \, \Pi_{i=1}^{g} [a_{i}, b_{i}] \rangle$ is the surface group and $\llangle \supp \mu \rrangle$ contains $a_{1}^{2g} b_{2}^{2g}$.
\item $G = F_{1} \ast F_{2} \ast \cdots \ast F_{N}$ is a free product of nontrivial groups and $\llangle \supp \mu \rrangle$ contains $w_{1}w_{2}$ for some $w_{i} \in F_{i} \notin \{id\}$ for $i=1, 2$.
\end{itemize}

Later, we will evaluate some RVs on $G^{\Z}$ with respect to several different measures. We introduce a notation for this purpose.

\begin{notation}\label{notat:prodExp}
Let $\{\mu_{i}\}_{i \in \Z}$ be signed measures on $G$, all but finitely many of which are probability measures. Let $F : G^{\Z} \rightarrow \mathbb{R}$ be measurable function. When $g_{i}$ are independently distributed according to $\mu_{i}$ $(i \in \Z)$, we denote the expectation of $F(\mathbf{g})$ by \[
\E_{\prod_{i \in \Z} \mu_{i}} F(\mathbf{g}),
\]
and also by  \[
\E_{ \prod_{i=1}^{\infty} \mu_{i} \,\otimes\, \mu_{0} \, \otimes\, \prod_{i=1}^{\infty} \mu_{-i}} F(\mathbf{g})
\]
when we need to specify the indices associated with each measure.
\end{notation}

\subsection{Alignment lemma}\label{subsection:align}

We fix a \emph{geodesic} metric space $X$. We summarize the content of Subsection 3.2 and 3.3 of \cite{choi2022random1}. We refer the readers to \cite{choi2022random1} for the proofs.

\begin{definition}[{\cite[Definition 3.5]{choi2022random1}}]\label{dfn:alignment}
For $i=1, \ldots, n$, let $\gamma_{i}$ be a path on $X$ whose beginning and ending points are $x_{i}$ and $y_{i}$, respectively. We say that $(\gamma_{1}, \ldots, \gamma_{n})$ is \emph{$C$-aligned} if \[
\diam\big(y_{i} \cup \pi_{\gamma_{i}}(\gamma_{i+1})\big) < C, \quad \diam\big(x_{i+1} \cup \pi_{\gamma_{i+1}} (\gamma_{i}) \big) < C
\]
hold for $i = 1, \ldots, n-1$.
\end{definition}

The following isclear: \begin{enumerate} \item If $(\gamma_{1}, \ldots,\gamma_{n})$ and $(\gamma_{n}, \ldots, \gamma_{m})$ are each $C$-aligned, then so is $(\gamma_{1}, \ldots, \gamma_{n}, \ldots, \gamma_{m})$. 
\item If $(\gamma_{1}, \ldots, \gamma_{n})$ is $C$-aligned, so is $(\bar{\gamma}_{n}, \ldots, \bar{\gamma}_{1})$.
\end{enumerate}

\begin{lem}[{\cite[Lemma 3.7]{choi2022random1}, \cite[Lemma 2.18]{choi2022random3}}]\label{lem:1segment}
For each $C>0$ and $K>1$, there exists $D=D(C, K) > \max(C, K)$ that satisfies the following.

Let   $\gamma$ and $\gamma'$ be $K$-BGIP axes whose beginning points are $x$ and $x'$, respectively. If $(\gamma, x')$ and $(x, \gamma')$ are $C$-aligned, then $(\gamma, \gamma')$ is $D$-aligned.
\end{lem}

The following is a combination of the results in \cite{yang2019statistically}; see \cite{choi2022random1} for another version.

\begin{prop}\label{prop:BGIPWitness}
For each $D>0$ and $K>1$, there exist $E = E(D, K) > \max(D, K)$ and $L = L(D, K)$ that satisfy the following. 

\begin{enumerate}
\item Let $x, y \in X$ and let $\gamma_{1}, \ldots, \gamma_{n}$ be $K$-BGIP axes whose domains are longer than $L$ and such that $(x, \gamma_{1}, \ldots, \gamma_{n}, y)$ is $D$-aligned. Then the geodesic $[x, y]$ has subsegments $\eta_{1}, \ldots, \eta_{n}$, in order from left to right, that are longer than $100E$ and such that $\eta_{i}$ and $\gamma_{i}$ are $0.1E$-fellow traveling. Moreover, the nearest point projection of $x$ onto $\gamma_{1} \cup \ldots \cup \gamma_{n}$ is equal to $\pi_{\gamma_{1}}(x)$.
\item Let $\gamma_{1}, \ldots, \gamma_{n}$ be $K$-BGIP axes such that: \begin{enumerate}
\item the beginning point of $\gamma_{i+1}$ and the ending point of $\gamma_{i}$ are the same for $i=1, \ldots, n-1$;
\item each $\gamma_{i}$ has domain longer than $L$, and 
\item $(\gamma_{1}, \ldots, \gamma_{n})$ is $D$-aligned.
\end{enumerate}
Then the concatenation   $\gamma_{1} \cup \ldots \cup \gamma_{n}$ is
 an $E$-BGIP axis.
\end{enumerate}
\end{prop}

\begin{proof}
Let us assume the hypothesis of Item (1). The first half of the conclusion is proved in \cite[Proposition 2.7]{yang2019statistically}, \cite[Proposition 3.11]{choi2022random1} and \cite[Proposition 2.20]{choi2022random3}. Now, note that $[x, y]$ contains disjoint subsegments $\eta_{1}, \ldots, \eta_{n}$, in order from left to right, such that $\eta_{i}$ and $\gamma_{i}$ are $0.1E$-fellow traveling. Moreover, each $\eta_{i}$ is longer than $100E$. From this, for each $i$ we deduce \[
d(x, \gamma_{i+1}) \ge d(x, \eta_{i+1}) - 0.1E \ge d(x, \eta_{i}) + 99E \ge d(x, \gamma_{i}) + 98E.
\]
This implies the second half of the conclusion.

Item (2) is proved in \cite[Proposition 2.9]{yang2019statistically}.
\end{proof}

\begin{remark}\label{rem:nonGeod}
We will later deal with non-geodesic metrics, namely, the Green metrics on a relatively hyperbolic group $G$. Note that we have other metrics on $G$   that are $G$-equivariantly quasi-isometric to $d$, namely, the word metrics on $G$. Since the word metrics are geodesic metrics, the alignment lemmata   apply to the contracting axes in $G$ with respect to them. For this reason, we always   use the word metrics when defining the alignment of paths in relatively hyperbolic groups. We will revisit this   issue in Section \ref{subsection:relhyp}.
\end{remark}

We now define the notion of Schottky set, a collection of BGIP axes (of uniform quality) that head into distinct directions. Given a sequence $s  = (s_{1}, \ldots, s_{n}) \in G^{n}$, we define: \[\begin{aligned}
\Pi(s) &:= s_{1} s_{2} \cdots s_{n}, \\
\Gamma(s) &:= \big(o, \,s_{1} o, \,s_{1} s_{2} o,\, \ldots, \,\Pi(s) o\big).
\end{aligned}
\]
Recall our notation for reversal: we have $\bar{\Gamma}(s) := (\Pi(s)o, \ldots,   s_{1} o, o)$.

\begin{definition}[cf. {\cite[Definition 3.11]{gouezel2022exponential}, \cite[Definition 3.14]{choi2022random1}}]\label{dfn:Schottky}
For a given constant $K_{0}>0$, we define: \begin{itemize}
\item $D_{0}= D(K_{0}, K_{0})$ be as in Lemma \ref{lem:1segment}, 
\item $E_{0}= E(D_{0}, K_{0})$, $L_{0} = L(D_{0}, K_{0})$ be as in Proposition \ref{prop:BGIPWitness}.
\end{itemize}
We say that $S \subseteq G^{n}$ is a \emph{long enough $K_{0}$-Schottky set} if: \begin{enumerate}
\item $n > L_{0}$;
\item $\Gamma(s)$ is $K_{0}$-contracting axis for all $s \in S$;
\item for each $x \in X$ we have \[
\#\Big\{ s \in S : \textrm{$\big(x, \Gamma(s)\big)$ and $\big(\Gamma(s), \Pi(s) x\big)$ are $K_{0}$-aligned}\Big\} \ge \# S - 1;
\]
\item for each $s, s' \in S$, $\big(\Gamma(s), \Pi(s)\Gamma(s')\big)$ is $K_{0}$-aligned.
\end{enumerate}

In addition, we say that $S$ is \emph{large} if it has cardinality at least 400. Once a long enough and large Schottky set $S$ is understood, its element $s$ is called a \emph{Schottky sequence} and the translates of $\Gamma(s)$ are called \emph{Schottky axes}. When $\mu$ is a probability measure on $G$ such that $S \subseteq (\supp \mu)^{n}$, then we say that $S$ is a  long enough and large Schottky set for $\mu$.
\end{definition}
Using the notation for the reversal, Item (3) above can be seen as: \[
\#\Big\{ s \in S : \textrm{$\big(x, \Gamma(s)\big)$ and $\big(\bar{\Gamma}(s), x\big)$ are $K_{0}$-aligned}\Big\} \ge \# S - 1.
\]
Also, Item (4) is requiring that $(\bar{\Gamma}(s), \Gamma(s'))$ is   $K_{0}$-aligned.

\begin{prop}[{\cite[Proposition 3.18]{choi2022random1}}] \label{prop:Schottky}
Let $\mu$ be a non-elementary probability measure on $G$. Then for each $N, L>0$, there exists $K = K(N) > 0$ and $n=n(N, L) > L$ such that there exists a $K$-Schottky set $S \subseteq (\supp \mu)^{n}$ of cardinality $N$.
\end{prop}

\begin{definition}[{\cite[Definition 2.25]{choi2022random1}}]\label{dfn:semiAlign}
Let $S$ be a long enough Schottky set, let $D>0$, let $x, y \in X$, let $(\gamma_{1}, \ldots, \gamma_{N})$ be a sequence of Schottky axes, and let $(\gamma_{i(1)}, \ldots, \gamma_{i(n)})$ be its subsequence. If $(x, \gamma_{1}, \ldots, \gamma_{N}, y)$ is $D$-aligned, then we say that $(x, \gamma_{i(1)}, \ldots, \gamma_{i(n)}, y)$ is \emph{$D$-semi-aligned}.
\end{definition}

By applying Proposition \ref{prop:BGIPWitness} to   the Schottky axes for a long enough Schottky set, we observe:

\begin{cor}\label{cor:semiAlign}
Let $n > 0$ and let $S \subseteq G^{n}$ be a long enough $K_{0}$-Schottky set, associated with the constants $D_{0}, E_{0}$ as in Definition \ref{dfn:Schottky}. If $x, y \in X$ and Schottky axes $\gamma_{1}, \ldots, \gamma_{N}$ are such that \[
(x, \gamma_{1}, \gamma_{2}, \ldots, \gamma_{N}, y)
\]
is $D_{0}$-semi-aligned, then the $E_{0}$-neighborhood of $[x, y]$ contains $\gamma_{1}$, $\ldots$, $\gamma_{N}$.

More precisely, $[x, y]$ has subsegments $\eta_{1}, \ldots, \eta_{N}$, in order from left to right, that are longer than $100E_{0}$ and such that $\eta_{i}$ and $\gamma_{i}$ are $0.1E_{0}$-fellow traveling for each $i$.
\end{cor}

  For a sequence $\alpha = (g_{1}, \ldots, g_{m}) \in G^{m}$, we denote by $\alpha^{(k)}$ the $k$-self-concatenation of $\alpha$, i.e., \[
\alpha^{(k)} := \big(  \underbrace{g_{1}, \ldots, g_{m}, g_{1}, \ldots, g_{m}, \ldots, g_{1}, \ldots, g_{m}}_{\textrm{$k$ times}}\big).
\] 

We now record a technical lemma about Schottky sets. It says that we can incorporate a given BGIP isometry into the construction of Schottky sets. 

\begin{prop}\label{prop:SchottkyPA}
Let $(X, G)$ be as in Convention \ref{conv:main}. Let $A \subseteq G$ be a subset such that the semigroup $\llangle A \rrangle$ generated by $A$ contains two independent BGIP elements $g$ and $h$. Then for each $n>0$, there exists $m>0$ and a large enough and long Schottky set $S \subseteq A^{m}$   such that for each $s \in S$, $\Gamma(s)$ contains a translate of $(o, go, \ldots, g^{n} o)$ as a subsequence.

More precisely, let $\alpha$ be a sequence in $A$ such that $\Pi(\alpha) = g$. For each $N>0$ there exists $K_{0} = K_{0}(N)>0$ and a finite set $S'' \subseteq A^{m}$ of cardinality $N$ such that, the   set  \[
S := \{(s, \alpha^{(n)}, s) : s \in S''\}
\]
forms a $K_{0}$-Schottky set for any sufficiently large $n$.
\end{prop}

\begin{proof}
By the assumption, the BGIP isometry $g$ is realized as $g = \Pi(\alpha)$ by some sequence   $\alpha$ in $A$. Note that  $\Gamma(\alpha^{(m)})$ is quasi-isometric (uniformly in $m$) to its subpath $(o, go, \ldots, g^{m} o)$. Because $g$ is a BGIP isometry, there exists a uniform constant $K_{g}$ such that $\Gamma(\alpha^{(m)})$ is a  $K_{g}$-BGIP axis for any $m \in \Z_{>0}$.

Let $K_{temp}$ be $K(N+1)$ as in Proposition \ref{prop:Schottky} or $K_{g}$, whichever larger. Let: \begin{itemize}
\item $D_{temp} = D(K_{temp}, K_{temp}) > K_{temp}$ be as in Lemma \ref{lem:1segment},
\item $K_{0} = E(D_{temp}, K_{temp}) > D_{temp}$, $L_{0} = L(D_{temp}, K_{temp})$ be as in Proposition \ref{prop:BGIPWitness}, 
\item $D_{0} = D(K_{0}, K_{0})$ be as in Lemma \ref{lem:1segment},
\item $E_{0} = E(D_{0}, K_{0})$, $L_{1} = L(D_{0}, K_{0})$ be as in Proposition \ref{prop:BGIPWitness}.
\end{itemize}

By Proposition \ref{prop:Schottky}, there exists a $K_{temp}$-Schottky set $S' \subseteq A^{m}$ of cardinality $N+1$, with $m > L_{0} + L_{1}$. Now for $n > L_{0} + L_{1} + 2E_{0} + 3$,  $\Gamma(\alpha^{(n)})$ is a $K_{g}$-BGIP axis whose domain is longer than $L_{0}$, $L_{1}$ and $2E_{0}+2$. 

By the Schottky property of $S'$, we can take a subset $S''\subseteq S'$ of cardinality $N$ such that \[
\big(\Pi(\alpha^{(n)})^{-1} o, \Gamma(s)\big)\,\,\textrm{and}\,\,\big(\Gamma(s), \Pi(s)\Pi(\alpha^{(n)}) o\big)\,\,\textrm{are $K_{temp}$-aligned}. \quad( \forall \, s \in S'').
\] 
Note that for every path $\gamma$, $\big((\textrm{beginning point of $\gamma$}),\, \gamma\big)$ and $\big(\gamma, \,(\textrm{ending point of $\gamma$}) \big)$ are always $0$-aligned. This means that $\big(o, \Gamma(s)\big)$ and $\big(\Gamma(s), \Pi(s) o \big)$ are $0$-aligned for each $s \in S''$. Note also that $\Gamma(\alpha^{(n)})$ and $\Gamma(s)$ for $s \in S'$ are $K_{temp}$-BGIP axes whose domain is longer than $L_{0}$. Lemma \ref{lem:1segment} implies that \begin{equation}\label{eqn:alignSchottkySingle}
\Big(\Gamma(s),\, \Pi(s) \Gamma(\alpha^{(n)}), \,\Pi(s) \Pi(\alpha^{(n)}) \Gamma(s) \Big)\,\,\textrm{is $D_{temp}$-aligned} \quad (\forall \,s \in S'').
\end{equation} We now consider the set \[
S:= \Big\{ \big(s, \alpha^{(n)}, s\big) : s \in S''\Big\}.
\]
By Display \ref{eqn:alignSchottkySingle} and Proposition \ref{prop:BGIPWitness}, $\Gamma(s, \alpha^{(n)}, s)$ is a $K_{0}$-BGIP axis for each $s \in S''$. Next, given $x \in X$, the Schottky property of $S'$ tells us that \begin{equation}\label{eqn:alignSchottkyTriple}
\big( x, \Gamma(s)\big)\,\,\textrm{and}\,\, \big(\Gamma(s), \Pi(s) x \big)\,\,\textrm{are $K_{temp}$-aligned}
\end{equation}
for all $s \in S'$ except at most one. Moreover, when the condition in Display \ref{eqn:alignSchottkyTriple} holds, then \[
\pi_{\Gamma(s, \alpha^{(n)}, s)}(x) = \pi_{\Gamma(s)} (x) \subseteq N_{K_{temp}}(o)
\]
by Display \ref{eqn:alignSchottkySingle}, Proposition \ref{prop:BGIPWitness} and the fact that $\Gamma(s)$ and $\Gamma(\alpha^{(n)})$ are $K_{temp}$-BGIP axes whose domains are longer than $L_{0}$. Hence, $\big(x, \Gamma(s, \alpha^{(n)}, s)\big)$ is $K_{temp}$-aligned. For a similar reason, the condition in Display \ref{eqn:alignSchottkyTriple} implies that $\big(\Gamma(s, \alpha^{(n)}, s), \Pi(s) \Pi(\alpha^{(n)})\Pi(s) x\big)$ is $K_{temp}$-aligned. Summing up, $S$ satisfies the requirements for   the $K_{0}$-Schottky property. Furthermore, $\#S = \#S'' = N$ holds, and each element of $S$ has domain longer than $L_{1}$. Hence, $S$ is a long enough and large Schottky set.
\end{proof}

\subsection{Squeezing isometries} \label{subsection:squeezing}

We now   describe a geometric ingredient for Theorem \ref{thm:driftDiffSqueeze} and \ref{thm:driftDiffSqueezeSecond}.

\begin{definition}\label{dfn:squeezing}
Let $(X, d)$ be a (possibly asymmetric) geodesic metric space. Let $\gamma_{1}, \ldots, \gamma_{n}$ be $K$-BGIP axes in $X$.

We say that the sequence $(\gamma_{1}, \ldots, \gamma_{n})$ is \emph{$(K, \epsilon)$-squeezing} if the following holds. Let $x_{1}, x_{2}, y_{1}, y_{2} \in X$ be points such that $(x_{i}, \gamma_{1})$ is $K$-aligned and $(\gamma_{n}, y_{i})$ is $K$-aligned for $i=1, 2$. Then we have \[
\big| d(x_{1}, y_{1}) + d(x_{2}, y_{2}) - d(x_{1}, y_{2}) - d(x_{2}, y_{1}) \big| < \epsilon.
\]
\end{definition}

In particular, a BGIP axis $\gamma$ is said to be $(K, \epsilon)$-squeezing if, for every quadruple of points $x_{1}, x_{2}, y_{1}, y_{2} \in X$ such that $(x_{i}, \gamma, y_{i})$ is $K$-aligned for $i=1, 2$, we have \[
\big| d(x_{1}, y_{1}) + d(x_{2}, y_{2}) - d(x_{1}, y_{2}) - d(x_{2}, y_{1}) \big| < \epsilon.
\]

Since we   have chosen to use the word metric when discussing alignment in a relatively hyperbolic group, the definition reads as follows.

\begin{definition}\label{dfn:squeezingGreen}
Let $G$ be a relatively hyperbolic group equipped with a Green metric $d$. Let $\gamma_{1}, \ldots, \gamma_{n}$ be paths in $G$ that are $K$-contracting with respect to the word metric.

We say that $(\gamma_{1}, \ldots, \gamma_{n})$ is \emph{$(K, \epsilon)$-squeezing} if the following holds. Let $x_{1}, x_{2}, y_{1}, y_{2} \in G$ be points such that $(x_{i}, \gamma_{1})$ is $K$-aligned and $(\gamma_{n}, y_{i})$ is $K$-aligned with respect to the word metric for $i=1, 2$. Then we have \[
\big| d(x_{1}, y_{1}) + d(x_{2}, y_{2}) - d(x_{1}, y_{2}) - d(x_{2}, y_{1}) \big| < \epsilon.
\]
\end{definition}

Archetypal examples of squeezing axes are geodesic segments in an $\mathbb{R}$-tree. In Figure \ref{fig:tree},   four points are arranged so that $x_{1}$ and $x_{2}$ are on the left of $\gamma$ while $y_{1}$ and $y_{2}$ are on the right of $\gamma$. Suppose that $(x_{i}, \gamma, y_{i})$ are $K$-aligned and that $\gamma$ is longer than $2K$. Define $p$ to be the rightmost one between $\pi_{\gamma}(x_{1})$ and $\pi_{\gamma}(x_{2})$, and $q$ to be the leftmost one between $\pi_{\gamma}(y_{1})$ and $\pi_{\gamma}(y_{2})$. Then the assumption tells us that $p$ is   to the left of $q$, and we have \[
d(x_{i}, y_{j}) = d(x_{1}, p) + d(p, q) + d(q, y_{j})  \quad (i, j \in \{1, 2\}).
\]
This implies that $d(x_{1}, y_{1}) + d(x_{2}, y_{2}) - d(x_{1}, y_{2}) - d(x_{2}, y_{1})=0$.

\begin{figure}
\begin{tikzpicture}
\draw[very thick] (0, 0) -- (5, 0);
\fill (-0.4, 0.9) circle (0.07);
\fill (0.5, -0.9) circle (0.07);
\fill (4.3, 1.1) circle (0.07);
\fill (3.6, -1.3) circle (0.07);
\draw[dashed] (-0.4, 0.9) -- (-0.4, 0) -- (0, 0);
\draw[dashed] (0.5, -0.9) -- (0.5, 0);
\draw[dashed] (4.3, 1.1) -- (4.3, 0);
\draw[dashed] (3.6, -1.3) -- (3.6, 0);
\draw (-0.4, 1.15) node {$x_{1}$};
\draw (0.5, -1.15) node {$x_{2}$};
\draw (4.3, 1.35) node {$y_{1}$};
\draw (3.6, -1.55) node {$y_{2}$};
\draw (0.5, 0.25) node {$p$};
\draw (3.6, 0.25) node {$q$};
\draw (5.22, 0) node {$\gamma$};
\end{tikzpicture}
\caption{  Four points and a geodesic in an $\mathbb{R}$-tree.}
\label{fig:tree}
\end{figure}

In general, when the $K$-alignment of $(x_{i}, \gamma, y_{j})$ forces $[x_{i}, y_{j}]$ to pass nearby the midpoint of $\gamma$, then a similar computation is possible and $\gamma$ is $(K, \epsilon)$-squeezing for some $\epsilon$. In view of Proposition \ref{prop:BGIPWitness}, an aligned sequence of sufficiently long BGIP axes is squeezing. In the following subsections, we will see that the squeezing constant can be controlled explicitly in many spaces.

\subsection{Teichm{\"u}ller space} \label{subsection:Teich}

Let $\Sigma$ be a connected, finite-type hyperbolic surface. Its Teichm{\"u}ller space $\T(\Sigma)$ carries two canonical metrics, namely, the Teichm{\"u}ller metric $d_{\T}$ and the Weil-Petersson metric $d_{WP}$. The mapping class group $\Mod(\Sigma)$ of $\Sigma$ acts on $X$ by isometries with respect to both metrics.

A mapping class in $\Mod(\Sigma)$ either is finite-order, preserves a multicurve or exhibits Anosov dynamics on $\Sigma$ away from finitely many points. The last ones are called pseudo-Anosov mapping classes, and they are contracting isometries with respect to both $d_{\T}$ and $d_{WP}$ (see \cite[Contraction Theorem]{minsky1996quasi-projections}, \cite[Theorem 8.1]{rafi2014hyperbolicity}; \cite[Proposition 8.1]{bestvina2009higher}). Conversely, the orbit of a non-pseudo-Anosov mapping class lies in a quasi-flat   and is not contracting. Since the mapping class group action is properly discontinuous, every pseudo-Anosov mapping class is WPD. In this context, a probability measure $\mu$ on $\Mod(\Sigma)$ is non-elementary iff the semigroup generated by $\supp \mu$ contains two independent pseudo-Anosov mapping classes.

Each pseudo-Anosov mapping class has two fixed points on the Thurston boundary of $(\T(\Sigma)
, d_{\T})$, both uniquely ergodic. Masur's stability result \cite[Theorem 2]{masur1980uniquely} tells us that these fixed points are connected by a unique Teichm{\"u}ller geodesic. Using this, we can observe:

\begin{fact}\label{fact:TeichSqueeze}
Let $\varphi$ be a pseudo-Anosov mapping class and let $\gamma : \mathbb{R} \rightarrow \T(\Sigma)$ be its invariant Teichm{\"u}ller geodesic. Then for each $\epsilon > 0$ there exists $T>0$ such that, if $x, y \in \T(\Sigma)$ satisfy \[
\pi_{\gamma}(x) \ni \gamma(t_{1}), \,\, \pi_{\gamma}(y) \ni \gamma(t_{2})
\] for some $t_{2} > t_{1} + 3T$, then $[x, y]$ contains a subsegment $\eta$ that $\epsilon$-fellow travels with $\gamma|_{[t_{1} + T, t_{2} - T]}$ (with respect to $d_{\T}$).
\end{fact}

This fact is well-known to the experts (cf. \cite[Lemma 5.3]{eskin2011counting}). In view of Proposition \ref{prop:Schottky}, Fact \ref{fact:TeichSqueeze} has the following consequence:

\begin{prop}\label{prop:TeichSqueeze}
Let $(X, d)$ be the Teichm{\"u}ller space with the Teichm{\"u}ller metric, let $G$ be the mapping class group and let $o \in X$. Let $\mu$ be a non-elementary probability measure on $G$ and let $\epsilon > 0$. Then there exists a long enough and large Schottky set $S$, associated with the constants $D_{0}, E_{0}$ as in Definition \ref{dfn:Schottky}, such that every Schottky axis is $(D_{0}, \epsilon)$-squeezing.
\end{prop}

The same conclusion for the Weil-Petersson geometry follows from the uniqueness of a recurrent Weil-Petersson geodesic connecting two ending laminations (Theorem 1.1, \cite{brock2010asymptotics}). 

\subsection{CAT(0) spaces} \label{subsection:CAT(0)}

CAT(0) spaces are geodesic spaces where geodesic triangles are not fatter than the Euclidean triangles with the same side lengths. Typical examples   of CAT(0) spaces are complete Riemannian manifolds with non-positive sectional curvature,   such as  Euclidean spaces, $n$-dimensional hyperbolic spaces and their products. The Weil-Petersson metric renders Teichm{\"u}ller space CAT(0): this is indeed the context of Proposition \ref{prop:bestvina2009Contracting} below.

If an isometry of a CAT(0) space preserves a bi-infinite geodesic that does not bound a flat half-plane, then we call it \emph{rank-1}. Rank-1 isometries are essential for our theory due to: \begin{prop}[{\cite[Theorem 5.3]{bestvina2009higher}}] \label{prop:bestvina2009Contracting}
Let $X$ be a proper CAT(0) space. Then an isometry $g$ of $X$ is rank-1 if and only if it is contracting.
\end{prop}

We have plenty of contracting isometries in the following settings.

\begin{fact}[{\cite[Theorem 6.5]{bestvina2009higher}}] \label{prop:bestvina2009higher}
Let $X$ be a proper CAT(0) space. Suppose that the action of a group $\Gamma$ of isometries of $X$ contains a WPD element (see \cite[Definition 6.4]{bestvina2009higher}). Then $\Gamma$ is non-elementary.
\end{fact}

\begin{fact}[{\cite[Corollary 5.4]{hamenstadt2009rank}}]
Let $X$ be a proper CAT(0) space that admits a rank-1 isometry. Suppose that the limit set of a group $\Gamma$ of isometries of $X$ on the visual boundary has at least 3 points, and that $\Gamma$ does not globally fix a point in $\partial X$. Then $\Gamma$ is non-elementary.
\end{fact}

\begin{fact}[{\cite[Proposition 3.4]{caprace2010rank-one}}]
Let $X$ be a proper CAT(0) space that admits a rank-1 isometry. Suppose that a group $\Gamma$ of isometries of $X$does not globally fix a point in $\partial X$ nor stabilize a geodesic line. Then $\Gamma$ is non-elementary.
\end{fact}

\subsection{CAT(0) cube   complexes} \label{subsection:CAT(0)Cube}

CAT(0) cube   complexes are CAT(0) metric spaces  obtained by gluing Euclidean cells. A cube complex is CAT(0) if and only if it is simply connected and each vertex   satisfies a non-positive curvature condition,  namely, that its link is a flag complex. Originally proposed by Gromov, CAT(0) cube   complexes have become one of the central objects in geometric group theory. One instance of their importance lies in the     proof of the Virtual Haken Theorem \cite{agol2013virtual}, a major breakthrough in 3-manifold theory. 

We refer the readers to \cite{caprace2011rank} and \cite{chatterji2016median} for basic terminology about CAT(0) cube   complexes.   CAT(0) cube complexes often contain plenty of halfspaces, which cut the space into two (closed) halfspaces. Each halfspace $\mathfrak{h}$ is associated with a hyperplane that we denote by $\hat{\mathfrak{h}}$, and also with   the complementary halfspace $\overline{X \setminus \mathfrak{h}}$ that we denote by $\mathfrak{h}^{\ast}$. We say that two parallel hyperplanes of a CAT(0) cube complex are \emph{strongly separated} if no hyperplane is transverse to both of them. 

  We record Carpace-Sageev's rank rigidity theorem for CAT(0) cube complexes.

\begin{fact}\label{fact:cubeDist}
Let $X$ be a CAT(0) cube complex, let $x, y \in X$, and let $x \in \mathfrak{h}_{1} \subsetneq \mathfrak{h}_{2} \subsetneq \ldots \subsetneq \mathfrak{h}_{n} \not\ni y$ be a nested sequence of distinct halfspaces between $x$ and $y$. Then we have $n \le d(x, y)$.
\end{fact}

\begin{prop}[{\cite[cf. Theorem A]{caprace2011rank}}]\label{prop:capraceSageev}
Let $X$ be   a finite-dimensional CAT(0) cube complex and $G \le \Aut(X)$ be a group that    neither globally fix a point nor stabilize a 1-dimensional flat in $X \cup \partial X$. Then $X$ contains a geodesically convex $G$-invariant subcomplex $Y$ such that either $Y$ is the product of two unbounded subcomplexes or or the action of $G$ on $Y$ is strongly non-elementary. 

In particular, if $G$ acts on an $X$ essentially, without fixing a point  in $X \cup \partial X$, and if $X$ is irreducible, then $G$ contains independent contracting isometries that skewers a pair of strongly separated hyperplanes.
\end{prop}

This result essentially follows from the Double Skewering Lemma of Caprace and Sageev in \cite{caprace2011rank} that we describe below, with an additional information from \cite{chatterji2016median} and \cite{fernos2018the-furstenberg-poisson}. 

\begin{fact}[{\cite[Proposition 3.2]{caprace2011rank}, \cite[Lemma 2.18, 2.24]{chatterji2016median}, \cite[Lemm 3.28]{fernos2018the-furstenberg-poisson}}] \label{fact:CAT(0)Bridge}
Let $G$ be a group that essentially acts on an irreducible, finite-dimensional CAT(0) cube complex $X$ without  fixing a point. Then for any halfspace $\mathfrak{h}$ of $X$ associated with a hyperplane $\hat{\mathfrak{h}}$, there exists a rank-1 isometry $g \in G$ such that $\hat{\mathfrak{h}}$ and $g\hat{\mathfrak{h}}$ are strongly separated and such that $\mathfrak{h} \subsetneq g\mathfrak{h}$. (We say that $g$ skewers $\hat{\mathfrak{h}}$ and $g\hat{\mathfrak{h}}$.) Moreover, there exists a geodesic segment $\mathcal{I}$ connecting $\hat{\mathfrak{h}}$ to $g\hat{\mathfrak{h}}$ such that \begin{equation}\label{eqn:CAT(0)Bridge}
d(x, y) = d(x, p) +d(p,y)
\end{equation}
holds for every $x \in \mathfrak{h}$, $y \in (g\mathfrak{h})^{\ast}$ and $p \in \mathcal{I}$. (We call $\mathcal{I}$ the \emph{bridge} between $\hat{\mathfrak{h}}$ and $g\hat{\mathfrak{h}}$.)
\end{fact}

Using this, we can observe that: 

\begin{prop}\label{prop:CubeSqueeze}
Let $(X, d)$ be a CAT(0) cube complex with the combinatorial metric, let $G$ be a countable automorphism group of $X$, and let $o$ be a vertex of $X$. Let $\mu$ be a strongly non-elementary probability measure on $G$ and $\epsilon > 0$. Then there exists a long enough and large Schottky set $S$, associated with the constants $D_{0}, E_{0}$ as in Definition \ref{dfn:Schottky}, such that every Schottky axis is $(E_{0}, 0)$-squeezing with respect to $d$.
\end{prop}

\begin{proof}
Since $\mu$   is strongly non-elementary, $\llangle \supp \mu \rrangle$ contains a skewer $g$ for strongly separated hyperplanes, say $\hat{\mathfrak{h}}$ and $g \hat{\mathfrak{h}}$, satisfying $\mathfrak{h} \subsetneq g \mathfrak{h}$. Let us pick $x \in g \mathfrak{h}\setminus \mathfrak{h}$. Since $d(x, o)$ is finite,   there are finitely many mutually disjoint hyperplanes between $x$ and $o$. Using this, we can deduce that $o \in g^{m} \mathfrak{h} \setminus g^{n} \mathfrak{h}$ for some $m \neq n$, and consequently, that $o \in g^{k+1} \mathfrak{h} \setminus g^{k} \mathfrak{h}$ for some $k$. By replacing $\mathfrak{h}$ with $g^{k}\mathfrak{h}$, we may assume that $o \in g\mathfrak{h} \setminus \mathfrak{h}$.

Keeping the choice of $g$, we pick $K_{0}$ as in Proposition \ref{prop:SchottkyPA}. This determines the constants $D_{0}$ and $E_{0}$. Let $n = 2E_{0} + 2$. By Proposition \ref{prop:SchottkyPA}, there exists a long enough and large $K_{0}$-Schottky set $S$ such that   for each $s \in S$, $\Gamma(s)$ contains a translate of $(o, go, \ldots, g^{n} o)$ as a subsequence.

It remains to prove that $\Gamma(s)$ is $(E_{0}, 0)$-squeezing for each $s \in S$. To show this, let $x, y \in X$ be such that $(x, \Gamma(s), y)$ is $E_{0}$-aligned. By Proposition \ref{prop:BGIPWitness}, $[x, y]$ contains a subsegment that is $E_{0}$-fellow traveling with $\Gamma(s)$. Hence, there exist $p, q \in [x, y]$, $p$ coming earlier than $q$, such that \[
d\big(p, \varphi_{s}o\big) < E_{0}, \quad d\big(q, \varphi_{s} g^{n} o\big) < E_{0}.
\] 
Since $d(p, \varphi_{s} o)< E_{0}$ and $\varphi_{s}  o \in \varphi_{s}g\cdot \mathfrak{h}$, Fact \ref{fact:cubeDist} tells us that $p \in \Pi(s) g^{1 + E_{0}} \mathfrak{h} \subseteq \Pi(s) g^{n - E_{0} - 1} h$. Similarly, from the ingredients $d(q, \varphi_{s} g^{n} o) < E_{0}$, $\varphi_{s} g^{n} o \notin \varphi_{s} g^{n} \mathfrak{h}$, we deduce that $q \notin \varphi_{s} g^{n - E_{0}} \mathfrak{h}$. We then apply Fact \ref{fact:CAT(0)Bridge} to conclude that \[
d(x, y) = d(x, p) + d(p, y)
\]
for any point $p$ in the bridge between $\varphi_{s} g^{n - E_{0}- 1} \hat{\mathfrak{h}}$ and $\varphi_{s} g^{n - E_{0}} \hat{\mathfrak{h}}$. Using this, given any $x_{1}, x_{2}, y_{1}, y_{2} \in X$ such that $(x_{i}, \Gamma(s), y_{i})$ is $E_{0}$-aligned, we observe  \[
d(x_{1}, y_{1}) + d(x_{2}, y_{2}) = d(x_{1}, p) + d(x_{2}, p) + d(p, y_{1}) + d(p, y_{2}) = d(x_{1}, y_{2}) + d(x_{2}, y_{1}). \qedhere
\]
\end{proof}

An important family of examples comes from right-angled Artin groups (RAAGs). Recall that a RAAG $\Gamma$ is associated with a simply connected CAT(0) cube complex $\tilde{X}_{\Gamma}$; $\Gamma$ acts properly and cocompactly on $\tilde{X}_{\Gamma}$, and the resulting quotient is called the \emph{Salvetti complex} $X_{\Gamma}$ of $\Gamma$. It is proved in \cite[Theorem 5.2]{behrstock2012divergence} that if $\Gamma$ is not a direct product, then the universal cover $\tilde{X}_{\Gamma}$ of the Salvetti complex admits a rank-1 isometry. 

\subsection{CAT(-1) spaces} \label{section:CAT(-1)}

CAT(-1) spaces are uniquely geodesic spaces where geodesic triangles are not fatter than the corresponding ones in the hyperbolic plane. Since the triangles in the hyperbolic plane are uniformly thin, so are the triangles in CAT(-1) spaces. In particular, CAT(-1) spaces are Gromov hyperbolic and all $K$-quasigeodesics are $D$-contracting for a constant $D = D(K)$ determined by $K$. Further, the following is true.

\begin{fact}\label{fact:CAT(-1)Squeeze}
There exists a constant $0<c<1$ such that the following holds.

Let $X$ be a CAT(-1) space and let $x, y, x', y' \in X$ be such that \[
d(x, x'), d(y, y') < 0.02 d(x, y).
\]
Then $[x', y']$ is   $c^{d(x, y)}$-close to the midpoint of $[x, y]$.
\end{fact}

\subsection{Culler-Vogtmann Outer space} \label{subsection:Outer}

We now turn to the Culler-Vogtmann Outer space equipped with the Lipschitz metric. We will only record the necessary facts; for precise definitions and proofs, we refer the readers to \cite{vogtmann2015outer}, \cite{francaviglia2011metric}, \cite{bestvina2014hyperbolicity} and \cite{kapovich2022random}.

The Culler-Vogtmann Outer space $CV_{N}$ parametrizes projectivized marked minimal metric graphs with fundamental group $F_{N}$. The outer automorphism group of $F_{N}$, denoted by $\Out(F_{N})$, naturally acts on $CV_{N}$ by change of marking. Outer space has a natural simplicial structure that is preserved by the action of $\Out(F_{N})$. The difference of two marked graphs   is  measured by the Lipschitz constant of an optimal map between them, which gives the Lipschitz metric on Outer space. This metric is not symmetric but is geodesic. It is necessary to define the neighborhoods of sets in terms of the symmetrization of $d_{CV}$, namely, \[
d^{sym}(x, y) := d_{CV}(x, y) + d_{CV}(y, x).
\]Hence, the $R$-neighborhood of a set $A\subseteq X$ is now defined by \[
\{x \in X : d^{sym}(x, A) \le R\}.
\]

From now on, when discussing Outer space with basepoint $o$, we only consider the standard geodesics between $\Out(F_{N})$-orbit points of $o$, which are canonical $d_{CV}$-geodesics proposed by Bestvina and Feighn in \cite[Proposition 2.5]{bestvina2014hyperbolicity}. For each $go, ho \in G \cdot o$ there exists a \emph{standard geodesic} between $go$ and $ho$, which is a concatenation $\gamma_{1}\gamma_{2}$ of two subsegments $\gamma_{1}$ and $\gamma_{2}$, where $\gamma_{1}$ is a rescaling segment on the simplex of $x$ and $\gamma_{2}$ is a folding path. (For the definition of rescaling/folding path, refer to \cite{bestvina2014hyperbolicity}.) Note that the orbit points $G \cdot o$ are all $\epsilon$-thick for some $\epsilon > 0$. This gives a uniform bound on the length of rescaling path: $l(\gamma_{1}) < \log (2/\epsilon)$ (\cite[Lemma 2.6]{dowdall2018hyperbolic}).

We say that an outer automorphism $\phi \in \Out(F_{N})$ is \emph{reducible} if there exists a free product decomposition $F_{N} = C_{1} \ast \cdots \ast C_{k} \ast C_{k+1}$, with $k \ge 1$ and $C_{1}, \ldots, C_{k+1} \neq \{e\}$, such that $\phi$ permutes the conjugacy classes of $C_{1}, \ldots, C_{k}$. If not, we say that $\phi$ is \emph{irreducible}. We say that $\phi$ is \emph{fully irreducible} if no power of $\phi$ is reducible. For us, an important point is that each fully irreducible automorphism $\phi \in \Out(F_{N})$ has \emph{bounded geodesic image property (BGIP)}, i.e., there exists $K>0$ \[
\diam( \pi_{\{\phi^{i} o: i \in \Z\}} (\eta)) < K
\]
holds for   every geodesic $\eta$ that does not enter the $K$-neighborhood of $\{\phi^{i} o : i \in \Z\}$ (\cite[Theorem 7.8]{kapovich2022random}, cf. \cite[Proposition 1.3]{choi2022random3}).   This implies the strongly contracting property due to Algom-Kfir \cite{algom-kfir2011strongly}.

We now show a finer stability exhibited by certain outer   automorphisms. For each $N \ge 3$, $\Out(F_{N})$ contains \emph{principal} fully irreducible outer    automorphisms (\cite[Example 6.1]{algom-kfir2019stable}). These automorphisms have   a lone axis on $CV_{N}$ (cf. \cite{mosher2016lone}, \cite[Lemma 3.1]{kapovich2022tree}). Moreover, we have: 
\begin{fact}[{\cite[Lemma 5.9]{kapovich2022random}}]\label{fact:OuterKapo}
Let $\varphi \in \Out(F_{N})$ be a fully irreducible outer automorphism with lone axis $\gamma: \mathbb{R} \rightarrow CV_{N}$ (parametrized by length), let $t \in \mathbb{R}$ and let $\epsilon \ge 0$. Then there exists $R > 0$ such that, if $x, y \in G \cdot o$ satisfies \[
\pi_{\gamma}(x) \in \gamma((-\infty, t-R]), \quad \pi_{\gamma}(y) \in \gamma([t+R, +\infty)),
\]
then $[x, y]$ passes through the $\epsilon$-neighborhood of $\gamma(t)$.
\end{fact}

This implies the following: 

\begin{fact}\label{fact:OuterSqueeze}
Let $(X, d)$ be the Culler-Vogtmann Outer space of rank $N$, let $G$ be the outer automorphism group $\Out(F_{N})$ and let $\mu$ be a strongly non-elementary probability measure on $G$. Then for each $\epsilon>0$, there   exists a long enough and large $K_{0}$-Schottky set $S$ for $\mu$, associated with the constants $D_{0}$ and $E_{0}$, such that each Schottky sequence is $(E_{0}, \epsilon)$-squeezing.
\end{fact}

\subsection{Relatively hyperbolic groups} \label{subsection:relhyp}

Relatively hyperbolic groups are generalizations of free products of groups and geometrically finite Fuchsian and Kleinian groups.   We refer the readers to \cite{bowditch2012relatively}, \cite{yaman2004a-topological}, \cite{osin2006relatively} and \cite{gerasimov2013quasi-isometric} for further details. We recall that a finitely generated group $G$ is \emph{word hyperbolic} if its Cayley graph is Gromov hyperbolic, or equivalently, if it acts geometrically on a Gromov hyperbolic space. Another characterization of hyperbolic Cayley graphs is that geodesics in hyperbolic Cayley graphs are $K$-contracting   \cite{papasoglu1995strongly}.

More generally, let $\{P_{\alpha}\}_{\alpha}$ be a   conjugacy invariant collection of subgroups of $G$. We say that $G$ is \emph{hyperbolic relative to the peripheral system $\{P_{\alpha}\}_{\alpha}$}  if it has a minimal, geometrically finite, convergence action on a compact space $X$ in a way that the maximal parabolic subgroups are $P_{\alpha}$'s. Here, if $G$ contains two hyperbolic elements with disjoint   pairs of fixed points, we say that $G$ is a \emph{non-elementary relatively hyperbolic group}.

Let us now fix a word metric $d_{S}$ for a finite generating set $S$ of a relatively hyperbolic group $G$. Then each hyperbolic element of $G$ is contracting with respect to $d_{S}$ (\cite[Lemma 7.2]{yang2014growth}, \cite[Proposition 8.5]{gerasimov2016quasiconvexity}). Meanwhile, there is another natural metric on $G$ studied by Blach{\`e}re, Ha{\" i}ssinsky and Mathieu \cite{blachere2011harmonic}. Given a symmetric, finitely supported and admissible probability measure $\mu$ on $G$, we define the Green function associated to $\mu$ by \[
G^{\mu}(g) = \sum_{n=0}^{\infty} \mu^{\ast n} (g)
\]
and the Green metric with respect to $\mu$ by \[
d_{G}^{\mu}(g, h) := \log G^{\mu}(id) - \log G^{\mu}(g^{-1} h).
\]

\begin{remark}[{\cite[Proposition 3.5]{blachere2011harmonic}}]\label{rem:Green}
Green metrics are not geodesic in general but are roughly geodesic \cite{MR1100282}. That means, there exists a uniform constant $C\ge 0$ such that for each $g, h \in G$ there is a $C$-coarse geodesic connecting $g$ to $h$, which is a path $\gamma : [0, L] \rightarrow G$ such that $\gamma(0) = g$, $\gamma(L) = h$ and \[
|t_{1} - t_{2}| - C \le d_{G}^{\mu}\big(\gamma(t_{1}), \gamma(t_{2})\big) \le |t_{1} - t_{2}| + C
\]
holds for all $t_{1}, t_{2} \in [0, L]$. For this metric, we   say that a $K$-quasigeodesic $\gamma : I \rightarrow G$ is a $K$-contracting axis if $\diam(\pi_{\gamma}(\eta)) < K$ holds for all $C$-coarse geodesic $\eta$ that does not enter the $K$-neighborhood of $\gamma$.   We can then discuss the alignment lemmata by using $C$-coarse geodesics in place of geodesics. Since we already have a theory of alignment on the relatively hyperbolic group $G$ in terms of the word metric, we will not go deeper.
\end{remark}

Let $f$ be a hyperbolic element of $G$ and let $\epsilon > 0$. Then the orbit $\{f^{n}\}_{n \in \Z}$ of $f$ cannot stay long in the $\epsilon$-neighborhood of each peripheral subgroup $gP_{\alpha}$.   That means, there exists $L > 0$ such that $\{f^{n} o, f^{n+1} o, \ldots, f^{n+L} o\}$ is not contained in $N_{\epsilon}(gP_{\alpha})$ for any $n \in \Z$, $g \in G$ and any peripheral $P_{\alpha}$. We record the strong Ancona inequality along a geodesic that does not go deep into any translates of peripheral subgroups.

\begin{prop}[{\cite[Proposition 2.2, Theorem 3.14]{dussaule2021stability}}] \label{prop:greenDecay}
Let $G$ be a non-elementary relatively hyperbolic group, with a word metric $d_{S}$ and a Green metric $d_{G}^{\mu}$. Let $f$ be a hyperbolic element of $G$ and let $K>0$. Then there exists $L>0$ such that, if $\gamma_{1}, \ldots, \gamma_{N}$ are paths on $G$ of the form \[
\gamma_{i} = (g o, gf o, \ldots, gf^{L} o) \quad (g \in G)
\]
and if $x, y, x', y' \in G$ satisfy that $(x, \gamma_{1}, \ldots, \gamma_{N}, y)$, $(x', \gamma_{1}, \ldots, \gamma_{N}, y')$ are $K$-aligned (with respect to $d_{S}$), then we have\begin{equation}\label{eqn:greenDecay}
|d_{G}^{\mu}(x, y) + d_{G}^{\mu}(x', y') - d_{G}^{\mu}(x, y') - d_{G}^{\mu}(x', y)| \le e^{-N}.
\end{equation}
\end{prop}

Combining this with Proposition \ref{prop:SchottkyPA}, we deduce:

\begin{cor}\label{cor:greenDecay}
Let $G$ be a non-elementary relatively hyperbolic group   equipped with a word metric $d$ and a Green metric $d_{G}$. Then there   exists a long and large enough $K_{0}$-Schottky set $S$ such that for each $N$, every $D_{0}$-semi-aligned sequence of Schottky axes $(\gamma_{1}, \ldots, \gamma_{N})$ is ($E_{0}, e^{-N})$-squeezing.
\end{cor}

In some cases, we have an estimate analogous to Inequality \ref{eqn:greenDecay}   for the Cayley graph of $G$ with   a word metric. Two notable examples are  the surface groups and   the free products. First consider the Cayley graph of $G = \langle a_{1}, \ldots, a_{g}, b_{1}, \ldots, b_{g} | \prod_{i=1}^{g} [a_{i}, b_{i}] \rangle$, equipped with the word metric $d$ for the   standard generating set $\{a_{1}, \ldots, a_{g}, b_{1}, \ldots, b_{g}\}$. Note that: 

\begin{lem}[{cf. \cite[Lemma 2.3]{sambusetti2002growth}}]\label{lem:corridor}
Let $w = a_{i}^{k} b_{i+1}^{k} $, let $n > 2$ and let $x, y, z \in G$ be such that: \begin{enumerate}
\item there exists a geodesic $\gamma$ from $x$ to $y$ containing $w^{n}$ as a subword, and 
\item $\pi_{\gamma}(z)$ is left to the $w^{n}$-subword of $\gamma$ by more than $4g$.
\end{enumerate}
Then every geodesic on $G$ from $z$ to $y$ must contain the middle subsegment $w^{n-2}$ of $w^{n}$.
\end{lem}

In \cite{sambusetti2002growth}, the author describes the situation $z= x$. One can reduce the setting of Lemma \ref{lem:corridor} to that of  \cite[Lemma 2.3]{sambusetti2002growth}  by considering one pod of the geodesic triangle $\triangle xyz$, whose one side contains $w^{n}$ as subword. From this, we can deduce:

\begin{cor}\label{cor:corridor}
Let $w = a_{i}^{k} b_{i+1}^{k} $ and let $x, y \in G$ be such that \[
\pi_{\{w^{i} : i \in \Z\}}(x) = w^{n}, \,\, \pi_{\{w^{i} : i \in \Z\}}(y) = w^{m} \quad (n < m - 5).
\] 
Then any geodesic on $G$ from $x$ to $y$ must pass through the vertex $w^{n+2}$.

In particular, if we consider another pair of points $x', y' \in G$ such that \[
\pi_{\{w^{i} : i \in \Z\}}(x') = w^{n'}, \,\, \pi_{\{w^{i} : i \in \Z\}}(y') = w^{m'} \quad (n' \le n, m' \ge n+4),
\] 
Then we have \[
d_{G}(x, y) - d_{G}(x', y) = d_{G}(x, w^{n+2}) - d_{G}(x' ,w^{n+2}) = d_{G}(x, y') - d_{G}(x', y').
\]
\end{cor}

A similar bottlenecking phenomenon happens in the free product of nontrivial groups $G = \ast_{\alpha} G_{\alpha}$, with the word metric for the standard generating set $S = \bigcup \{\textrm{generating sets of $G_{\alpha}$}\}$. Indeed, if we take two distinct free factors $G_{\alpha}$, $G_{\beta}$ and pick nontrivial elements $a \in G_{\alpha}$ and $b \in G_{\beta}$, then $w = ab$ does the same job   as $w = a_{i}^{k} b_{i+1}^{k}$ in Lemma \ref{lem:corridor}.

\section{Pivoting technique: summary of results} \label{section:pivoting}

In this section, we review   Gou{\"e}zel's pivoting technique developed in \cite{gouezel2022exponential}.   The version we present is the one adapted to metric spaces with contracting isometries \cite{choi2022random1}, \cite{choi2022random2}. Readers who are familiar with the techniques in \cite{gouezel2022exponential} can assume the following proposition and proceed to Section \ref{section:diff}.

\begin{prop}\label{prop:gouezelRWEventual}
  Let $(X, G)$ be as in Convention \ref{conv:main}. Let $0< \epsilon \le 1$ and let $S \subseteq G^{M_{0}}$ be a long and large enough $K_{0}$-Schottky set. Then there exists $\kappa = \kappa(\epsilon, S)>0$ such that the following holds.

Let $\{ \mu_{i} : i \in \Z_{>0}\}$ be probability measures such that for each $s \in S$ and $k \in \Z$, \[
\big(\mu_{k + 1} \times \cdots \times \mu_{k + M_{0}} \big) ( s) > \epsilon.
\]
Let $g_{i}$ be independently distributed according to $\mu_{i}$ for each   $i \in \Z$.

Then for   every  $M$ integers $t(1), \ldots, t(M) \in \Z$   and for every $g_{t(1)}, \ldots, g_{t(M)} \in G$, we have  \[
\Prob \left(\left.\begin{array}{c} \textrm{$\exists$ Schottky axes $\gamma_{1}, \gamma_{2}, \ldots, \gamma_{\lceil \kappa n \rceil} \subseteq \{o, Z_{1} o, Z_{2} o, \ldots, Z_{n} o\}$ such that } \\
\textrm{$(Z_{-m} o, \gamma_{1}, \ldots, \gamma_{\lceil \kappa n \rceil}, Z_{n+l} o)$ is $D_{0}$-semi-aligned for all $m, l\ge0$} \end{array}\, \right| \, g_{t(1)}, \ldots, g_{t(M)}\right) \ge 1-\frac{1}{\kappa}   e^{-(\kappa n - M)}.
\]
\end{prop}

  The remaining of this section will be devoted to the proof of Proposition \ref{prop:gouezelRWEventual}.

We first reduce random walks (and more generally, products of independent increments) to a combinatorial model. Here is a toy case: suppose that a probability measure $\mu$ on $G$   assigns weights $\epsilon$ on $a, b \in G$ each. Then $\mu = \epsilon \cdot 1_{\{a, b\}} + (\mu - \epsilon 1_{\{a, b\}})$ holds. Hence, a random path $(g_{1}, \ldots, g_{n})$ following the law $\mu^{n}$ can be obtained as follows. First flag each index independently for probability $\epsilon$. Then we choose $g_{i}$ according to the law $1_{\{a, b\}}$ for flagged $i$, and according to the law $\frac{1}{1-\epsilon} \cdot (\mu - \epsilon 1_{\{a, b\}})$ for  unflagged $i$, all independently. The following proposition describes a situation analogous to the above, but with some indices reserved (which will account for conditional probabilities).

\begin{lem} \label{lem:pivotFirstInfty}
For each $0 < \epsilon \le 1$ and $M_{0} \in \Z_{>0}$, there exists $\kappa = \kappa(\epsilon, M_{0})>0$ such that the following holds.

Let    $M\in \Z_{>0}$. Let $S \subseteq G^{M_{0}}$ be a finite set of sequences in $G$. Let $t(1), \ldots, t(M)$ be $M$ distinct integers. Let $\mu_{1}, \mu_{2}, \ldots$ be probability measures such that \begin{enumerate}
\item  $\mu_{t(l)}$ is the Dirac measure concentrated on an element $g_{t(l)} \in G$ for $l=1, \ldots, M$, and  
\item   for each $s \in S$ and $k \in \Z$ such that $\{M_{0} k + 1, \ldots, M_{0} k + M_{0} \} \subseteq \{1, \ldots, n \} \setminus \{t(1), \ldots, t(M)\}$, we have\[
\big(\mu_{M_{0} k + 1} \times \cdots \times \mu_{M_{0} k + M_{0}} \big) ( s) > \epsilon 
\]
\end{enumerate}
Then there   exist a probability space $\Omega$, a measurable partition $\mathcal{Q}$ on $\Omega$ and random variables \[\begin{aligned}
\{g_{1}, g_{2}, \ldots \} \in G^{\Z_{>0}}, \\ 
\big\{\vartheta(1) < \vartheta(2) < \ldots \big\} &\subseteq \Z_{\ge 0} 
\end{aligned}
\]
such that the following hold: \begin{enumerate}
\item For each   $n\in \Z$ we have\[
\Prob \Big( \{\vartheta(1), \ldots, \vartheta(\lfloor \epsilon^{4} n/8M_{0} \rfloor - M) \} \subseteq \{0, 1, \ldots, \lfloor n/4M_{0} \rfloor - 1\} \Big) \ge 1 - \frac{1}{\kappa} e^{-\kappa n - M}.
\]
\item $(g_{1},g_{2}, \ldots )$ is distributed according to $\mu_{1} \times \mu_{2} \times \cdots$   on $\Omega$.
\item  On each equivalence class $\mathcal{F} \in \mathcal{Q}$, the random variables $\vartheta(1)(\w),\vartheta(2)(\w)$, $\ldots$ are constant and the $M_{0}$-long sequences \begin{equation}\label{eqn:alphaBeta}
\left\{\begin{aligned} (g_{4M_{0}\stopping(l) + 1}, \,\,\ldots, \,\,g_{4M_{0}   \stopping(l) +M_{0}} ),\\  (g_{4M_{0} \stopping(l) + M_{0}+1}, \,\ldots, \,g_{4M_{0} \stopping(l) +2M_{0}} ),\\
(g_{4M_{0} \stopping(l) + 2M_{0}+1}, \,\ldots, \,g_{4M_{0}   \stopping(l)  +3M_{0}} ),\\ (g_{4M_{0}\stopping(l) + 3M_{0}+1}, \,\ldots, \,g_{4M_{0}\stopping(l)  +4M_{0}} )\end{aligned} :  l=1, 2, \ldots \right\} 
\end{equation}
are all independently and uniformly distributed on $S$. Furthermore, all the increments $g_{i}$'s not appearing in Display \ref{eqn:alphaBeta} are fixed   on $\mathcal{F}$.
\end{enumerate}
\end{lem}

A version of this lemma was proven in \cite{choi2022random2}: see \cite[Lemma 4.6]{choi2022random2}. We give proof for the reader's convenience. 

\begin{proof}[Proof of Lemma \ref{lem:pivotFirstInfty}]
 Let \[
\mathcal{I} := \big\{ \lfloor t(l)/4M_{0} \rfloor : l=1, \ldots, M\big\}, \,\, \mathcal{J} := \big\{ 4M_{0} i +1, \ldots,   4M_{0}i + 4M_{0} : i \in \mathcal{I} \big\}.
\]
We now consider the partition $\mathcal{P}$ of the ambient probability space   based on  the values of $\{g_{i} : i \in \mathcal{J}\}$. On each $\mathcal{E} \in \mathcal{P}$,   the measures $\mu_{i} |_{\mathcal{E}}$ are Dirac atoms for $i \in \mathcal{J}$, and $\mu_{i}|_{\mathcal{E}}= \mu_{i}$ otherwise. Our goal is to realize the RVs \[
\{g_{1}, \ldots, g_{n}\} \in G, \,\, \{\vartheta(1),\ldots, \vartheta(\lfloor \epsilon^{4} n / 8M_{0} \rfloor - M )\} \subseteq \{0, \ldots, \lfloor n/4M_{0} \rfloor -1\} \setminus \mathcal{I},
\]and a partition $\mathcal{Q}_{\mathcal{E}}$ on $\mathcal{E}$ that satisfy the conditions (1), (2) and (3) with respect to $\Prob(\cdot | \mathcal{E}) \sim \mu_{1}|_{\mathcal{E}} \times  \cdots \times \mu_{n}|_{\mathcal{E}}$. If this is possible, then $\bigcup_{\mathcal{E} \in \mathcal{P}} \mathcal{Q}_{\mathcal{E}}$ becomes the desired partition.

Hence, from now on, we assume that $\mu_{i}$'s are Dirac atoms for $i \in \mathcal{J}$. Let us denote the uniform measure on $S$ by $\mu_{S}$.   Thanks to the assumption, for each   $k \in \Z_{\ge 0} \setminus \mathcal{I}$: we have \[
\mu_{4M_{0} k + 1} \times \ldots \times \mu_{4M_{0} k +   4M_{0}}  = \epsilon^{4} \cdot \mu_{S}^{4} + (1- \epsilon^{4}) \cdot \nu_{k}
\]
  for some probability measure $\nu_{k}$.

  For each $k \in \Z_{>0}$ we consider a Bernoulli RV   $\rho_{k}$ with average $\epsilon^{4}$,   an RV $\eta_{k}$ with the law $\mu_{S}^{4}$, and an RV $v_{k}$ with the law $\nu_{k}$.   All these RVs are set to be independent. Now let \[
(g_{4M_{0}k+1}, \,\ldots, \,g_{4M_{0}(k+1)}) = \left\{\begin{array}{cc} v_{k} & \textrm{when}\,\, \rho_{k} = 0, \\ \eta_{k} & \textrm{when} \,\, \rho_{k} = 1  \end{array}\right. \quad (k \in \Z_{\ge 0}  \setminus \mathcal{I}).
\]
For $k \in \mathcal{I}$, we decide $(g_{4M_{0}k+1}, \,\ldots, \,g_{4M_{0}(k+1)})$ using the Dirac measure $\mu_{4M_{0}k+1} \times \cdots \times \mu_{4M_{0}(k+1)}$. Then $(g_{1}, g_{2}, \ldots)$ is distributed according to $\mu_{1} \times \mu_{2} \times \cdots$. We denote by $\Omega$ the ambient probability space on which $(\rho_{i},v_{i}, \eta_{i})_{i}, (g_{i})_{i}$ are all measurable. 

Recall that $\rho_{i}$'s are i.i.d. Bernoulli RVs with average $\epsilon^{4}$. Chernoff-Hoeffding's inequality implies that $\Prob\big(\sum_{i} \rho_{i} \le \frac{1}{2}\epsilon^{4} (n/4M_{0} - M)\big)$ decays exponentially   in $n$:\[
\Prob\left( B_{n}:= \Bigg\{\sum_{i \in \{0, 1, \ldots,    \lfloor n/4M_{0} \rfloor-1 \} \setminus \mathcal{I} } \rho_{i}< \frac{\epsilon^{4} n}{8M_{0}}- M\Bigg\}\right) \le \frac{1}{\kappa}  e^{-(\kappa n-M)}
\]holds for some $\kappa = \kappa(\epsilon, M_{0})$. This gives the first item.

  For $\w \in \Omega$ we collect those indices $i$ at which $\rho_{i} = 1$ and denote them by  $\vartheta(1), \vartheta(2), \ldots$ in the increasing    order. That means, we set \[
\vartheta(k)(\w) := \min \{ l: \sum_{i = 0}^{l} \rho_{i}(\w) = k\}.
\]
Then $\vartheta( \lfloor \frac{\epsilon^{4} n}{8M_{0}}\rfloor- M) \le \lfloor n/4M_{0} \rfloor - 1$ for $\omega \in B_{n}$ by the definition of $B_{n}$.

We now define   a partition $\mathcal{Q}$ of $\Omega$ determined by the values of \begin{equation}\label{eqn:rhoVG}
\big\{\rho_{i}, v_{i} : i \notin \mathcal{I} \big\} \cup \{g_{i} : i \in \mathcal{J}\} 
\end{equation}
and define   \begin{equation}\label{eqn:wAlpha}\begin{aligned}
w_{i-1} &:= g_{4M_{0}[\stopping(i-1) + 1] + 1} \cdots g_{4M_{0} \stopping(i)}  & (i=1, 2, \ldots),\\
\alpha_{i} &:= (g_{4M_{0}\stopping(i) + 1}, \,\ldots, \,g_{4M_{0} \stopping(i) +M_{0}} ),& (i=1, 2, \ldots),\\
\beta_{i} &:= (g_{4M_{0} \stopping(i) + M_{0}+1}, \,\ldots, \,g_{4M_{0} \stopping(i) +2M_{0}} ),& (i=1, 2, \ldots),\\
\gamma_{i} &:= (g_{4M_{0} \stopping(i) + 2M_{0}+1}, \,\ldots, \,g_{4M_{0} \stopping(i) +3M_{0}} ),& (i=1, 2, \ldots),\\
\delta_{i} &:= (g_{4M_{0}\stopping(i) + 3M_{0}+1}, \,\ldots, \,g_{4M_{0}\stopping(i) +4M_{0}} ).& (i=1, 2, \ldots).
\end{aligned}
\end{equation}

On each equivalence class $\mathcal{F}$ of $\mathcal{Q}_{n}$, the increments involved in $w_{i}$'s are constants (since these are determined by the RVs in Display \ref{eqn:rhoVG}) and $(\alpha_{i}, \beta_{i},  \gamma_{i}, \delta_{i})$ are i.i.d.s with distribution $\mu_{S}^{4}$. This ends the proof.
\end{proof}

\begin{remark}\label{rem:pivotFirstInfty}
 Given $M$ first, by enlarging $\kappa(\epsilon, M_{0})$ if necessary, we even have
\[
\Prob\Big ( \# \big(\{\vartheta(1), \vartheta(2), \ldots \} \cap \{n, n+1, \ldots, n+k \}\big)\ge 0.5\epsilon^{4}(n-k) \Big) \ge 1 - \frac{1}{\kappa} e^{-\kappa k} 
\]
for any $n, k \ge 0$ for some $\kappa = \kappa(\epsilon, M_{0})$.
\end{remark}

The above lemma associates each $\w \in \Omega$ with an infinite word \[
w_{0} \Pi(\alpha_{1})\Pi(\beta_{1}) \Pi(\gamma_{1}) \Pi(\delta_{1}) \cdots w_{k} \Pi(\alpha_{k})\Pi(\beta_{k}) \Pi(\gamma_{k}) \Pi(\delta_{k}) \cdots.
\]
  On each $\mathcal{F} \in \mathcal{Q}$ the isometries $w_{i}$'s are fixed, while $\alpha_{i}, \beta_{i}, \gamma_{i}, \delta_{i}$'s are independently and uniformly distributed on $S$. It might be the case that $w_{0}$ is the inverse of the remaining subword and $g_{1} \cdots g_{n}$ becomes $id$. The following proposition, which is the heart of the pivoting technique, says that such a situation is unlikely.

\begin{prop} \label{prop:pivotEquivDescrip}
For each $N_{0} \ge 400$, there exists $\kappa= \kappa(N_{0})>0$ such that the following holds.

Fix a long enough Schottky set  $S$ with cardinality $N_{0}$ and a sequence  $(w_{i})_{i=0}^{\infty}$ in $G$. Suppose that $\{\alpha_{i}, \beta_{i}, \gamma_{i}, \delta_{i} : i>0\}$ are independently and uniformly distributed on $S$. Then there exists a measurable set   $P(\mathbf{s})$ for each $\mathbf{s} = \{\{\alpha_{i}, \beta_{i}, \gamma_{i}, \delta_{i} : i>0\} \in (S^{\Z_{>0}})^{4}$: \[
P\big(\{\alpha_{i}, \beta_{i}, \gamma_{i}, \delta_{i} : i>0\}\big) = \{j(1) < j(2) <\ldots \} \subseteq \Z_{>0},
\] and a measurable partition $\mathcal{P}$ of the probability space that satisfy the following: 
\begin{enumerate}
\item (alignment) For each $\mathbf{s} \in S^{\Z_{>0}}$,  if we define \[
U_{1} := w_{0}\Pi(\alpha_{1}), \,\, U_{i+1} := U_{i} \cdot  \Pi(\beta_{i}) \cdot \Pi(\gamma_{i}) \Pi(\delta_{i}) w_{i} \Pi(\alpha_{i+1}) \,\, (i=1, \ldots, n),
\]then the sequence  \[
\big(o,\,\, U_{j(1)} \Gamma(\beta_{j(1)}), \,\,U_{j(2)}\Gamma(\beta_{j(2)}), \,\,\ldots, \big)
\]
is $D_{0}$-semi-aligned. 
\item (high chance) For each $n$, we have  \[
\Prob\big( \# (P(\mathbf{s})\cap\{1, \ldots, n\} )  \le (1- 10/N_{0})  n \big) \le e^{-\kappa n}.
\]
\item On each $\mathcal{E} \in \mathcal{P}$, the set  $P$ is constant, as well as \[
\{\alpha_{i}, \gamma_{i}, \delta_{i}: i>0\}, \quad \{\beta_{i} : i \notin P\}.
\]
Meanwhile, on $\mathcal{E}$, $(\beta_{j(1)}, \beta_{j(2)}, \ldots)$ are independently and uniformly distributed on $S$.
\end{enumerate}
\end{prop}

This proposition was first proved in \cite[Section 4B]{gouezel2022exponential} and adapted in \cite[Section 5.1]{choi2022random1}. Since the statements there are slightly different from this one, let us sketch the proof.
\begin{proof}[Proof of Proposition \ref{prop:pivotEquivDescrip}]
In \cite[Section 5.1]{choi2022random1}, we proved:
\begin{prop}\label{prop:pivotEquivOri}
In the setting of Proposition \ref{prop:pivotEquivDescrip}, there exists $K = K(N_{0})>0$, and for each $n$ there exists a measurable set $P_{n} = \{j(1) < j(2) < \ldots < j(\#P_{n})\} \subseteq \{1, \ldots, n\}$ and an equivalence relation $\sim_{n}$ such that the following holds: \begin{enumerate}[label=(\alph*)]
\item We have \begin{equation}\label{eqn:pivotingCond}\mathbf{s} \sim_{n} \mathbf{s}' \,\, \Leftrightarrow \,\, \begin{array}{cc}
\alpha_{i} = \alpha_{i}',\,\, \delta_{i} = \delta_{i}' & \textrm{for $i \in \{1, \ldots, n\}$}, \\
\beta_{i} = \beta_{i}', \,\,\gamma_{i} = \gamma_{i}' & \textrm{for $i \in \{1, \ldots, n\} \setminus P_{n}(\mathbf{s})$}.
\end{array}
\end{equation}
Furthermore, $P_{n}(\mathbf{s}) = P_{n}(\mathbf{s}')$ if $\mathbf{s} \sim_{n} \mathbf{s}'$.
\item $\big(o, \,U_{j(1)} \Gamma(\beta_{j(1)}),\,\ldots, \,U_{j(\#P_{n})} \Gamma(\beta_{j(\#P_{n})}), U_{n} \Pi(\beta_{n}) \Pi(\gamma_{n}) \Pi(\delta_{n}) w_{n} o \big)$ is $D_{0}$-semi-aligned.
\item $\Prob\big(\#P_{n}(\mathbf{s}) \le (1-10/N_{0}) n  \big) \le e^{-n/K}$. More generally, we have \[
\Prob\big( \#P_{n+k}(\mathbf{s}) \le \#P_{n}(\mathbf{s}) + (1-10/N_{0})k\big) \le e^{-k/K} \quad (\forall n, k\ge0).
\]
\item for each $n \in \Z_{>0}$, either $P_{n+1} = P_{n} \cup \{n\}$ or $P_{n+1}$ is an initial section of $P_{n}$, i.e., $P_{n+1} = \{ \textrm{the first $\#P_{n+1}$ elements of $P_{n}$}\}$.
\end{enumerate}
\end{prop}

Item (a) is according to the definition of pivoting in the paragraph before \cite[Proposition 5.6]{choi2022random1}; in our setting $\mathbf{v}$ is kept to be  the identity sequence,   cf. the paragraph before \cite[Section 5.2]{choi2022random1}. The condition in Display \ref{eqn:pivotingCond} is indeed an equivalence relation thanks to \cite[Lemma 5.5]{choi2022random1}. Item (b) is due to \cite[Proposition 5.1]{choi2022random1}. Item (c) is proven in \cite[Corollory 5.8]{choi2022random1}. Item (d) is according to the definition of the set of pivotal times before \cite[Proposition 5.1]{choi2022random1}.

Let us define $P := \liminf_{n} P_{n} = \limsup_{n} P_{n}$. Here, the second equality holds for the following reason: if an integer $k$ is removed from some $P_{n}$, it cannot appear again in $P_{n+1}, P_{n+2}, \ldots$. Hence, if $k$ appears infinitely often in $P_{1}, P_{2}, \ldots$, then $k \in P_{k}, P_{k+1}, \ldots$. In fact, we observe: \[
P \subseteq P_{k} \cup \{k+1, k+2, \ldots\}\quad \textrm{for each $k=1, 2, \ldots$}.
\]
We now define a relation: \[
\mathbf{s} \sim \mathbf{s}'\,\, \Leftrightarrow \,\,\begin{array}{cc}
\alpha_{i} = \alpha_{i}',\,\, \gamma_{i} = \gamma_{i}', \,\, \delta_{i} = \delta_{i}' & \textrm{for $i =1, 2, \ldots$}, \\
\beta_{i} = \beta_{i}'  & \textrm{for $i \in \{1, 2, \ldots \} \setminus P(\mathbf{s})$}.
\end{array}
\]
Then $\sim$ is a finer relation   than each $\sim_{n}$. Hence, if $\mathbf{s} \sim \mathbf{s}'$ then $P_{n}(\mathbf{s}) = P_{n}(\mathbf{s}')$ for each $n$, which implies $P(\mathbf{s}) = P(\mathbf{s}')$. It follows that $\sim$ is an equivalence relation.

Furthermore, any initial section of $P$ is an initial section of some (in fact, all but finitely many) $P_{k}$. This combined with Item (a) implies Item (1), the alignment property. Furthermore, for each $M$ and $N$ we have \[
 \#P_{k}(\mathbf{s}) \ge N\,\, \textrm{for all $k \ge M$} \,\, \Rightarrow\,\, \textrm{$P(\mathbf{s})$ contains the $N$-initial section of $P_{M}(\mathbf{s})$}.
 \]
This observation and Item (c) imply that for every $n$, \[\begin{aligned}
&\Prob \big( P(\mathbf{s}) \cap P_{n}(\mathbf{s}) \subseteq \{1, \ldots, n\}\,\,\textrm{has at least $(1-10/N_{0}) n$ elements}\big) \\
&\ge 
\Prob \big(\# P_{k}(\mathbf{s}) \ge (1-10/N_{0}) n \,\,\textrm{for each $k \ge n$}\big) \\
&\ge 1- \sum_{k \ge n} e^{-k/K} \ge 1 - \frac{1}{1-e^{-1/K}} e^{-n/K}. \qedhere
\end{aligned}
\]
\end{proof}

Given RVs $g_{1}, g_{2}, \ldots$, we introduced the notation $Z_{i} := g_{1} \cdots  g_{i}$. For convenience we also write \begin{equation}\label{notat:Yi}
\mathbf{Y}_{i} := (Z_{i-M_{0}}, Z_{i-M_{0} + 1}, \ldots, Z_{i})
\end{equation}
when $M_{0}>0$ is understood. We now recall  the following definition:

\begin{definition}[{\cite[Definition 4.1]{choi2022random1}}]\label{dfn:pivotEquivClass}
Let $M_{0} \in \Z_{>0}$. 
Let $S$ be a long enough Schottky set in $G^{M_{0}}$, and let $\Omega$ be a probability space with $G$-valued RVs $\{g_{i} : i > 0\}$.

A subset $\mathcal{E}$ of $\Omega$, accompanied with the choice of a subset $\diffPivot(\mathcal{E}) = \{j(1) < j(2) < \ldots\}\subseteq M_{0} \Z_{>0}$, is called a \emph{pivotal equivalence class (for $S$)} if: \begin{enumerate}
\item for each $i \notin \{j(k) - l : k \ge 1, l =0, \ldots, M_{0} - 1 \}$, $g_{i}(\w)$ is fixed on $\mathcal{E}$;
\item for each $\w \in \mathcal{E}$ and $k \ge 1$, the following is a Schottky sequence: \[
s_{k}(\w) := \big(g_{j(k) - M_{0} + 1}(\w),\, g_{j(k) - M_{0} + 2}(\w), \, \ldots, \, g_{j(k)} (\w) \big) \in S;
\]
\item for each $\w \in \mathcal{E}$, $\big(o, \, \axes_{j(1)}(\w), \, \axes_{j(2)} (\w), \, \ldots)$ is $D_{0}$-semi-aligned, and
\item on $\mathcal{E}$, $\{s_{1}(\w), s_{2}(\w), \ldots \}$ are i.i.d.s distributed according to the uniform measure on $S$.
\end{enumerate}
We say that $\mathcal{P}(\mathcal{E})$ is the \emph{set of pivotal times} for $\mathcal{E}$.

When a pivotal equivalence class $\mathcal{E} \subseteq \Omega$ is  understood with the set of pivotal times $\mathcal{P}(\mathcal{E})$, for each element $\w$ of $\mathcal{E}$ we call $\diffPivot(\mathcal{E})$ the \emph{set of pivotal times for $\w$} and write it as $\diffPivot(\w)$.
\end{definition}

Lemma \ref{lem:pivotFirstInfty} and Proposition \ref{prop:pivotEquivDescrip} together imply the following:

\begin{prop}\label{prop:gouezelRW1}
For each $0 < \epsilon \le 1$ and $M_{0} \in \Z_{>0}$, there exists $\kappa = \kappa(\epsilon, M_{0})>0$ such that the following holds.

Let  $n, M\in \Z_{>0}$. Let $S \subseteq G^{M_{0}}$ be large and long-enough $K_{0}$-Schottky set. Let $t(1), \ldots, t(M)$ be $M$ distinct integers. Let $\mu_{1}, \mu_{2}, \ldots$ be probability measures such that \begin{enumerate}
\item  $\mu_{t(l)}$ is the Dirac measure concentrated on an element $g_{t(l)} \in G$ for $l=1, \ldots, M$, and  
\item whenever $s \in S$ and $\{k + 1, \ldots, k + M_{0} \} \subseteq \{1, \ldots, n \} \setminus \{t(1), \ldots, t(M)\}$, we have \[
\big(\mu_{ k + 1} \times \cdots \times \mu_{ k + M_{0}} \big) ( s) > \epsilon.
\]
\end{enumerate}

Then there exist a probability space $\Omega$ equipped with RVs \[
(g_{1}, g_{2}, \ldots)\,\, \textrm{distributed according to}\,\, \mu_{1} \times \mu_{2} \times \cdots,
\]and a measurable partition $\mathcal{Q}$ of $\Omega$ into pivotal equivalence classes for $S$ such that \begin{equation}\label{eqn:gouezelRWDecay}
\Prob\big(\w : \# (\diffPivot(\w) \cap \{1, \ldots, n\}) \le \kappa n - M\big) \le \frac{1}{\kappa} e^{- (\kappa n - M)}
\end{equation}
for each $n$.
\end{prop}

When we are concerned with the pivotal time construction for step $n$, we do not need to construct $P$  from $P_{k}$'s. By combining Proposition \ref{prop:pivotEquivOri} with Lemma \ref{lem:pivotFirstInfty} (with Remark \ref{rem:pivotFirstInfty} in mind), we obtain the following:

\begin{prop}\label{prop:gouezelRW1Ori}
For each $0 < \epsilon \le 1$ and $M_{0} \in \Z_{>0}$, there exists $\kappa = \kappa(\epsilon, M_{0})>0$ such that the following holds.

Let  $n\in \Z_{>0}$. Let $S \subseteq G^{M_{0}}$ be large and long-enough $K_{0}$-Schottky set.  Let $\mu_{1}, \mu_{2}, \ldots$ be probability measures such that for each $s \in S$ and $k \in \Z$, we have \[
\big(\mu_{M_{0} k + 1} \times \cdots \times \mu_{M_{0} k + M_{0}} \big) ( s) > \epsilon.
\]
Let us then change just one measure $\mu_{i}$ into a Dirac measure.

Then there exist a probability space $\Omega_{n}$ equipped with RVs \[
(g_{1}, g_{2}, \ldots, g_{n} )\,\, \textrm{distributed according to}\,\, \mu_{1} \times \mu_{2} \times \cdots \mu_{n} 
\]and a measurable partition $\mathcal{Q}_{n}$ of $\Omega_{n}$ into pivotal equivalence classes for $S$ such that the following conditions are satisfied: 
\begin{enumerate}
\item for each pivotal equivalence class $\mathcal{E}$, the set of pivotal times $\mathcal{P}_{n}(\mathcal{E}) = \{j(1) < \ldots < j(\#\mathcal{P}_{n}(\mathcal{E}))\}$ is a subset of $\{1, \ldots, n\}$. Moreover, for each $\w \in \mathcal{E}$,  \[
(o, \mathbf{Y}_{j(1)}(\w), \ldots, \mathbf{Y}_{j(\#\mathcal{P}_{n}(\mathcal{E}))}, Z_{n} o)
\]
is $D_{0}$-semi-aligned;
\item for each $1 \le k < l \le n$, we have \[
\Prob \big( \w : \#(\mathcal{P}_{n}(\w) \cap \{k, \ldots, l\}) \le \kappa(k-l) \big) \le \frac{1}{\kappa}e^{-\kappa(k-l)}.
\]
\end{enumerate}
\end{prop}

Let us now recall two facts:

\begin{lem}[{\cite[Corollary 4.5]{choi2022random1}}]\label{cor:pivotingRW1}
Let $K_{0}, N_{0}>0$ and let  $S$ be a long enough $K_{0}$-Schottky set with cardinality $N_{0}$. Let $\mathcal{E}$ be a pivotal equivalence class for $S$ with $\diffPivot(\mathcal{E}) = \{j(1) < j(2) < \ldots\}$ and let $x \in X$. Then for each $k \ge 1$ we have \[
\Prob\left( \big(x, \,\axes_{j(k)}(\w),\, \axes_{j(k+1)}(\w), \ldots \big) \,\, \textrm{is $D_{0}$-semi-aligned} \, \Big| \, \mathcal{E} \right) \ge 1- (1/N_{0})^{k}.
\]
Moreover, for any $m\ge 1$, $n \ge j(m)$ and $k = 1, \ldots, m$, we have \[
\Prob\left( \big(\axes_{j(1)}(\w), \,\ldots, \,\axes_{j(m - k+1)}(\w),\, Z_{n}(\w) o \big) \,\, \textrm{is $D_{0}$-semi-aligned} \, \Big| \, \mathcal{E} \right) \ge 1- (1/N_{0})^{k}.
\]
\end{lem}

\begin{lem}[{\cite[Corollary 4.6]{choi2022random1}}]\label{cor:pivotingRW2}
Let $K_{0}, N_{0}>0$ and let $S$ and $\check{S}$ be long enough $K_{0}$-Schottky sets with cardinality $N_{0}$. Let $\mathcal{E}$ be a pivotal equivalence class for $S$ with $\diffPivot(\mathcal{E}) = \{j(1) < j(2) < \ldots\}$, and let $\check{\mathcal{E}}$ be a pivotal equivalence class for $\check{S}$ with $\diffPivot(\check{\mathcal{E}}) = \{\check{j}(1) < \check{j}(2) < \ldots\}$.  Then we have  \[
\Prob\left( \big(\bar{\axes}_{j(k)}(\w),\, \axes_{\check{j}(k)}(\check{\w})\big) \,\, \textrm{is $D_{0}$-semi-aligned} \, \Big| \, \mathcal{E} \right) \ge 1- (2/N_{0})^{k}. \quad (\forall k > 0)
\]
\end{lem}

We are now ready to prove Proposition \ref{prop:gouezelRWEventual}. Let us recall the statement:

\begin{prop}\label{prop:gouezelRW1well}
Let $0< \epsilon \le 1$ and let $S \subseteq G^{M_{0}}$ be a long and large enough $K_{0}$-Schottky set. Then there exists $\kappa = \kappa(\epsilon, S)>0$ such that the following holds.

Let $\{ \mu_{i} : i \in \Z_{>0}\}$ be probability measures such that \[
\big(\mu_{M_{0} k + 1} \times \cdots \times \mu_{M_{0} k + M_{0}} \big) ( s) > \epsilon. \quad (\forall s \in S, \forall k \in \Z)
\]
Let $g_{i}$ be distributed according to $\mu_{i}$ for each $i \in \Z$, all independently, and let $Z_{i}, \mathbf{Y}_{i}$ be as in Notation \ref{notat:location} and Display \ref{notat:Yi}.

Then for arbitrary $M$ integers $t(1), \ldots, t(M)$ and for arbitrary $g_{t(1)}, \ldots, g_{t(M)} \in G$, we have  \[
\Prob \left(\left.\begin{array}{c} \textrm{$\exists \,   j(1), j(2), \ldots, j(\lfloor \kappa n\rfloor) \in \{M_{0}, \ldots, n\}$ such that}\\
\textrm{$\mathbf{Y}_{j(1)}, \ldots, \mathbf{Y}_{j( \lfloor\kappa n\rfloor)}$ are Schottky axes and} \\
\textrm{$(Z_{-m} o, \mathbf{Y}_{j(1)}, \ldots, \mathbf{Y}_{j( \lfloor\kappa n \rfloor)}, Z_{n+l} o)$ is $D_{0}$-semi-aligned for all $m, l\ge0$} \end{array}\, \right| \, g_{t(1)}, g_{t(2)}, \ldots, g_{t(M)}\right) \ge 1-\frac{1}{\kappa} e^{-\kappa n - M}.
\]
\end{prop}

\begin{proof}
A very similar proof was given in \cite[Lemma 4.10]{choi2022random1} so we only sketch the idea.
Given $M$    integers $t(1), \ldots, t(M)$, we replace $\mu_{t(1)}$, $\ldots$, $\mu_{t(M)}$ with Dirac measures concentrated on the given elements, respectively. This will not affect the qualification for Proposition \ref{prop:gouezelRW1}.

We employ the notation \[\begin{aligned}
\check{g}_{i} &:= g_{-i+1}^{-1}   , \quad  \check{\mu}_{i}(\cdot) &:= \mu_{-i+1}(\cdot^{-1})   , \quad\check{S} &:= \{ (s_{M_{0}}^{-1}, \ldots, s_{1}^{-1}) : (s_{1}, \ldots, s_{M_{0}}) \in S\}.
\end{aligned}
\]
We then construct the probability space $\Omega$ for $\mu_{1} \times \mu_{2} \times \cdots$ and $\check{\Omega}$ for $\check{\mu}_{1} \times \check{\mu}_{2} \times \cdots$, equipped with partitions $\mathcal{Q}$ and $\check{\mathcal{Q}}$ into pivotal equivalence classes as in Proposition \ref{prop:gouezelRW1}. Let    $\kappa_{1} = \kappa(\epsilon, M_{0})$ be coefficient of the exponential decay rate coming from Proposition \ref{prop:gouezelRW1}.

Then    $(\Omega, \check{\Omega})$  comes with independent RVs $g_{i}$'s and $\check{g}_{i}$'s, each distributed according to $\mu_{i}$'s and $\check{\mu}_{i}$'s, respectively. We define \[
Z_{i} := g_{1} \cdots g_{i}, \quad  \check{Z}_{i} := \check{g}_{1} \cdots \check{g}_{i}, \quad \mathbf{Y}_{i}(\w) := (Z_{i-M_{0}}, \ldots, Z_{i}), \quad \mathbf{Y}_{i}(\check{\w}) := (\check{Z}_{i-M_{0}}, \ldots, \check{Z}_{i}).
\]
It is then immediate that $(\ldots, \check{Z}_{2}, \check{Z}_{1}, id, Z_{1}, Z_{2}, \ldots)$ has the same law as $(\ldots, Z_{-2}, Z_{-1}, Z_{0}, Z_{1}, Z_{2}, \ldots)$. 

Let \[
\begin{aligned}
\mathcal{A}_{n}& := \big\{ \w \in \Omega :\# ( \mathcal{P}(\w) \cap \{1, \ldots, k\} ) \ge    \kappa_{1} k - M \textrm{for all $k \ge n$}\big\}, \\
\check{\mathcal{A}}_{n} &:= \big\{ \check{\w} \in \check{\Omega} :\# ( \check{\mathcal{P}}(\check{\w}) \cap \{1, \ldots, k\} ) \ge    \kappa_{1} k - M \textrm{for all $k \ge n$}\big\}.
\end{aligned}
\]
These sets are unions of some pivotal equivalence classes and have probability at least    $1 - K_{1} e^{-(\kappa_{1} n - M)}$ for some $K_{1}$ depending only on    $\kappa_{1}$.

Next, we pick pivotal equivalence classes $\mathcal{E} \in \mathcal{Q}$ and $\check{\mathcal{E}} \in \check{\mathcal{Q}}$ from $\mathcal{A}_{n}$ and $\check{\mathcal{A}}_{n}$, respectively.    This choice also determines the sets of pivotal times $\mathcal{P}(\mathcal{E}) = \{j(1) < j(2) < \ldots\}$ and $\check{\mathcal{P}}(\check{\mathcal{E}}) = \{\check{j}(1) < \check{j}(2) < \ldots\}$. On $\mathcal{E} \times \check{\mathcal{E}}$, we estimate the conditional probability for the following event:

\begin{enumerate}
\item for $x \in \{o, \check{Z}_{1} o, \ldots, \check{Z}_{n} o\}$, the following sequence is $D_{0}$-semi-aligned: \[
\big(x, \,\axes_{j(\lceil \kappa_{1}n/3\rceil )},  \,\axes_{j(\lceil \kappa_{1}n/3\rceil + 1)}, \,\ldots\big);
\]
\item for each $k \ge n$, the following are $D_{0}$-semi-aligned: \[\begin{aligned}
&\big(o, \,\axes_{j(1)}(\w), \, \axes_{j(2)}(\w),  \, \ldots, \, \axes_{j(\lceil 2\kappa_{1}k/3\rceil)}(\w), \,Z_{k} o\big),\\
&\big(o, \axes_{\check{j}(1)}(\check{\w}) \, \axes_{\check{j}(2)}(\check{\w})  \, \ldots, \, \axes_{\check{j}(\lceil 2\kappa_{1}k/3\rceil)}(\check{\w}), \,\check{Z}_{k} o\big);
\end{aligned}
\]
\item $\big(\bar{\axes}_{\check{j}(i)}(\check{\w})   ,\, \axes_{j(i)}(\w) \big)$ is $D_{0}$-aligned for some $i \le    \kappa_{1}n/3$.
\end{enumerate}
Note    that $j(\lceil \kappa_{1} k \rceil - M) \le k$ and $\check{j}(\lceil \kappa_{1} k \rceil - M) \le k$ for each $k \ge n$. Lemma \ref{cor:pivotingRW1} and \ref{cor:pivotingRW2}     tell us the above    events happen at the same time with probability at least \[
1 -n (1/N_{0})^{\kappa_{1}n/3 - M} - 2\sum_{k\ge n} (1/N_{0})^{\kappa_{1}k/3- M} -  (2/N_{0})^{\kappa_{1}n/3}  \ge 1 - K(1/100)^{\kappa_{1}n/3 - M}
\]
for some suitable $K>0$ not depending on $n$. Here, we used the fact that $N_{0} \ge 400$.

Now    let $(\w, \check{\w})$ be an element of the  above event. For $m \in \{0, \ldots, n\}$ and $l \ge 0$, we observe that \[
(\check{Z}_{m}o, \mathbf{Y}_{j(\lceil \kappa_{1}n/3 \rceil)}, \ldots, \mathbf{Y}_{j(\lceil 2\kappa_{1}(n+l)/3 \rceil)}(\w), Z_{n+l} o)
\]
is $D_{0}$-semi-aligned. For $m \ge n$ and $l \ge 0$, we observe that \[
\big( \check{Z}_{m} o, \bar{\axes}_{\check{j}(\lceil 2\kappa_{1}m/3\rceil)}(\check{\w}), \ldots, \bar{\axes}_{\check{j}(i+1)}(\check{\w}),  \bar{\axes}_{\check{j}(i)}(\check{\w}), \axes_{j(i)}(\w), \ldots, \axes_{j(\lceil    \kappa_{1} n/3\rceil)}(\w), \ldots, \mathbf{Y}_{j(\lceil 2\kappa_{1}(n+l)/3 \rceil)}(\w), Z_{n+l} o\big)
\]
is $D_{0}$-semi-aligned. In either case, $(\check{Z}_{m} o, \mathbf{Y}_{j(\lceil   \kappa_{1} n/3\rceil)}, \ldots, \mathbf{Y}_{j(\lceil    2\kappa_{1} n/3 \rceil)}(\w), Z_{n+l} o)$ is a subsequence of a $D_{0}$-semi-aligned sequence, so it is itself a $D_{0}$-semi-aligned. Hence, we can choose \[
\{i(1), \ldots, i(\lfloor n/3K\rfloor)\} := \{j(\lceil    \kappa_{1} n/3 \rceil), \ldots, j(\lceil    2 \kappa_{1} n/3 \rceil) \} \subseteq M_{0}\Z \cap \{   1\le i  \le n\}.
\]
   Therefore the desired event holds for $\Prob(\cdot | \mathcal{E} \times \check{\mathcal{E}})$-probability at least $1-K (1/100)^{\kappa_{1} n/3 - M}$. The proof ends by averaging these conditional probabilities across $\mathcal{A}_{n} \times \check{\mathcal{A}}_{n}$ and multiplying by the bound $\Prob(\mathcal{A}_{n} \times \check{\mathcal{A}}_{n}) \ge    1 - K \cdot 100^{-(\kappa_{1} n - M)}.$
\end{proof}

\section{Differentiability of the escape rate}\label{section:diff}

In this section, we establish Lipschitz continuity and differentiability of the escape rate.

\subsection{Defect} \label{subsection:deficit}

We first define the \emph{defect} arising from    adding displacements along a random path. Let:
\[\begin{aligned}
R_{n, m}(\mathbf{g}) &:=   \|Z_{-(m+1)}^{-1} Z_{-1}\|+ \|Z_{-1}\| + \|Z_{n}\|- \|Z_{-(m+1)}^{-1} Z_{n}\| \\
&=
\| g_{-m} \cdots g_{-1}\| + \|g_{0}\| + \|g_{1} \cdots g_{n}\| - \|g_{-m} \cdots g_{n}\|.
\end{aligned}
\]

\begin{lem}\label{lem:conditionedDist}

   Let $(X, G)$ be as in Convention \ref{conv:main}. Let $\mu$ be a non-elementary probability measure on $G$. Then there    exist $K, \epsilon > 0$ such that\[
\Prob_{\mu_{1} \times\cdots \times \mu_{k}} \left(\left.(g_{1}, \ldots, g_{k}) : \begin{array}{c} \textrm{there exists $0 \le  l \le k$ such that}\\  \textrm{$d\left(Z_{l} o,    [g^{-1}o, Z_{k}h o]\right) < K$}\end{array} \right| g,  g_{j},    h\right) \ge 1 - K e^{-k/K}
\] for    every $1 \le j< k$, for   every choice of $g, h$ and $g_{j}$, and for    every sequence of probability measures $\{\mu_{i}\}_{i}$ such that $\| \mu_{i}- \mu\|_{0} < \epsilon$ for each $i$.
\end{lem}

In particular, we have \[
\Prob\left(\left. \begin{array}{c} \textrm{for each $m, l \ge 0$ there exists $i \in \{0, \ldots, M_{0}\}$ such that}\\
\textrm{ $d(Z_{i} o, [Z_{-m} o, Z_{n+l} o]) \le E_{0}$ }\end{array}\, \right| \, g_{t(1)}, g_{t(2)}, \ldots, g_{t(M)}\right) \ge 1-\frac{1}{\kappa} e^{-\kappa n - M}.
\]

\begin{proof}
Let $\mu$ be a non-elementary probability measure. Let $N_{0} =400$, let $K_{0} = K(N_{0})$ be as in Proposition \ref{prop:Schottky}, let $D_{0}, E_{0}, L_{0}$ be the ones determined as in Definition \ref{dfn:Schottky}, and let $M_{0} = n(N_{0}, 2L_{0})$ be as in Proposition \ref{prop:Schottky}. Then Proposition \ref{prop:Schottky} guarantees a $K_{0}$-Schottky set in $G^{M_{0}}$ such that $\#S = 400$ and $\mu^{M_{0}}(s) >0$ for each $s \in S$. This $S$ is clearly large and long enough.

Let $\epsilon= \frac{1}{2}\min_{s \in S} \mu^{M_{0}}(s)$. Now pick any Schottky sequence $s = (s_{1}, \ldots, s_{M_{0}}) \in S$. We claim that $\mu(s_{i}) > 2 \epsilon$ for each $i$.    To see this, note that $\mu(g) \le 1$ for each $g \in G$. This implies \[
\mu(s_{i}) \ge  \mu(s_{i}) \cdot \mu(s_{1}) \cdots \mu(s_{i-1}) \cdot \mu(s_{i+1}) \cdots \mu(s_{M_{0}}) = \mu^{M_{0}}(s) \ge 2\epsilon.
\]
We now fix $\kappa = \kappa(2^{-M_{0}} \epsilon, S)>0$ using Proposition \ref{prop:gouezelRWEventual}.

Now consider a sequence of arbitrary measures $\mu_{i}$'s such that $\|\mu_{i} - \mu\|_{0} < \epsilon$ for each $i$. In particular, $\mu_{i}(g) \ge \mu(g) - \epsilon$    for each $i$ and for each $g \in G$. Then for any $s = (s_{1}, \ldots, s_{i}) \in S$ and for any integers $i(1), \ldots, i(M_{0})$ we have \[
\prod_{l} \mu_{i(l)}(s_{i}) \ge \prod_{i} 0.5 \mu(s_{i}) \ge 2^{-M_{0}} \cdot \mu^{M_{0}}(s).
\]
In particular, $(\mu_{k + 1} \times \cdots \mu_{k+M_{0}})(s) > 2^{-M_{0}} \epsilon$ for each $s \in S$ and for each $k$. 

Now Proposition \ref{prop:gouezelRWEventual} guarantees $\kappa = \kappa(\epsilon, S)$ such that 
    \[
\Prob \left(\left.\begin{array}{c} \textrm{$\exists$  $S$-Schottky axes $\gamma \subseteq \{o, Z_{1} o, Z_{2} o, \ldots, Z_{k} o\}$ such that } \\
\textrm{$(Z_{-1} o, \gamma, Z_{k+1} o)$ is $D_{0}$-semi-aligned} \end{array}\, \right| \, g_{-1} = g, g_{j}, g_{k+1} = h\right) \ge 1-\frac{1}{\kappa} e^{-(\kappa k - 3)}.
\]
Here Proposition \ref{prop:BGIPWitness} implies that, if $(Z_{-1} o, \gamma, Z_{k+1} o)$ is $D_{0}$-semi-aligned, then $\gamma$ lies in the $E_{0}$-neighborhood of $[Z_{-1} o, Z_{k+1} o]$. Combining these, we conclude that \[
\Prob_{(g_{i})_{i \in \Z} \sim \prod_{i \in \Z} \mu_{i}} \left( \exists l\in \{0, \ldots, k\} \,\,\textrm{such that} \,\, d(Z_{l} o, [Z_{-1} o, Z_{k+1} o]) \le E_{0} \, \Big| \, g_{-1}= g, g_{j}, g_{k+1} = h\right) \ge 1 - \frac{1}{\kappa} e^{-(\kappa k - 3)}.
\]
\end{proof}

\begin{cor}\label{cor:bddDefect}
   Let $(X, G)$ be as in Convention \ref{conv:main}.
Let $\mu$ be a non-elementary probability measure on $G$ with finite first moment. Then there   exist $K_{1}, \epsilon > 0$ such that\[
\E_{\prod_{i=1}^{\infty} \mu_{i}\, \otimes \,1_{g_{0}=g} \,\otimes \,\prod_{i =1}^{\infty} \mu_{-i}}[R_{n, m}(\mathbf{g}) ] \le K_{1}
\] for every $n, m> 0$, for every $g \in G$, and for every probability measures $\{\mu_{i}\}_{i \neq 0}$ such that $\|\mu_{i} - \mu\|_{0, 1} < \epsilon$ for each $i$. 
\end{cor}

\begin{proof}
We define auxiliary RVs \[\begin{aligned}
R'(\mathbf{g}) &:= \|g_{-m} \cdots g_{-1} g_{0}\| + \|g_{1} \cdots g_{n}\| - \|g_{-m} \cdots g_{-1} \cdot g_{0} \cdot g_{1} \cdots g_{n}\|,\\
R''(\mathbf{g}) &:= \|g_{-m} \cdots g_{-1}\| + \|g_{0}\| -  \|g_{-m} \cdots g_{-1} g_{0}\|.
\end{aligned}
\]
Clearly, we have $R_{n, m}(\mathbf{g}) = R'(\mathbf{g}) + R''(\mathbf{g})$.

Let $K, \epsilon > 0$ be as in Lemma \ref{lem:conditionedDist} and set\[\begin{aligned}
E_{j} &:= \big\{ \mathbf{g} :  \textrm{there exists $i \in \{1, \ldots, j\}$ such that  $d\big(Z_{i} o, \,[Z_{-(m+1)} o, Z_{n} o]\big) < K$}\big\} \\
&= \left\{ \mathbf{g} : \begin{array}{c} \textrm{there exists $i \in \{0, \ldots, j\}$ such that\,} \\
d\big(Z_{i} o, [g^{-1} o, Z_{j} ho]\big) < K\end{array}\right\} \quad (g= g_{-m} \cdots g_{0}, \, h = g_{j+1} \cdots g_{n}).
\end{aligned}
\]
By the previous lemma, we have \[
\Prob_{\prod_{i} \mu_{i}}\Big(E_{j}^{c}\, \Big| \, g_{-m}, \ldots, g_{0}, g_{j+1}, g_{j+2}, \ldots, g_{n}\Big) < K e^{-j/K}
\] for each $j\ge 0$ and each prescribed choices of $g_{-m}, \ldots, g_{0}, g_{j+1}, g_{j+2}, \ldots, g_{n}$. Moreover, if $Z_{l} o$ is $K$-close to $[Z_{-(m+1)} o, Z_{n}o]$ for some $l \in \{0, \ldots, j\}$, then
 \[\begin{aligned}
R'(\mathbf{g}) 
&= \| g_{-m} \cdots g_{-1} g_{0}\| + \|g_{1} \cdots g_{n}\| - \|g_{-m} \cdots g_{-1} \cdot g_{0} \cdot g_{1} \cdots g_{n}\|\\
& \le \|g_{-m} \cdots g_{-1}  g_{0}\| + \| g_{1} \cdots g_{n}\| - \left( \|g_{-m} \cdots g_{-1} g_{0} g_{1}\cdots g_{l}\| + \|g_{l} \cdots g_{n}\| - K \right) \\
&\le \|g_{1} \cdots g_{l}\|^{sym} + 2K \le \sum_{i=1}^{j} \|g_{i}\|^{sym} + 2K.
\end{aligned}
\]

Hence, we have \[\begin{aligned}
\E[ R'] &\le \E\left[ \sum_{i=1}^{\min\{ j : \mathbf{g} \in E_{j}\}} \|g_{i}\|^{sym} + 2K\right] \le \sum_{i=1}^{ n} \E\left[\|g_{i}\|^{sym} 1_{E_{i-1}^{c}}\right] + 2K \\
&\le 2K+ \sum_{i=1}^{m} \E\left[\|g_{i}\|^{sym} \Prob(E_{i-1}^{c} | g_{i})\right] \\
&\le 2K+\sum_{i=0}^{m-1} K e^{- i/K} \|\mu_{i}\|_{1}  \le 2K  + \frac{K}{1- e^{-1/K}}(\|\mu\|_{1} + \epsilon).
\end{aligned}
\]

Similarly, if we define \[
F_{j} := \left\{ \mathbf{g} : \exists i \in \{1, \ldots, j\} \textrm{\, s.t.\, }d\left(Z_{-(i+1)} o, [Z_{-m} o,  o] \right) < K\right\},
\]
then  we have $\Prob(F_{j-1}^{c} | g_{-m}, \ldots, g_{-j}, g_{0}) < K e^{-(j-1)/K}$ for each $j$ and $g_{-m}, \ldots, g_{-j}, g_{0}$. Moreover, we similarly have \[
R'' \le K + \sum_{i=1}^{\min \{j : \mathbf{g} \in F_{j}\}} \|g_{-i}\|^{sym}.
\] Based on the same computation, we conclude   that  \[
\E[R_{n, m}] = \E[R'] + \E[R''] \le 4K + \frac{2K}{1-e^{-1/K}} (\|\mu\|_{1} +\epsilon).\qedhere
\]
\end{proof}

We  are now ready to prove Theorem \ref{thm:driftLipschitz}.

\begin{proof}
Let $\mu$ be a non-elementary probability measure   on $G$ with $\|\mu\|_{1} < +\infty$ and let $K_{1}, \epsilon > 0$ be as in Corollary \ref{cor:bddDefect}. Now consider a measure $\mu'$ with $\|\mu' - \mu\|_{0,1} < \epsilon$. Then we have
\begin{equation}\label{eqn:driftLipschitz}\begin{aligned}
L(\mu'^{\ast n}) - L(\mu^{\ast n}) &= \sum_{i=1}^{n} \left( \E_{\mu^{\ast (i-1)} \times \mu' \times \mu'^{\ast(n-i)}} \|g_{1}g_{2}g_{3}\| - \E_{\mu^{\ast (i-1)} \times \mu \times \mu'^{\ast(n-i)}} \|g_{1}g_{2}g_{3}\| \right) \\
&=\sum_{i=1}^{n}\left(  \E_{\mu'}\|g\| +\E_{\mu^{\ast(i-1)}} \|g\| + \E_{\mu'^{\ast (n-i)}}  \|g\| - \E_{\mu^{i-1} \otimes \mu' \otimes \mu'^{n-i}} [R_{n-i, i-1} (\mathbf{g})]\right)\\
 &- \sum_{i=1}^{n} \left( \E_{\mu}\|g\| +\E_{\mu^{\ast(i-1)}} \|g\| + \E_{\mu'^{\ast (n-i)}}  \|g\| - \E_{\mu^{i-1} \otimes\mu \otimes \mu'^{n-i}} [R_{n-i, i-1} (\mathbf{g})] \right) \\
 &= n \E_{\mu'- \mu} \|g\| + \sum_{i=1}^{n} \sum_{g \in G} \left( \mu'(g) - \mu(g) \right) \E_{\mu^{i-1} \otimes 1_{g_{0}=g} \otimes \mu'^{n-i}} [R_{n-i, i-1} (\mathbf{g})] .
\end{aligned}
\end{equation}
Here, $\E_{\mu'- \mu} \|g\|$ is bounded by $\|\mu'-\mu\|_{1}$. Moreover, Corollary \ref{cor:bddDefect} tells us that $\E_{\mu^{k}\, \otimes\, 1_{g_{0} = g} \,\otimes \,\mu'^{l}} [R_{l, k} (\mathbf{g})] < K_{1}$ for any $k, l\ge 0$ and $g \in G$. Hence, the second term is bounded by $n K_{1} \| \mu'-\mu\|_{0}$. As a result, we have \[
\left| \frac{L(\mu'^{\ast n}) - L(\mu^{\ast n})}{n}  \right| \le \| \mu'- \mu\|_{1} + K_{1} \|\mu' - \mu\|_{0}.
\]
By sending $n$ to infinity, we obtain the same bound for $|l(\mu') - l(\mu)|$.
\end{proof}

We   state some lemmata before proceeding to Theorem \ref{thm:driftDiff} , whose proofs are omitted.

\begin{lem}\label{lem:driftDiffMean0}
Let $\mu$, $\eta$ be signed measures and $\{\mu_{t}\}_{t \in [-1, 1]}$ be a family of signed measures such that \[\begin{aligned}
\| \mu\|_{0} = \|\mu_{t}\|_{0} <  +\infty\,\, \textrm{for all $t$}, \quad \|\eta\|_{0} < +\infty,
\end{aligned}
\]
and such that $\|\mu_{t} - \mu - t\eta\|_{0} =o(t)$. Then $\eta$ is balanced, i.e., $\sum_{g \in G} \eta(g) = 0$.
\end{lem}

\begin{lem}\label{lem:driftConv}
Let $\mu$, $\mu'$, $\eta$, $\eta'$ be signed measures with finite $\|\cdot\|_{0, 1}$-norm, and let $\{\mu_{t}\}_{t\in [-1, 1]}$, $\{\mu_{t}' \}_{t \in [-1, 1]}$ be families of signed measures such that \[
 \|\mu_{t} - \mu - t\eta\|_{0, 1}, \|\mu_{t}' - \mu' - t\eta'\|_{0, 1} = o(t).
\]
Then we have \[
\big\|\mu_{t}\ast \mu_{t}' - \mu \ast \mu' - t(\eta\ast \mu' + \mu \ast \eta') \big\|_{0, 1} = o(t)
\]
\end{lem}

\begin{proof}[Proof of Theorem \ref{thm:driftDiff}]
 Let $K_{1}, \epsilon>0$ be the constants for $\mu$ as described in Corollary \ref{cor:bddDefect}.  For each $t \in [-1, 1]$ let us define \[
K(\mu_{t}, \mu) := \sup_{g \in G}\sup_{n, m>0} \E_{\mu_{t}^{n}\, \otimes \, 1_{g_{0} = g} \, \otimes \,\mu^{m}} [R_{n, m}(\mathbf{g}) ].
\]
In the previous proof, we   proved $K(\mu_{t}, \mu) < K_{1}$ and  \[
\left| \frac{1}{n} [L(\mu_{t}^{\ast n}) - L(\mu^{\ast n})] - \E_{\mu_{t} - \mu} \|g\| \right| \le K(\mu_{t}, \mu) \| \mu_{t} - \mu\|_{0} < K_{1}  \| \mu_{t} - \mu\|_{0}
\]
for small enough   $t$ such that $\|\mu_{t} - \mu\|_{0, 1} < \epsilon$.

Let us now replace $\mu_{t}$ with $\mu_{t}^{\ast k}$ and $\mu$ with $\mu^{\ast k}$.  Since $\mu^{\ast k}$ is non-elementary and since $\| \mu_{t}^{\ast k} - \mu^{\ast k}\|_{0, 1} \rightarrow 0$ as $t \rightarrow 0$, we similarly have \[
\left| \frac{1}{n} [L(\mu_{t}^{\ast nk}) - L(\mu^{\ast nk})] -  \E_{\mu_{t}^{\ast k} - \mu^{\ast k}} \|g\| \right| \le K(\mu_{t}^{\ast k}, \mu^{\ast k}) \| \mu_{t}^{\ast k} - \mu^{\ast k}\|_{0}.
\]
  Note that $K(\mu_{t}^{\ast k}, \mu^{\ast k}) \le K(\mu_{t}, \mu) \le K_{1}$ for small enough $t$. Hence,   dividing both sides by $k$ and $t$ we obtain  \[
\frac{1}{t} \left| \frac{1}{nk} [L(\mu_{t}^{\ast nk}) - L(\mu^{\ast nk})] - \frac{1}{k} \E_{\mu_{t}^{\ast k} - \mu^{\ast k}} \|g\| \right| \le \frac{1}{t} \frac{K_{1}}{k} \| \mu_{t}^{\ast k} - \mu^{\ast k}\|_{0}.
\]
for arbitrary $k>0$. By   letting $n$ to infinity, we have \begin{equation}\label{eqn:theDic}
\left|\frac{1}{t} [l(\mu_{t}) - l(\mu)] - \frac{1}{t}\frac{1}{k} \E_{\mu_{t}^{\ast k} - \mu^{\ast k}} \|g\| \right| \le \frac{K_{1}}{t} \cdot \frac{ \| \mu_{t}^{\ast k} - \mu^{\ast k} \|_{0}}{k}.
\end{equation}
By   combining Equation \ref{eqn:theDic} and Lemma \ref{lem:driftConv}, we observe \[
\lim_{t \rightarrow 0} \frac{1}{t} \frac{1}{k} \E_{\mu_{t}^{\ast k} - \mu^{\ast k}} \|g\| = \frac{1}{k} \sum_{i=1}^{k} \E_{\mu^{\ast(i-1)} \ast \eta \ast \mu^{\ast (k-i)}} \|g\|
\]
and  \begin{equation}\label{eqn:driftDiff1}
\left|\limsup_{t \rightarrow 0} \frac{1}{t} [l(\mu_{t}) - l(\mu)] -  \frac{1}{k} \sum_{i=1}^{k} \E_{\mu^{\ast(i-1)} \ast \eta \ast \mu^{\ast (k-i)}} \|g\|\right| \le \limsup_{t \rightarrow 0} \left|\frac{K_{1}}{t} \cdot \frac{ \| \mu_{t}^{\ast k} - \mu^{\ast k} \|_{0}}{k}\right|.
\end{equation}
The inequality   remains valid if we replace $\limsup_{t} \frac{l(\mu_{t}) - l(\mu)}{t}$ with $\liminf_{t} \frac{l(\mu_{t}) - l(\mu)}{t}$. Hence, to conclude that $\lim_{t\rightarrow 0} \frac{l(\mu_{t}) - l(\mu)}{t} = \lim_{k} \frac{1}{k} \sum_{i=1}^{k} \E_{\mu^{\ast (i-1)} \ast \eta \ast \mu^{\ast (k-i)}} \|g\|$, it suffices to prove that \[
\lim_{k \rightarrow \infty} \limsup_{t \rightarrow 0} \left|\frac{K_{1}}{t} \cdot \frac{ \| \mu_{t}^{\ast k} - \mu^{\ast k} \|_{0}}{k}\right| = 0.
\]
Again, Lemma \ref{lem:driftConv} implies\begin{equation}\label{eqn:driftDiff2}\begin{aligned}
\limsup_{t \rightarrow 0} \left| \frac{K_{1}}{t} \cdot \frac{ \|\mu_{t}^{\ast k} - \mu^{\ast k} \|_{0} }{k} \right| &= \limsup_{t \rightarrow 0} \frac{K_{1}}{kt} \cdot \left(  t \left\|\sum_{i=1}^{k} \mu^{\ast (i-1)} \ast \eta \ast \mu^{\ast(k-i)} \right\|_{0} + o(t) \right)\\
& = \frac{K_{1}}{k} \cdot \left\|\sum_{i=1}^{k} \mu^{\ast (i-1)} \ast \eta \ast \mu^{\ast (k-i)} \right\|_{0}.
\end{aligned}
\end{equation}
Let $f(g) := \eta(g) / \mu(g)$ for $g \in \supp \mu$. Then $f$ is a $\mu$-integrable function (since $\sum_{g} |f(g)| \mu(g) = \sum_{g} |\eta(g)|< \infty$) with mean 0 (since $\sum_{g} f(g) \mu(g) = \sum_{g} \eta(g) = 0$). Hence, we have \[\begin{aligned}
\frac{1}{k} \left\|\sum_{i=1}^{k} \mu^{\ast (i-1)} \ast \eta \ast \mu^{(k-i)} \right\|_{0} &= \frac{1}{k} \sum_{g_{i} \in G} |f(g_{1}) + \cdots + f(g_{k})| \, \mu(g_{1}) \cdots \mu(g_{k})\\
&\le \E_{\mu^{k}} \left|\frac{f(g_{1}) + \cdots + f(g_{k})}{k}\right|.
\end{aligned}
\]
The final term tends to 0 as $k$ tends to infinity, by the recurrence of integrable balanced random walk on $\mathbb{R}$ and the subadditive ergodic theorem. This concludes the differentiability of $l(\mu_{t})$.

Let us now combine Inequality \ref{eqn:driftDiff1} and \ref{eqn:driftDiff2}: under the assumptions on $\mu, \eta, \{\mu_{t}\}$, we have \begin{equation}\label{eqn:driftDiff3}
\left| \frac{d}{dt} l(\mu_{t}) \Big|_{t=0} - \frac{1}{k} \sum_{i=1}^{k} \E_{\mu^{\ast(i-1)} \ast \eta \ast \mu^{\ast(k-i)}} \|g\| \right| \le \frac{K_{1}}{k} \cdot \left\|\sum_{i=1}^{k} \mu^{\ast (i-1)} \ast \eta \ast \mu^{\ast (k-i)} \right\|_{0}
\end{equation}
for each $k$. Here, the constant $K_{1}$ depends on $\mu$, and can be kept the same even if we replace $\mu$ with $\mu'$ such that $\|\mu' - \mu\| < \epsilon/2$.

Let us now show continuity of the derivative. Fix $\delta >0$. Thanks to the previous argument, there exists $k>0$ such that \begin{equation}\label{eqn:driftDiff4}
\frac{1}{k} \left\|\sum_{i=1}^{k} \mu^{\ast (i-1)} \ast \eta \ast \mu^\ast {(k-i)} \right\|_{0} < \frac{\delta}{3K_{1}}.
\end{equation}
Note that this finite combination of convolutions of $\mu$ and $\eta$ is continuous with respect to $\mu$ and $\eta$ (under the $\|\cdot \|_{0}$-topology). For the same reason, \[
\frac{1}{k} \sum_{i=1}^{k} \E_{\mu^{\ast (i-1)} \ast \eta \ast \mu^{\ast(k-i)}} \|g\| 
\]
is continuous with respect to $\mu$ and $\eta$ (under the $\|\cdot\|_{0,1}$-topology). Hence, we can take $0 < \epsilon_{1} < \epsilon/2$ such that, for all probability measures $\mu'$ and for all signed measures $\eta'$ such that $\|\mu' - \mu\|_{0, 1}, \|\eta' - \eta\|_{0, 1} < \epsilon_{1}$, we have: \begin{equation}\label{eqn:driftDiff5}\begin{aligned}
\frac{1}{k} \left\|\sum_{i=1}^{k} \mu'^{\ast (i-1)} \ast \eta' \ast \mu'^{\ast (k-i)} \right\|_{0} &< \frac{\delta}{2K_{1}}, \\
\left| \frac{1}{k} \sum_{i=1}^{k} \E_{\mu'^{\ast (i-1)} \ast \eta' \ast \mu'^{\ast(k-i)}} \|g\| - \frac{1}{k} \sum_{i=1}^{k} \E_{\mu^{\ast (i-1)} \ast \eta \ast \mu^{\ast(k-i)}} \|g\| \right| &< \frac{\delta}{2}.
\end{aligned}
\end{equation}
Now let $\{\mu_{t}'\}_{t}$ be another family of probability measures such that $\|\mu_{t}' - \mu'\|_{0, 1} \rightarrow 0$, $\frac{1}{t}\|\mu_{t}' - \mu' - t\eta'\|_{0, 1} \rightarrow 0$. By plugging Inequality \ref{eqn:driftDiff4} and \ref{eqn:driftDiff5} into Inequality \ref{eqn:driftDiff3}, we conclude that $\left|\frac{d}{dt} l(\mu_{t}) - \frac{d}{dt} l(\mu_{t}')\right| < 2\delta$   for small enough $t$ as desired.
\end{proof}

\section{Squeezing isometries and escape rate} \label{section:squeezing}

As  in Lemma \ref{lem:conditionedDist}, we describe another simple geometric situation.   As long as this geometric condition holds for high probability, Theorem \ref{thm:driftDiffSqueeze} can be deduced from elementary arguments.

\begin{definition}\label{dfn:squeezingProbable}
Let $(X, G, o)$ be as in Convention \ref{conv:main} and let $\epsilon>0$. We define a collection $ A_{k; \epsilon} \subseteq G^{\Z}$ of step paths as follows. A step path $\mathbf{g} \in G^{\Z}$   belongs to $\mathcal{A}_{k; \epsilon}$ if there exist $E>0$ and an $(E, \epsilon)$-squeezing sequence of paths $(\gamma_{1}, \ldots, \gamma_{N})$, $(\eta_{1}, \ldots, \eta_{M})$ on $X$ such that the following sequences are all $E$-aligned for every $m \ge k$ and $n \ge 0$:\[
\big(Z_{-n} o, \gamma_{1} \big), \,\, \big(\gamma_{N}, Z_{m} o \big), \, \,  \big(Z_{-(m+1)} o, \eta_{1}\big), \,\, \big(\eta_{M}, Z_{n-1} o\big).
\]
\end{definition}

We now state the main probabilistic estimate. 

\begin{prop}\label{prop:squeeze}
Let $(X, G, o)$ be as in Theorem \ref{thm:driftDiffSqueeze} and let $\mu$ be a strongly non-elementary probability measure on $G$. Then there exists $\epsilon>0$ such that for each $\epsilon' > 0$, the following holds for some $K = K(\mu, \epsilon, \epsilon')$.

For every sequence of probability measures $\{\mu_{i}\}_{i \in \Z}$ with $\|\mu_{i} - \mu\|_{0} < \epsilon$, for every $j \in \Z \setminus \{0\}$ and $k \in \Z_{>0}$, and for every prescribed choice  of $g_{0}, g_{j} \in G$, we have \[
\Prob_{\prod_{i} \mu_{i}} \big(\mathbf{g} \in A_{k; \epsilon'} \,\big| \, g_{0}, g_{j}\big) \ge 1 - K e^{-k/K}.
\]
\end{prop}

\begin{proof}
The condition for $\big(Z_{-n} o, \gamma_{1} \big), \,\, \big(\gamma_{N}, Z_{m} o \big)$ and the condition for $ \big(Z_{-(m+1)} o, \eta_{1}), \,\, (\eta_{M}, Z_{n-1} o\big)$ are symmetric, so we will prove that the   former condition holds for high probability.

Let $\mu$ be a strongly non-elementary probability measure on $G$. This implies that $\supp \mu$ generates two independent BGIP elements $g$ and $h$. Here,   $g$ is chosen  as the one specified in the definition of strongly non-elementary measure, namely: \begin{itemize}
\item If $G$ is the mapping class group, we choose $g$ to be a pseudo-Anosov (in fact, it must be so \emph{a fortiori}).
\item If $G = \Out(F_{N})$, we take $g$ to be a principal fully irreducible outer automorphism.
\item If $G$ is a countable group of isometries of a CAT(-1) space, then we take $g$ to be an arbitrary loxodromic element.
\item If $G$ is a countable group of automorphisms of a CAT(0) cube complex, then we take $g$ to be an automorphism that skewers a pair of strongly separated hyperplanes.
\item If $G$ is the surface group or a free product of nontrivial groups, then we take $g$ to be the element that is 0-constricting.
\end{itemize}

Let $N_{1},   l, l'$ be such that \[
\min\big(\mu^{\ast N_{1}} (  g^{l}), \mu^{\ast N_{1}}(h^{l'}) \big):= \epsilon_{min} > 0.
\]
We then choose $\epsilon = \frac{1}{2} \epsilon_{min}$. Then for any probability measures $\mu_{i}$'s with $\|\mu_{i} - \mu\|_{0} < \epsilon$ for each $i$, we have \begin{equation}\label{eqn:stabilizedMu}\begin{aligned}
(\mu_{k+1} \ast \cdots \ast \mu_{k+N_{1}})(g) \ge 2^{-N_{1}} \mu^{\ast N_{1}} (  g^{l}) \ge 2^{-N_{1}} \epsilon_{min},\\
(\mu_{k+1} \ast \cdots \ast \mu_{k+N_{1}})(h) \ge 2^{-N_{1}} \mu^{\ast N_{1}} (  h^{l'}) \ge 2^{-N_{1}} \epsilon_{min}.
\end{aligned}
\end{equation}

Now, let $K_{0}>0$ be the constant given by Proposition \ref{prop:SchottkyPA} for   $g^{l}$, $h^{l'}$ and $  A = \{g^{l}, h^{l'}\}$, which determines $D_{0}, E_{0}$ as well. Now let $n$ be long enough that $\Gamma_{g} :=   (o, g^{l} o, \ldots, g^{ln} o)$ is an $(E_{0}, \epsilon')$-squeezing axis. We now construct a long and large enough $K_{0}$-Schottky set $S \subseteq G^{M_{0}}$ as in Proposition \ref{prop:SchottkyPA} for $n$. Namely, $\Gamma(s)$ for each $s \in S$ has a subpath that is a translate of   $(o, g^{l} o, \ldots, g^{ln} o)$. Furthermore, Inequality \ref{eqn:stabilizedMu} tells us that $(\mu_{k+1} \ast \cdots \ast \mu_{k+M_{0}})(s) > 2^{-M_{0}N_{1}} \epsilon_{min}$ for each $s \in S$. 

Now let $\kappa = \kappa(2^{-M_{0}N_{1}} \epsilon_{min}, S)$ for this specific $S$, as in Proposition \ref{prop:gouezelRWEventual}. Then we have \[
\Prob_{(g_{i})_{i \in \Z} \sim \prod_{i \in \Z} \mu_{i}} \left( \begin{array}{c}  \exists \textrm{  an $S$-Schottky axis $\gamma$ such that $(Z_{-m} o, \gamma,   Z_{k+l} o)$ is}\\
 \textrm{$D_{0}$-semi-aligned for each $m, l\ge 0$} \end{array}\, \Big| \, g_{0}, g_{j} \right) \ge 1 - \frac{1}{\kappa} e^{-\kappa (k-1) - 2},
\]
regardless of $j$. Furthermore, when $(Z_{-m} o, \gamma, Z_{n+l} o)$ is $D_{0}$-semi-aligned for some $S$-Schottky axis $\gamma$, then Proposition \ref{prop:BGIPWitness} tells us that $\gamma$ is $0.1E_{0}$-fellow traveling with the subpath $\varphi \Gamma_{g}$, for some $\varphi$. In particular, we have: \[
\textrm{$(Z_{-m}, \varphi \Gamma_{g}, Z_{n+l} o)$ is $E_{0}$-aligned for each $m, l\ge 0$}.
\]
This ends the proof.
\end{proof}

Our notion of $A_{k, \epsilon'}$ is justified   by the following lemma.   Recall that $R_{n, m}(\mathbf{g}) := \|g_{-m} \cdots g_{-1}\| + \|g_{0}\| + \|g_{1} \cdots g_{n}\| - \|g_{-m} \cdots g_{n}\|$.

\begin{lem}\label{lem:squeezeCompare}
Let $\mathbf{g} \in A_{k;\epsilon'}$. Then for each $n, m, n', m' \ge k$, we have \[
|R_{n, m}(\mathbf{g})- R_{n', m'}(\mathbf{g})| < 2\epsilon'.
\]
\end{lem}

\begin{proof}
We will prove that $|R_{n, m}(\mathbf{g}) - R_{n', m}(\mathbf{g})| < \epsilon'$. Together with its symmetric counterpart, this leads to the conclusion.

Since $\mathbf{g} \in A_{k; \epsilon'}$, there exists an $(E, \epsilon')$-squeezing sequence of  $(\gamma_{1}, \ldots, \gamma_{N})$ for some $E$ such that \[
\big(Z_{-(m+1)} o, \gamma_{1}\big),\,\, \big(o =: Z_{0} o, \gamma_{1}\big),\,\,  \big(\gamma_{N}, Z_{n} o\big), \,\, \big( \gamma_{N}, Z_{n'} o\big)
\] are all $E$-aligned. Since $(\gamma_{1}, \ldots, \gamma_{N})$ is $(E, \epsilon')$-squeezing, we conclude that \[
\big|d(Z_{-(m+1)} o, Z_{n} o) + d(o, Z_{n'} o) - d( o, Z_{n} o) - d(Z_{-(m+1)} o, Z_{n'} o) \big| < \epsilon'.
\]
By adding and subtracting $d(Z_{-(m+1)} o, Z_{-1} o) + d(o, Z_{-1} o)$, the above display can be rewritten as \[
\big| R_{n, m}(\mathbf{g}) - R_{n', m} (\mathbf{g}) \big| < \epsilon'. \qedhere
\]
\end{proof}

Note that for each $\epsilon'$, we have \[
A_{1; \epsilon'} \subseteq A_{2; \epsilon'} \subseteq \ldots, \quad \Prob(A_{k; \epsilon'}) \nearrow 1.
\]
It follows that $\cap_{n >0} \cup_{k >0} A_{k; 1/n}$ has measure 1. Moreover, Lemma \ref{lem:squeezeCompare} says that   the sequence $\{R_{n, m}(\mathbf{g})\}_{n, m}$ is Cauchy for each $\mathbf{g} \in \cap_{n >0} \cup_{k>0} A_{k;1/n}$. Hence, we have:

\begin{cor}\label{cor:squeezeRV}
In the setting of Proposition \ref{prop:squeeze}, \[
R(\mathbf{g}) := \lim_{n, m \rightarrow +\infty} R_{n, m}(\mathbf{g})
\]
is well-defined almost everywhere.
\end{cor}
Note that in Corollary \ref{cor:squeezeRV}, $R_{n, m}(\mathbf{g})$   is  well-defined for each $n, m$ and is finite everywhere. This does not guarantee that $R_{n, m}(\mathbf{g})$ is integrable. Still, $R(\mathbf{g}) = \lim_{n, m} R_{n, m}(\mathbf{g})$ is defined as a finite value almost everywhere. 

If , moreover,  $\mu$   has  finite first moment and if $\mu_{i}$'s are suitably close to $\mu$ in $\|\cdot \|_{0, 1}$-norm,  then we also have the $L^{1}$-convergence of $R_{n, m}$.

\begin{lem}\label{lem:squeezeL1}
Let $(X, G)$ be as in Theorem \ref{thm:driftDiffSqueeze}, let $\mu$ be a strongly non-elementary probability measure on $G$ with finite first moment and let $\epsilon> 0$ be as in Proposition \ref{prop:squeeze}. Then for each $\epsilon' > 0$ there exists $N>0$ such that, for every sequence of probability measures $\{\mu_{i}\}_{i}$  with $\|\mu_{i} - \mu\|_{0, 1} < \epsilon$ for each $i$, we have \[
\E_{\prod_{i=1}^{\infty} \mu_{i} \, \otimes\, 1_{g_{0} = g} \,\otimes \, \prod_{i=1}^{\infty} \mu_{-i}}   \big|R(\mathbf{g}) - R_{n, m}(\mathbf{g})\big|   < \epsilon'
\]
 for every $n, m > N$ and for every $g \in G$.
\end{lem}

\begin{proof}
Observe a simple inequality: for   four  points $x, y, x', y' \in X$, we have \[
|d(x, y) - d(x', y')| \le d(x, x') + d(y, y').
\]
It follows that for any $0 \le n \le n'$ and $0 \le m \le m'$, we have  \begin{equation}\label{eqn:CAT(-1)SqueezePrelim}
|R_{n', m'}(\mathbf{g}) - R_{n, m}(\mathbf{g})| \le 2\sum_{i=n+1}^{n'} \|g_{i}\| + 2 \sum_{i=m+1}^{m'} \|g_{-i}\|.
\end{equation}

Let $\eta = 0.01 \epsilon'$. For each $n, m > N$ we claim that  \begin{equation}\label{eqn:CAT(-1)SqueezeL1}
|R(\mathbf{g}) - R_{n, m}(\mathbf{g})| \le 10\eta + 2\sum_{1 \le |i| \le N} \|g_{i}\| \cdot 1_{\mathbf{g} \notin A_{N-1; \eta}} +  2\sum_{|i| > N} \|g_{i}\| \cdot 1_{\mathbf{g} \notin A_{|i|-1; \eta}}.
\end{equation}

To prove this, let $k = \sup\{i : \mathbf{g} \notin A_{i-1; \eta}\}$. Lemma \ref{lem:squeezeCompare} tells us that $|R(\mathbf{g}) - R_{n', m'} (\mathbf{g})| \le 10\eta$ for any $n', m' \ge k$. In particular, Inequality \ref{eqn:CAT(-1)SqueezeL1} holds when $n, m \ge k$.

If $n, m < k$, then we have \[\begin{aligned}
|R(\mathbf{g}) - R_{n, m}(\mathbf{g})| &\le |R(\mathbf{g}) - R_{k, k}(\mathbf{g})| +
|R_{k, k}(\mathbf{g}) - R_{n, m}(\mathbf{g})| \\
&\le 10\eta  + 2 \sum_{1 \le |i| \le k} \|g_{i}\| & (\because \textrm{Inequality \ref{eqn:CAT(-1)SqueezePrelim}}) \\
&\le 10 \eta + 2 \sum_{1 \le |i| \le N} \|g_{i}\| \cdot 1_{\mathbf{g} \notin A_{N-1; \eta}} + 2\sum_{|i| > N} \|g_{i}\| \cdot 1_{\mathbf{g} \notin A_{|i|-1; \eta}}. & (\because N < n< k)
\end{aligned}
\]
If $n< k \le m$, then we have \[\begin{aligned}
|R(\mathbf{g}) - R_{n, m}(\mathbf{g})| &\le |R(\mathbf{g}) - R_{k, m}(\mathbf{g})| +
|R_{k, m}(\mathbf{g}) -R_{n, m}(\mathbf{g})| \\
&\le 10 \eta + 2 \sum_{i=n+1}^{k} \|g_{i}\|  & (\because \textrm{Inequality \ref{eqn:CAT(-1)SqueezePrelim}}) \\
&\le 10 \eta  + 2\sum_{|i| > N} \|g_{i}\| \cdot 1_{\mathbf{g} \notin A_{|i|-1; \eta}}. & (\because N < n < k)
\end{aligned}
\]
The case of $m < k \le n$ can be handled in a similar way.

We now estimate the expectation of the RHS of Inequality \ref{eqn:CAT(-1)SqueezeL1}. For each $j \in \Z$ and $k > 0$ we have \[\begin{aligned}
\E_{\prod_{i=1}^{\infty} \mu_{i}\, \otimes \,1_{g_{0} = g}\, \otimes\, \prod_{i=1}^{\infty} \mu_{-i}}\left[ \|g_{j}\| \cdot 1_{\mathbf{g} \notin A_{k; \eta}} \right] &= \E_{\mu_{j}}\big[ \|g_{j}\| \Prob (A_{k; \eta}^{c} | g_{0}, g_{j})\big] \\
&\le \E_{\mu_{j}} \|g_{j}\| \cdot K e^{-k/K} & (\because \textrm{Proposition \ref{prop:squeeze}})\\
&\le K(1+\epsilon) \|\mu\|_{1} \cdot e^{-k/K}.
\end{aligned}
\]

  Therefore, we have \[\begin{aligned}
&\E\Big[10\eta + 2\sum_{1 \le |i| \le N} \|g_{i}\| \cdot 1_{\mathbf{g} \notin A_{N-1; \eta}} +  2\sum_{|i| > N} \|g_{i}\| \cdot 1_{\mathbf{g} \notin A_{|i|-1; \eta}}\Big] \\
&\le 10\eta + 2 K (1+\epsilon) \|\mu\|_{1} \cdot \left( N e^{-(N-1)/K} + \frac{1}{1-e^{-1/K}} e^{-(N-1)/K} \right).
\end{aligned}
\]
By taking suitably large $N$, we can guarantee $\E|R(\mathbf{g}) - R_{n, m}(\mathbf{g})| < 20\eta$ for each $n, m > N$. 
\end{proof}

We now prove Theorem \ref{thm:driftDiffSqueeze}  which was as follows:

\begin{thm}\label{thm:driftDiffSqueezeRe}
Let $(X, G)$ be as in Theorem \ref{thm:driftDiffSqueeze},  let $\mu$ be a strongly non-elementary probability measure on $G$ with finite first moment, let $\eta$ be a signed measure on $G$ such that $\|\eta\|_{0, 1} < \infty$, and let $\{\mu_{t}: t \in [-1, 1]\}$ be a family of probability measures such that \[
\| \mu_{t} - \mu - t \eta \|_{0, 1} = o(t).
\]
Then $l(\mu_{t})$ is differentiable at $t=0$, with derivative \[
 \sigma_{1}(\mu, \eta) := \E_{\eta} d(o, go) - \E_{\prod_{i=1}^{\infty} \mu \,\otimes\, \eta \,\otimes\, \prod_{i =1}^{\infty} \mu} [R(\mathbf{g})]
\]
Moreover, this derivative is continuous with respect to $\mu$ and $\eta$.
\end{thm}

\begin{proof}[Proof of Theorem \ref{thm:driftDiffSqueezeRe}]
Let $\epsilon > 0$ be as in Proposition \ref{prop:squeeze}. Without loss of generality , we assume that $\|\mu_{t} - \mu\|_{0, 1} \le \epsilon$ for $t \in [-1, 1]$.

Recall first that Corollary \ref{cor:bddDefect} guarantees a uniform bound $K_{1}$ such that \[
\E_{\mu_{t}^{i} \,\otimes\, 1_{g_{0} = g} \,\otimes \,  \mu^{j}} R_{i, j}(\mathbf{g}) < K_{1}
\] for any $i, j > 0$ and $g \in G$. Given $n >0$ and $t \in [-1, 1]$ we define \[
E_{n, t}:= \frac{1}{n} \sum_{g \in G} (\mu_{t}(g) - \mu(g)) \cdot \sum_{i=1}^{n} \E_{\mu_{t}^{n-i} \,\otimes\,  1_{g_{0} = g} \,\otimes\, \mu^{i-1}} \E[R_{n-i, i-1} (\mathbf{g})].
\]
Then Equation \ref{eqn:driftLipschitz} reads   \[
\frac{L(\mu_{t}^{\ast n}) - L(\mu^{\ast n})}{n}  = \E_{\mu_{t} - \mu} \|g\| + E_{n, t}.
\]

Now for $(g, u) \in G \times [0, 1]$ we define  \[
F_{n, t}(g, u) := \left(\mu_{t}(g) - \mu(g) \right) \cdot \E_{\mu_{t}^{n-i} \,\otimes\,  1_{g_{0} = g} \,\otimes\, \mu^{i-1}} [R_{n-i, i-1}(\mathbf{g}) ] \quad (i := n \lceil u/n \rceil).
\] 
Then $\sum_{g \in G} \int_{0}^{1} F_{n, t}(g, u) \,du = E_{n, t}$, where $du$ is the Lebesgue measure on $[0, 1]$. Moreover, $|F_{n, t}(g, u)|$ is pointwise dominated by $  2K_{1} \big( \mu(g) + \mu_{t}(g)\big)$, which is integrable: $\sum_{g \in G} 2K_{1}\big(\mu(g)+\mu_{t}(g)\big) = 2K_{1} < +\infty$. Finally, for all $u \in (0, 1)$ and for all $g \in G$, we have \[
\lim_{n \rightarrow +\infty} F_{n, t}(g, u) =  \left(\mu_{t}(g) - \mu(g) \right) \cdot \E_{\prod_{i=1}^{\infty} \mu_{t} \, \otimes \, 1_{g_{0} = g} \, \otimes \,\prod_{i=1}^{\infty} \mu} R(\mathbf{g})
\]
by Lemma \ref{lem:squeezeL1}. Hence, by the dominated convergence theorem, we deduce that \[\begin{aligned}
l(\mu_{t}) - l(\mu)  &= \E_{\mu_{t} - \mu} \|g\|+\lim_{n \rightarrow +\infty} E_{n, t} = \E_{\mu_{t} - \mu} \|g\|+ \lim_{n \rightarrow +\infty} \sum_{g \in G} \int_{0}^{1}F_{n, t}(g, u) \,du\\
&= \E_{\mu_{t} - \mu} \|g\|+ \sum_{g \in G} \int_{0}^{1} \left( \lim_{n} F_{n, t}(g, u)\right) \, du \\
&= \E_{\mu_{t} - \mu} \|g\|+ \sum_{g \in G} \left(\mu_{t}(g) - \mu(g) \right) \cdot \E_{\prod_{i=1}^{\infty} \mu_{t} \, \otimes \, 1_{g_{0} = g} \, \otimes \,\prod_{i=1}^{\infty} \mu} R(\mathbf{g}).
\end{aligned}
\]
We now divide this by $t$ and claim \begin{equation}\label{eqn:squeezeDiff}
\lim_{t \rightarrow 0} \frac{1}{t} \big(l(\mu_{t}) - l(\mu) \big) = \E_{\eta} \|g\| - \E_{\prod_{i=1}^{\infty} \mu \, \otimes\, \eta \,\otimes \, \prod_{i =1}^{\infty} \mu} R(\mathbf{g}).
\end{equation}
First, we have \[
\left|\frac{1}{t} ( \E_{\mu_{t} - \mu} \|g\|) - \E_{\eta} \|g\|\right| = \frac{1}{t} \left|\E_{\mu_{t} - \mu - t\eta} \|g\| \right| \le \frac{1}{t} \| \mu_{t} - \mu - t\eta\|_{1} \rightarrow 0. \quad( \textrm{as $t\rightarrow 0$})
\]
Next, note that
\[\begin{aligned}
&\left|\frac{1}{t}\sum_{g \in G} \left(\mu_{t}(g) - \mu(g) \right) \cdot \E_{\prod_{i=1}^{\infty} \mu_{t} \, \otimes \,  1_{g_{0} = g} \, \otimes \,  \prod_{i=1}^{\infty}  \mu} R(\mathbf{g})  - \sum_{g \in G} \eta(g) \cdot \E_{\prod_{i=1}^{\infty} \mu \, \otimes \, 1_{g_{0} = g} \, \otimes \, \prod_{i=1}^{\infty} \mu} R(\mathbf{g}) \right| \\
& \le \sum_{g \in G} \left|\frac{\mu_{t}(g) - \mu(g)}{t} - \eta(g) \right| \left|\E_{\prod_{i=1}^{\infty} \mu_{t} \, \otimes \, 1_{g_{0} = g} \, \otimes \, \prod_{i=1}^{\infty} \mu} R(\mathbf{g}) \right|  \\
&+ \sum_{g \in G} |\eta(g)| \left| \E_{\prod_{i=1}^{\infty} \mu_{t} \, \otimes \, 1_{g_{0} = g} \, \otimes \,  \prod_{i=1}^{\infty} \mu} R(\mathbf{g})- \E_{\prod_{i=1}^{\infty} \mu \, \otimes \, 1_{g_{0} = g} \, \otimes \,  \prod_{i=1}^{\infty} \mu} R(\mathbf{g})\right|.
\end{aligned}
\]
We   now show that both summations tend to zero. The first summation is bounded by $K_{1} \cdot \|\frac{\mu_{t} - \mu}{t} - \eta\|_{0}$, which tends to 0 as $t \rightarrow 0$. The second summation is over $G$ and the summand is bounded by $2K_{1} |\eta(g)|$ for any $t$, which is integrable. Hence, we can conclude by DCT once we show for each $g \in G$ that \begin{equation}\label{eqn:squeezeDiffFinal}
\lim_{t \rightarrow 0} \E_{\prod_{i=1}^{\infty} \mu_{t} \, \otimes \, 1_{g_{0} = g} \, \otimes \,  \prod_{i=1}^{\infty} \mu} R(\mathbf{g}) = \E_{\prod_{i=1}^{\infty} \mu \, \otimes \, 1_{g_{0} = g} \, \otimes \,  \prod_{i=1}^{\infty} \mu} R(\mathbf{g}).
\end{equation}
Let $\epsilon '> 0$ and let $N$ be as in Lemma \ref{lem:squeezeL1}. Then we have \[\begin{aligned}
\left| \E_{\prod_{i=1}^{\infty} \mu_{t} \, \otimes \, 1_{g_{0} = g} \, \otimes \,  \prod_{i=1}^{\infty} \mu} R(\mathbf{g}) - \E_{\prod_{i=1}^{N} \mu_{t} \, \otimes \, 1_{g_{0} = g} \, \otimes \,  \prod_{i=1}^{N} \mu} R(\mathbf{g}) \right| &< \epsilon' \quad (t \in [-1, 1]), \\
\left| \E_{\prod_{i=1}^{\infty} \mu\, \otimes \, 1_{g_{0} = g} \, \otimes \,  \prod_{i=1}^{\infty} \mu} R(\mathbf{g}) - \E_{\prod_{i=1}^{N} \mu \, \otimes \, 1_{g_{0} = g} \, \otimes \,  \prod_{i=1}^{N} \mu} R(\mathbf{g}) \right| &< \epsilon', \\
\lim_{t \rightarrow 0} \E_{\prod_{i=1}^{N} \mu_{t} \, \otimes \, 1_{g_{0} = g} \, \otimes \,  \prod_{i=1}^{N} \mu} R(\mathbf{g}) = \E_{\prod_{i=1}^{N} \mu \, \otimes \, 1_{g_{0} = g} \, \otimes \,  \prod_{i=1}^{N} \mu} R(\mathbf{g}).
\end{aligned}
\]
  It follows that the two expectations in Equation \ref{eqn:squeezeDiffFinal} differ by at most $2\epsilon'$ eventually as $t \rightarrow 0$. We now send $\epsilon' \rightarrow 0$ to establish Equation \ref{eqn:squeezeDiffFinal}.

We now discuss the continuity of $\sigma_{1}(\mu, \eta)$ with respect to $\mu$ and $\eta$. For this let $\mu_{n}, \eta_{n}$ be probability measures and signed measures, respectively, such that $\|\mu_{n} - \mu\|_{0, 1}, \|\eta_{n} - \eta\|_{0, 1} \rightarrow 0$. Without loss of generality we assume $\|\mu_{n} - \mu\|_{0, 1} < \epsilon$ for each $n$. For notational purposes we temporarily introduce \[\begin{aligned}
D_{n}(g) &:= \E_{\prod_{i=1}^{\infty} \mu_{n} \,\otimes\, 1_{g_{0} = g} \, \otimes \, \prod_{i=1}^{\infty} \mu_{n}} R(\mathbf{g}), \\
D(g) &:= \E_{\prod_{i=1}^{\infty} \mu \,\otimes\, 1_{g_{0} = g} \, \otimes \, \prod_{i=1}^{\infty} \mu} R(\mathbf{g}).
\end{aligned}
\]
  Note that the above argument for Equation \ref{eqn:squeezeDiffFinal} did not require the differentiability of $\{\mu_{t}\}_{t \in [-1, 1]}$. Namely, the exactly same argument implies that 

\[
\lim_{n \rightarrow \infty} D_{n}(g) = D(g)
\]
holds for each $g \in G$. Moreover, Corollary \ref{cor:bddDefect} gives the bound $D_{n}(g), D(g) \le K_{1}$ for all $g \in G$ and $n \in \Z_{>0}$.

Now given $\epsilon > 0$, let $A$ be a finite set of $G$ such that $\|\eta 1_{G \setminus A}\|_{0,1} < \epsilon$. We then have \[
\begin{aligned}
|\sigma_{1}(\mu_{n}, \eta_{n}) - \sigma_{1}(\mu, \eta)| &= \left| \E_{\eta_{n}} d(o, go) - \E_{\eta}d(o, go) - \left(\sum_{g \in G} \eta_{n}(g) D_{n}(g) - \sum_{g \in G} \eta(g) D(g)\right) \right|\\
&\le \|\eta_{n} - \eta\|_{1} +  \sum_{g \in A} |\eta_{n}(g) D_{n}(g) - \eta(g) D(g)| \\
&+ \sum_{g \notin A} |\eta_{n}(g) - \eta(g)| |D_{n}(g)| + \sum_{g \notin A} |\eta(g)| |D_{n}(g) - D(g)|.
\end{aligned}
\]
Here, the first summation is over a finite set, and it tends to 0 as $n$ goes to infinity. The second summation is bounded by $K_{1} \|\eta_{n} - \eta\|_{0}$, which also tends to 0. The third summation is bounded by $2K_{1} \cdot \| \eta 1_{G \setminus A}\|_{0} < 2K_{1}\epsilon$. Letting $\epsilon$ to zero completes the proof.

\end{proof}

\section{Exponentially squeezing geodesics} \label{section:expSqueeze}

In this section, we discuss higher regularity of the escape rate on CAT(-1) spaces and their ilk.
Throughout the section we assume that:

\begin{conv}\label{conv:expSqueeze}
  Let $(X, G)$ be one of the following: \begin{itemize}
\item a CAT(-1) space and a countable group of its isometries;
\item a CAT(0) cube complex and a countable group of its isometries;
\item relatively hyperbolic group equipped with a Green metric;
\item the standard Cayley graph of a surface group, or;
\item the standard Cayley graph of a free product of nontrivial groups.
\end{itemize}
\end{conv}

\begin{prop}\label{prop:expSqueeze}
Let $(X,   G)$ be as in Convention \ref{conv:expSqueeze} and let $\mu$ be a strongly non-elementary probability measure on $G$. Then there exist $\epsilon, K>0$ such that for every sequence of probability measures $\{\mu_{i}\}_{i}$ with $\|\mu_{i} - \mu\|_{0, 1} < \epsilon$, for every $k \in \Z_{>0}$, and for every combination of $N$ integers $\{j_{1}, \ldots, j_{N}\} \subseteq \Z \setminus \{0\}$, we have \[
\Prob_{\prod_{i} \mu_{i}} \big(\mathbf{g} \in A_{k; e^{-k/K}} \,\big| \, g_{0}, g_{j_{1}}, g_{j_{2}}, \ldots, g_{j_{N}}\big) \ge 1 - K   e^{-(k/K - N)}.
\]
\end{prop}

\begin{proof}
 Our goal is to find an $(E_{0}, e^{-k/K})$-squeezing sequence   $(\gamma_{1}, \ldots, \gamma_{N})$ and $(\eta_{1}, \ldots, \eta_{N})$ for the random path with high probability.   Since the condition for $\big(Z_{-n} o, \gamma_{1} \big), \,\, \big(\gamma_{N}, Z_{m} o \big)$ and the condition for $ \big(Z_{-(m+1)} o, \eta_{1}), \,\, (\eta_{M}, Z_{n-1} o\big)$ are symmetric,   we will only deal with the first condition.

Let $\mu$ be a strongly non-elementary probability measure on $G$. This implies that $\supp \mu$ generates two independent BGIP elements $g$ and $h$.   We choose $g$ as the one specified in the definition of strongly non-elementary measure, namely: \begin{itemize}
\item If $G$ is a countable isometry group of a CAT(-1) space, we take $g$ a loxodromic element.
\item If $G$ is a countable automorphism   group of a CAT(0) cube complex, then we take $g$ to be an automorphism that skewers a pair of strongly separated hyperplanes.
\item If $G$ is the surface group or a free product of nontrivial groups, then we take $g$ to be the element that is 0-squeezing.
\end{itemize}

Let $N_{1},   l, l'$ be such that \[
\min\big(\mu^{\ast N_{1}} (g^{l}), \mu^{\ast N_{1}}(h^{l'}) \big):= \epsilon_{min} > 0.
\]
We then choose $\epsilon = \frac{1}{2} \epsilon_{min}$. Then for any probability measures $\mu_{i}$'s with $\|\mu_{i} - \mu\|_{0} < \epsilon$ for each $i$, we have \begin{equation}\label{eqn:stabilizedMuSq}\begin{aligned}
(\mu_{k+1} \ast \cdots \ast \mu_{k+N_{1}})(g) \ge 2^{-N_{1}} \mu^{\ast N_{1}} (g^{l}) \ge 2^{-N_{1}} \epsilon_{min},\\
(\mu_{k+1} \ast \cdots \ast \mu_{k+N_{1}})(h) \ge 2^{-N_{1}} \mu^{\ast N_{1}} (h^{l'}) \ge 2^{-N_{1}} \epsilon_{min}.
\end{aligned}
\end{equation}

Now, let $K_{0}>0$ be the constant given by Proposition \ref{prop:SchottkyPA} for $g^{l}$, $h^{l'}$ and $A = \{g^{l}, h^{l'}\}$, which determines $D_{0}, E_{0}$ as well. Now let $n$ be   a large enough multiple of $l$ such that either:
\begin{itemize}
\item $d(o, g^{n} o) > 10E_{0}$ if $G$ is a countable isometry group of a CAT(-1) space; or
\item $\Gamma(g, \ldots, g) = (o, go, \ldots, g^{n} o)$ is $(E_{0}, 0)$-squeezing in the remaining case.
\end{itemize}
Here, for CAT(-1) space we are using Fact \ref{fact:CAT(-1)Squeeze}.

We now construct a long and large enough $K_{0}$-Schottky set $S \subseteq G^{M_{0}}$ as in Proposition \ref{prop:SchottkyPA} for $n$. Namely, $\Gamma(s)$ for each $s \in S$ has a subpath that is a translate of $(o, g^{k} o, \ldots, g^{kn} o)$. In this case,  every  $D_{0}$-semi-aligned sequence $N$ Schottky axes is $(E_{0}, E_{0} e^{-N/K'})$-squeezing for some $K'>0$. Note that Inequality \ref{eqn:stabilizedMuSq} tells us that $(\mu_{k+1} \ast \cdots \ast \mu_{k+M_{0}})(s) > 2^{-M_{0}N_{1}} \epsilon_{min}$ for each $s \in S$.

We take $\kappa = \kappa(2^{-M_{0}N_{1}} \epsilon_{min}, S)$ for this specific $S$, as in Proposition \ref{prop:gouezelRWEventual}. Then we have \[
\Prob_{(g_{i})_{i \in \Z} \sim \prod_{i \in \Z} \mu_{i}} \left.\left( \begin{array}{c}  \exists \textrm{($S$-)Schottky axes $\gamma_{1}, \ldots, \gamma_{ \lfloor \kappa k\rfloor } $ such that} \\
\textrm{$(Z_{-m} o, \gamma_{1}, \ldots, \gamma_{ \lfloor \kappa k\rfloor }, Z_{n+l} o)$ is}\\
 \textrm{$D_{0}$-semi-aligned for each $m, l\ge 0$} \end{array}\, \right|\, g_{0}, g_{j_{1}}, \ldots, g_{j_{N}} \right) \ge 1 - \frac{1}{\kappa} e^{-\kappa (k-1) - N},
\]
regardless of the values of $j_{1}, \ldots, j_{N}$. Recall that our choice of $S$ is such that $D_{0}$-semi-aligned sequences of   $\lfloor \kappa k \rfloor$-many Schottky axes are $(E_{0}, K e^{-k/K})$-squeezing for some $K>0$. This ends the proof.
\end{proof}

By Lemma \ref{lem:squeezeCompare}, we have $|R_{n, m}(\mathbf{g}) - R_{n', m'}(\mathbf{g})| < 2Ke^{-k/K}$ for all $n, m, n', m' \ge k$   provided that $\mathbf{g} \in A_{k; e^{-k/K}}$. Proposition \ref{prop:expSqueeze} tells us that for whatever finite set of $\{j_{1}, \ldots, j_{N}\}$ and values of $g_{0}, g_{j_{1}}, \ldots, g_{j_{N}}$, if we distribute $g_{i}$ independently according to $\mu_{i}$ for $i \notin \{0, j_{1}, \ldots, j_{N}\}$ as in the assumption of Proposition \ref{prop:expSqueeze},   then  $\{R_{n, m}(\mathbf{g})\}_{n, m}$ is a.e. Cauchy. This remains true if we then equip $g_{0}, g_{j_{1}}, \ldots, g_{j_{N}}$ with arbitrary finite measures or finite signed   measures. Hence, we can discuss the RV $R(\mathbf{g})$   when  some finitely many of $g_{i}$'s are defined on a signed measure space.

\begin{lem}\label{lem:expSqueezeL1}
Let $(X, d)$ be as in Convention \ref{conv:expSqueeze}, let $\mu$ be a strongly non-elementary probability measure on $G$ with finite moment and let $\epsilon, K > 0$ be as in Proposition \ref{prop:expSqueeze}.

Let $\{j_{1}, \ldots, j_{N}\}$ be $N$ integers, and let $\mu_{0}, \mu_{j_{1}}, \ldots, \mu_{j_{N}}$ be probability measures with finite $\| \cdot \|_{1}$-norm. For $i \in \Z \setminus \{0, j_{1}, \ldots, j_{N}\}$, let $\mu_{i}$ be a probability measures such that $\|\mu_{i} - \mu\|_{0, 1} < \epsilon$. Then we have \[
\E_{\prod_{i} \mu_{i}} \big| R(\mathbf{g}) - R_{n, m}(\mathbf{g})| < K_{2} e^{-\min(n, m)/K +N} \cdot \left(1 + \sum_{l=1}^{N} \|\eta\|_{1}\right)
\]
for every $n, m$, where $K_{2}$   depends only on $\epsilon, K$ and $\|\mu\|_{1}$.
\end{lem}

\begin{proof}
For notational convenience, let  $u = \min(n, m)$. As in the proof of Lemma \ref{lem:squeezeL1}, we will first bound $R - R_{n, m}$ with a summation of RVs related to $A_{k; e^{-k/K}}$'s and estimate the latter summation. One   subtlety is that $\{A_{k; e^{-k/K}}\}_{k=1}^{\infty}$ is not nested, whereas $\{A_{k; \epsilon}\}_{k=1}^{\infty}$ is.   To overcome this issue, we define \[
B_{k} := \cap_{i \ge k} A_{i; e^{-i/K}}.
\]

This time, we claim  \begin{equation}\label{eqn:expSqueezeL1}
|R(\mathbf{g}) - R_{n, m}(\mathbf{g})| \le 10 e^{-u/K} +2 \sum_{1 \le |i| \le u} \|g_{i}\| 1_{\mathbf{g} \notin B_{u-1}} + 2\sum_{|i| > u} \|g_{i}\| \cdot 1_{\mathbf{g} \notin B_{|i|-1}}.
\end{equation}
The proof is almost identical to the one in the proof of Lemma \ref{lem:squeezeL1}, so we omit it.

We now estimate the expectation of the RHS of Inequality \ref{eqn:expSqueezeL1}. By   the Tonelli theorem, it suffices to compute the expectation of each summand. Moreover,   since the RVs $\|g_{i}\|1_{\mathbf{g} \notin B_{u-1}}$ and $\|g_{i}\| \cdot 1_{\mathbf{g}\notin B_{|i|-1}}$ are integrable,   we can condition on $g_{0}, g_{j_{1}}, \ldots, g_{j_{N}}$, compute the conditional expectation and average them. 

For each $j \in \Z \setminus \{0, j_{1}, \ldots, j_{N}\}$, for each $k>0$ and for each prescribed choice $g_{0}, g_{j_{1}}, \ldots, g_{j_{N}}$, we have 
\[\begin{aligned}
\E\left[\left. \|g_{j}\| \cdot 1_{\mathbf{g} \notin B_{k}} \right| g_{0}, g_{j_{1}}, \ldots, g_{j_{N}}\right] &\le \E\left[ \|g_{j}\| \Prob \left( \cup_{i \ge k} A_{i; e^{-i/K}}^{c} \, \Big| \, g_{0}, g_{j_{1}}, \ldots, g_{j_{N}}, g_{j}\right)\right] \\
&\le \E_{\mu_{j}} \|g_{j}\| \cdot \frac{K}{1-e^{-1/K}} e^{-k/K + N} & (\because \textrm{Proposition \ref{prop:expSqueeze}})\\
&\le \frac{K}{1-e^{-1/K}} e^{-k/K + N}(1+\epsilon) \|\mu\|_{1}.
\end{aligned}
\]
The expectation of $\E\left[\left. \|g_{j}\| \cdot 1_{\mathbf{g} \notin B_{k}} \right| g_{0}, g_{j_{1}}, \ldots, g_{j_{N}}\right]$ for $g_{0} \sim \mu_{0}, g_{j_{1}} \sim \mu_{j_{1}}$, $\ldots$, $g_{j_{N}} \sim \mu_{j_{N}}$ is still bounded by the same bound $ \frac{K}{1-e^{-1/K}} e^{-k/K + N}(1+\epsilon) \|\mu\|_{1}$.

if $j \in\{j_{1}, \ldots, j_{m}\}$ and $k>0$, we have 
\[\begin{aligned}
\E\left[\left. \|g_{j}\| \cdot 1_{\mathbf{g} \notin B_{k}} \right| g_{0}, g_{j_{1}}, \ldots, g_{j_{N}}\right] &\le \E\left[ \|g_{j}\| \Prob \left( \cup_{i \ge k} A_{i; e^{-i/K}}^{c} \, \Big| \, g_{0}, g_{j_{1}}, \ldots, g_{j_{N}}, g_{j}\right)\right] \\
&\le  \|g_{j}\| \cdot \frac{K}{1-e^{-1/K}} e^{-k/K + N} & (\because \textrm{Proposition \ref{prop:expSqueeze}})\\
&\le \frac{K}{1-e^{-1/K}} e^{-k/K + N} \|g_{j}\|.
\end{aligned}
\]
The expectation of this quantity for $g_{0} \sim \mu_{0}, g_{j_{1}} \sim \mu_{j_{1}}$, $\ldots$, $g_{j_{N}} \sim \mu_{j_{N}}$ is bounded by $ \frac{K}{1-e^{-1/K}} e^{-k/K + N} \| \mu_{j}\|_{1}$.

Hence, we have \[\begin{aligned}
&\E\Big[10 e^{-u/K} + 2\sum_{1 \le |i| \le u} \|g_{i}\| \cdot 1_{\mathbf{g} \notin B_{u-1}} +  2\sum_{|i| > u} \|g_{i}\| \cdot 1_{\mathbf{g} \notin   B_{|i|-1}}\Big] \\
&\le 10 e^{-u/K} +  \frac{2K}{1-e^{-1/K}} (1+\epsilon) \|\mu\|_{1} \cdot \left( u e^{-(u-1)/K + N} + \frac{1}{1-e^{-1/K}} e^{-(u-1)/K+N}  \right) \\
&+ \frac{K}{1-e^{-1/K}} e^{-(u-1)/K+N} \sum_{l=1}^{N} \|\mu_{j_{l}}\|_{1} \\
&\le K_{2} e^{-(u-1)/K + N} \left(1 +  \|\mu_{j_{1}}\|_{1}+ \ldots +    \|\mu_{j_{N}}\|_{1}\right)
\end{aligned}
\]
for $K_{2} = 100K (1+\epsilon) e^{1/K} (2K/e) \|\mu\|_{1} (1-e^{-1/K})^{-2}$.
\end{proof}

In fact, the constant $\epsilon, K$ we obtained in Proposition \ref{prop:expSqueeze} only depends on how much weight $\mu^{\ast n}$ puts on independent contracting isometries $g, h \in G$. 

Also, we can modify $\mu_{0}, \mu_{j_{1}}, \ldots, \mu_{j_{N}}$ by constant scaling or by taking linear combination, which will render them finite measures or signed measures, respectively. After this modification, an estimate analogous to the one in Lemma \ref{lem:expSqueezeL1} holds, namely: \[
  \E_{\prod_{i} \mu_{i}} \left|R(\mathbf{g}) - R_{n, m}(\mathbf{g}) \right| < K_{2} e^{-\min(n, m)/K +N} \cdot \left(1 + \sum_{l=1}^{N} \|  \mu_{j_{l}}\|_{1}\right) \cdot \|\mu_{0}\|_{0} \cdot \prod_{l=1}^{N} \|\mu_{j_{N}}\|_{0}.
\]
We will now see that this is in fact similar to $\E_{\prod_{i} \mu_{i}} R(\mathbf{g})$ itself.

\begin{cor}\label{cor:expSqueezeL1}
Let $(X, d)$ be as in Convention \ref{conv:expSqueeze}, let $\mu$ be a strongly non-elementary probability measure on $G$ with finite  first moment and let $\epsilon, K > 0$ be as in Proposition \ref{prop:expSqueeze}. Let $K_{2}$ be as in Lemma \ref{lem:expSqueezeL1}. Then for $K_{3} = K_{2} e^{1/K}$, the following holds.

Let $\{j_{1}, \ldots, j_{N}\}$ be $N$ integers, and let $\mu_{0}, \mu_{j_{1}}, \ldots, \mu_{j_{N}}$ be signed measures with finite $\| \cdot \|_{0, 1}$-norm. Suppose further that $\mu_{j_{1}}$ is balanced, i.e., $\sum_{g \in G} \mu_{j_{1}}(g) = 0$. For $i \in \Z \setminus \{0, j_{1}, \ldots, j_{N}\}$, let $\mu_{i}$ be a probability   measure such that $\|\mu_{i} - \mu\|_{0, 1} < \epsilon$. Then we have \[
\Big |\E_{\prod_{i} \mu_{i}} R(\mathbf{g}) \Big| < K_{3} e^{-|j_{1}|/K +N} \cdot \left(1 + \sum_{l=1}^{N} \|   \mu_{j_{l}}\|_{1}\right)\cdot \|\mu_{0}\|_{0} \cdot \prod_{l=1}^{N} \|\mu_{j_{N}}\|_{0}.
\]
\end{cor}

\begin{proof}
Under the assumption, we claim that $\E_{\prod_{i}} R_{|j_{1}|- 1, |j_{1}| - 1}(\mathbf{g}) = 0$. To see this, we note that $R_{|j_{1}|-1,   |j_{1}| - 1} (\mathbf{g})$ does not depend on the value of $g_{j_{1}}$.  Let $E := \E_{g_{i} \sim \mu_{i} : |i| < |j_{1}|} R_{|j_{1}|- 1, |j_{1}| - 1}(\mathbf{g})$. Then we observe  \[
\E_{\prod_{i}} R_{|j_{1}|- 1, |j_{1}| - 1}(\mathbf{g}) = \E_{g_{j_{1}} \sim \mu_{j_{1}}} \E_{g_{i} \sim \mu_{i} : |i| < |j_{1}|} R_{|j_{1}|- 1, |j_{1}| - 1}(\mathbf{g}) = \sum_{g} \mu_{j_{1}}(g) \cdot E = 0.
\]
Given this, Lemma \ref{lem:expSqueezeL1} for $n=m=|j_{1}|-1$ implies that \[
\left| \E_{\prod_{i} \mu_{i}} R(\mathbf{g})\right| < K_{2} e^{- (|j_{1}| - 1)/K +N} \cdot \left(1 + \sum_{l=1}^{N}   \|\mu_{j_l}\|_{1}\right) \cdot \|\mu_{0}\|_{0} \cdot \prod_{l=1}^{N} \|\mu_{j_{N}}\|_{0}. \qedhere
\]
\end{proof}

Let us observe a simple lemma:

\begin{lem}\label{lem:seriesDiff}
Let $\{F_{i}(t)\}_{t \in [-\epsilon, \epsilon]}$ be a family of functions for each $i \in \Z$ and let $(c_{i})_{i \in \Z}$ be a summable sequence of positive numbers. Suppose that   $F_{i}(t)$ is differentiable at $t=0$ for each $i$ and \[
|F_{i}'(0)| \le c_{i}\quad  (\forall i).
\]
Suppose also that there exists a function $f(t)$ with $\lim_{t \rightarrow 0} f(t)/t = 0$ such that \[
|F_{i}(t) - F_{i}(0) - tF_{i}'(0)| \le c_{i} f(t) \quad (\forall i, \forall t).
\]

Then $F(t) := \sum_{i \in \Z} F_{i}(t)$ is differentiable at $t =0$, with $F'(0) = \sum_{i \in \Z} F_{i}'(0)$.
\end{lem}

\begin{proof}
For any $N$ and $t$, we have \[
\left| \frac{F(t) - F(0)}{t} - \frac{\sum_{|i| \le N} F_{i}(t) - \sum_{|i| \le N} F_{i}(0)}{t} \right| \le \sum_{|i| > N} \frac{|F_{i}(t) - F_{i}(0)|}{t} \le \sum_{|i| \ge N} |F_{i}'(0)| +  \frac{f(t)}{t} \sum_{|i| \ge N} c_{i}.
\]
We   let $t \rightarrow 0$ and conclude that $\sum_{|i| \le N} F_{i}'(0)$, $\limsup_{t\rightarrow 0} \frac{F(t) - F(0)}{t}$ and $\liminf_{t \rightarrow 0} \frac{F(t) - F(0)}{t}$ differ by at most $2 \cdot \sum_{|i| \ge N}   c_{i}$.   Now take $N \rightarrow \infty$  to conclude.
\end{proof}

We now prove Theorem \ref{thm:driftDiffSqueezeSecond}.

\begin{proof}

Let $\mu$ be a   non-elementary probability measure, and let $\epsilon, K>0$ be as in Proposition \ref{prop:expSqueeze} for $\mu$. Let us now consider a family of probability   measures $\{\mu_{t}^{(0)}\}_{t \in [-\epsilon, \epsilon]}$, families of   balanced signed measures $\{\mu_{t}^{(i)} : i =1, 2, \ldots\}_{t \in [-\epsilon, \epsilon]}$ and   a family of $o(t)$-functions $\{f_{i}(t) : i=0,1, \ldots\}_{t \in [-\epsilon, \epsilon]}$ such that \[
\begin{aligned}
\|\mu_{t}^{(i)} - \mu^{(i)} - t\mu^{(i+1)} \|_{0, 1} \le f_{i}(t). \quad (\forall t, \forall i).
\end{aligned}
\]
For convenience, let us assume $f_{1} \le f_{2} \le \ldots$; we can just take the maximum of the first $i$ functions if needed. Also, let $M_{1} \le M_{2} \le \ldots$ be such that \[
M_{i} \ge \|\mu^{(j)}(t)\|_{0, 1}, \quad M_{i} \ge f_{j}(t)/t
\]
for each $j \le i$ and each $t \in [-\epsilon, \epsilon]$.

Let us now consider the set \[
\mathcal{I} := \{ (a_{i})_{i \in \Z} : \textrm{$a_{i} \in \Z_{>0}$ for each $i$ and $a_{i} = 0$ for all but finitely many entries}\}.
\]
For $\mathbf{a} = (a_{i})_{i \in \Z}$, we let $\|\mathbf{a}\| = \max\{|i| : a_{i} \neq 0\}$, $\sigma(\mathbf{a}) = \sum_{i} a_{i}$ and $\supp \mathbf{a} = \{ i : a_{i} > 0\}$.

For this proof, we will use the notation $\E_{\mathbf{a}, t}$ and $\E_{\mathbf{a}, t; A}$ for the expectation of RVs with respect to a specific product measure. Namely, $\E_{\mathbf{a}, t} F(\mathbf{g})$   denotes the expectation of the RV $F(\mathbf{g})$  when $g_{i}$'s are all independent and $g_{i} \sim \mu_{t}^{(a_{i})}$.   We know for example that $R(\mathbf{g})$ is well-defined almost surely with respect to any $\Prob_{\mathbf{a}, t}(\cdot)$ so we can discuss $\E_{\mathbf{a}, t} R(\mathbf{g})$.

To differentiate $\E_{\mathbf{a}, t} R(\mathbf{g})$, we need another notation. For each finite subset $A$ of $\Z$, we define $\Prob_{\mathbf{a}, t; A}$ to denote that $g_{i}$'s are independently distributed according to 
\[
g_{i} \sim \left\{\begin{array}{cc} \mu_{t}^{(a_{i})} &  \textrm{if}\,\,  i \in A\\
 \mu_{0}^{(a_{i})} & \textrm{if}\,\, i \notin A.
\end{array}\right.
\]
Finally, if we first distribute $g_{i}$'s in the above rules except for some steps $g_{j}$, we write $\Prob_{\mathbf{a}, t; A; g_{j} \sim \textrm{distribution}}$.

Here, $\E_{\mathbf{a}, t; A}R(\mathbf{g})$ is meant to be a finite approximation of $\E_{\mathbf{a}, t} R(\mathbf{g})$. More precisely, we claim:

\begin{claim}\label{claim:fAA}
For each $\mathbf{a}  \in \mathcal{I}$ and $t \in [-\epsilon, \epsilon]$, we have \[
\E_{\mathbf{a}, t}R(\mathbf{g})= \lim_{N \rightarrow \infty} \E_{\mathbf{a}, t; [-N, N]}R(\mathbf{g})
\]
\end{claim}

\begin{proof}
Take an arbitrary $\epsilon'>0$. Then by Lemma \ref{lem:expSqueezeL1}, there exists $n$ such that \[
\E_{\mathbf{a}, t} |R(\mathbf{g}) - R_{n, n}(\mathbf{g})| < \epsilon',\,\,
\E_{\mathbf{a}, t; [-N, N]} |R(\mathbf{g}) - R_{n, n}(\mathbf{g})| < \epsilon' \quad (\forall N).
\]
We now take an arbitrary $N > n$. Then $R_{n, n}(\mathbf{g})$ has the same distribution for both $\Prob_{\mathbf{a}, t}$ and $\Prob_{\mathbf{a}, t; [-N, N]}$. This implies that \[
|\E_{\mathbf{a}, t}R(\mathbf{g})-  \E_{\mathbf{a}, t; [-N, N]}R(\mathbf{g})| <   2\epsilon' \quad (\forall N > n).
\]
Since such a threshold $n(\epsilon')$ exists for each $\epsilon'$, we conclude the claim.
\end{proof}

From this claim, we observe that \[\begin{aligned}
&\E_{\mathbf{a}, t}R(\mathbf{g}) =  \E_{\mathbf{a}, t; \{0\}} R(\mathbf{g})\\
&+ \sum_{N=0}^{\infty} \big(\E_{\mathbf{a}, t; [-N, N]; g_{N+1} \sim \mu_{t}^{(a_{N+1})} - \mu_{0}^{(a_{N+1})} } R(\mathbf{g}) +  \E_{\mathbf{a}, t; [-N, N+1]; g_{-(N+1)} \sim \mu_{t}^{(a_{-(N+1)})}  - \mu_{0}^{(a_{-(N+1)})} } R(\mathbf{g})\big).
\end{aligned}
\]
Next, note that \[\begin{aligned}
\E_{\mathbf{a}, t; [-N, N]; g_{N+1} \sim\mu_{t}^{(a_{N+1})} - \mu_{0}^{(a_{N+1})}} R(\mathbf{g}) &= \E_{\mathbf{a}, 0; g_{N+1} \sim \mu_{N+1}^{(a_{N+1})}(t) - \mu_{N+1}^{(a_{N+1})}(t)} R(\mathbf{g}) \\
&+ \sum_{k=-N}^{N} \E_{\mathbf{a}, t; [-N, k-1]; g_{N+1} \sim \mu_{t}^{(a_{N+1})} - \mu_{0}^{(a_{N+1})}, g_{k} \sim \mu_{t}^{(a_{k})}- \mu_{0}^{(a_{k})}} R(\mathbf{g}),
\end{aligned}
\]
and similarly $\E_{\mathbf{a}, t; [-N, N+1]; g_{-(N+1)} \sim \mu_{t}^{(a_{-(N+1)})} - \mu_{0}^{(a_{-(N+1)})} } R(\mathbf{g})$ is equal to $\E_{\mathbf{a}, 0; g_{-(N+1)} \sim \mu_{t}^{(a_{-(N+1)})}- \mu_{0}^{(a_{-(N+1)})} } R(\mathbf{g})$ plus $2N+1$ expectations of the form $\E_{\mathbf{a}; t; [k+1, N+1]; g_{-N-1} \sim \mu_{-N-1}^{(a_{-N-1})}(t) - \mu_{-N-1}^{(a_{-N-1})}, g_{k} \sim g_{k} \sim \mu_{k}^{(a_{k})} (t) - \mu_{k}^{(a_{k})}} R(\mathbf{g})$.

Hence, if we define \[\begin{aligned}
S_{k}(t) &:= \E_{\mathbf{a}, 0; g_{k} \sim \mu_{k}^{(a_{k})}(t) - \mu_{k}^{(a_{k})}} R(\mathbf{g}, \\
S_{k, l}(t) &:= \E_{\mathbf{a}, t; [\min(k, l) +1, \max(k, l)+1] ; g_{k} \sim \mu_{k}^{(a_{k})}(t) - \mu_{k}^{(a_{l})}, g_{l} \sim \mu_{l}^{(a_{l})}(t) - \mu_{l}^{(a_{l})}} R(\mathbf{g})
\end{aligned}
\]
then we have \[\begin{aligned}
\E_{\mathbf{a}, t} R(\mathbf{g}) &= \E_{\mathbf{a}, t; \{0\}}R(\mathbf{g}) + \sum_{k \in \Z} S_{k}(t) + \sum_{k \neq l}  S_{k, l}(t).
\end{aligned}
\]In fact, the above is the rearrangement of the implied summation, and it remains to verify the summands are absolutely summable. To check it, first note that $\mu_{t}^{(i)} - \mu_{0}^{(i)}$ is always balanced for each $t$ and for each $i$. For $i = 0$ this is because both $\mu_{t}^{(i)}$ and $\mu_{0}^{(i)}$ have total weight $1$. For $i > 0$ this is because both signed measures are balanced. We can therefore apply Corollary \ref{cor:expSqueezeL1} to observe 
 \[\begin{aligned}
S_{k}(t)&= \E_{g_{k} \sim \mu_{t}^{(a_{k})} - \mu_{0}^{(a_{k})}}  \E_{ g_{i} \sim \mu^{(a_{i})} : i \in \supp \mathbf{a} \setminus \{k\} } \cdot \E_{\mu^{\Z}} \big[R(\mathbf{g}) \, \big| \, g_{i} : i \in k \cup \supp \mathbf{a}\big] \\
&\le \big(1 + \|\mu_{t}^{(a_{k})} - \mu_{0}^{(a_{k})}\|_{1} + \sum_{i \in \supp \mathbf{a} \setminus \{k\}} \|\mu^{(a_{i})}\|_{1}\big) \cdot \|\mu_{t}^{(a_{k})}- \mu_{0}^{(a_{k})}\|_{0} \\
&\,\, \cdot  \prod_{i \in \supp \mathbf{a} \setminus \{k\}} \|\mu^{(a_{i})}\|_{0} \cdot K_{3} e^{-\| \mathbf{a}^{(k)}\|/K + \sigma(\mathbf{a})+1} \\
&\le \Big(1+ \|\mu_{t}^{(a_{k})}\|_{1} + \|\mu_{0}^{(a_{k})}\| + \sum_{i \in \supp \mathbf{a} \setminus \{k\}} M_{a_{i}}\Big ) \cdot \prod_{i \in \supp \mathbf{a} \setminus \{k\}} M_{a_{i}} \\
&\,\, \cdot  K_{3} e^{-\|\mathbf{a}^{(k)}\| /K + \sigma(\mathbf{a}) + 1} \cdot  (f_{a_{k}}(t)/t + \|\mu^{(a_{k} + 1)}\|_{0, 1}) \cdot t\\
&\le  (\sigma(\mathbf{a}) + 2) M_{\sigma(\mathbf{a})+1} \cdot M_{\sigma(\mathbf{a})}^{\sigma(\mathbf{a})} \cdot K_{3} e^{-\| \mathbf{a}^{(k)}\|/K + \sigma(\mathbf{a})+1} \cdot 2M_{\sigma(\mathbf{a})} t:= C_{k} t. 
\end{aligned}
\]
Here, $\mathbf{a}^{(k)}$ denotes the sequence made from $\mathbf{a}$ after increasing $a_{k}$ by $1$.
A similar computation shows that $ \E_{\mathbf{a}, t; A_{k, l} ; g_{k} \sim \mu_{k}^{(a_{k})}(t) - \mu_{k}^{(a_{l})}, g_{l} \sim \mu_{l}^{(a_{l})}(t) - \mu_{l}^{(a_{l})}} R(\mathbf{g})$ is bounded by \begin{equation}\label{eqn:SklExp}\begin{aligned}
S_{k, l}(t)&\le (\sigma(\mathbf{a}) + 4) M_{\sigma(\mathbf{a})} \cdot M_{\sigma(\mathbf{a})}^{\sigma(\mathbf{a})} \cdot K_{3} e^{-\|\mathbf{a}^{(k, l)}\|/K + \sigma(\mathbf{a}) + 2} (2M_{\sigma(\mathbf{a})})^{2} t^{2} := C_{k, l} t^{2}.
\end{aligned}
\end{equation}
Note that $C_{k}$ and $C_{k, l}$ are exponentially decreasing in $k$ and $k$, $l$, respectively. Hence, they are absolutely summable. Also, Inequality \ref{eqn:SklExp} tells us that  $S_{k, l}(t)' = 0$ for each $k, l$. For each $S_{k}$, we observe \[\begin{aligned}
&\big|S_{k}(t)- t \E_{\mathbf{a}, 0; g_{k} \sim \mu_{k}^{(a_{k}+1)}}R(\mathbf{g})\big| \\
&= \Big| \E_{ g_{k} \sim \mu_{k}^{(a_{k})}(t) - \mu_{k}^{(a_{k})} - t\mu_{k}^{(a_{k}+1)}}  \E_{ g_{i} \sim \mu^{(a_{i})} : a_{i} \neq 0} \cdot \E_{\mu^{\Z}} \big[R(\mathbf{g}) \, \big| \, g_{i} : i \in k \cup \supp \mathbf{a}\big] \Big|\\
&\le (1 + f_{a_{k}+1}(t) + \sum_{i \in \supp \mathbf{a} \setminus \{k\}} M_{a_{i}}) \cdot \prod_{i \in \supp \mathbf{a} \setminus \{k\}} M_{a_{i}} \\
& \,\, \cdot K_{3} e^{-\|\mathbf{a}^{(k)}\| / K + \sigma(\mathbf{a}) + 1} \cdot f_{a_{k} + 1}(t)\\
&\le f_{\sigma(\mathbf{a}) + 1}(t) \cdot C_{k}.
\end{aligned}
\]
We now apply Lemma \ref{lem:seriesDiff} to conclude that:

\begin{claim}\label{claim:derivFinal}
For any $\mathbf{a} \in \mathcal{I}$, we have 
\begin{equation}\label{eqn:linExpDeliResult}
\frac{d}{dt} \Big|_{t=0} \E_{\mathbf{a}, t} R(\mathbf{g}) = \sum_{k \in \Z} (S_{k}'(t)) = \sum_{k \in \Z} \E_{\mathbf{a}^{(k)}, 0} R(\mathbf{g}),
\end{equation}
Furthermore, we have a  uniform bound for all $t \in [-\epsilon, \epsilon]$: \begin{equation}\label{eqn:linExpDeli}
\begin{aligned}
\big| \E_{\mathbf{a}, t} R(\mathbf{g})  - \E_{\mathbf{a}, 0} R(\mathbf{g})  - t \sum_{k \in \Z} \E_{\mathbf{a}^{(k)}, 0} R(\mathbf{g})
 \big| 
 &\le \sum_{k} C_{k} f_{\sigma(a) + 1} (t) + \sum_{k, l} C_{k, l} t^{2},
 \end{aligned}
\end{equation}
where $C_{k} \lesssim_{\|\mathbf{a}\|} e^{-\sigma(\mathbf{a})^{(k)}}$, and $C_{k, l} \lesssim_{\|\mathbf{a}\|}  e^{-\sigma(\mathbf{a})^{(k, l)}}$. 
\end{claim}

It is now time to sum up the discussion. Recall that $\sigma_{1}(\mu, \eta) = \E_{\mu^{(1)} }(g) - \E_{\mathbf{a}}R(\mathbf{g})$ for $\mathbf{a} = \delta_{0}$, the Dirac function at $0$. Here, given the setting of the theorem, it is straightforward to see that $\frac{d}{dt}|_{t=0} \E_{\mu_{t}^{(k)}} (g) = \E_{\mu_{t}^{(k+1)}}(g)$. So it remains to examine the differentiability of $\E_{\mathbf{a}}R(\mathbf{g})$. We observe that for each $n=0, 1, \ldots$, \[
\frac{d}{dt} \Big|_{t=0} \sum_{k_{1}, \ldots, k_{n} \in \Z} \E_{\mathbf{a}^{(k_{1}, \ldots, k_{n})}, t} R(\mathbf{g})  =  \sum_{k_{1}, \ldots, k_{n}, k_{n+1} \in \Z} \E_{\mathbf{a}^{(k_{1}, \ldots, k_{n+1})}, t} R(\mathbf{g}).
\]
First, the summand-wise derivative is given in Display \ref{eqn:linExpDeliResult}, and the uniform linear approximation is given in Display \ref{eqn:linExpDeli}. Now Lemma \ref{lem:seriesDiff} implies the claim. This ends the proof.
\end{proof}

\section{Continuity of the entropy}\label{section:entropy}

In this section, we prove Theorem \ref{thm:entropyContinuity}. Our strategy is as follows. Let $\mu$ be a non-elementary probability measure $\mu$ on $G$ with finite  time-one entropy, and suppose that   $\|\mu_{i} - \mu\|_{0} \rightarrow 0$ and $|H(\mu_{i}) - H(\mu)| \rightarrow 0$ holds. Let $\epsilon>0$ be arbitrary.

\begin{enumerate}
\item For a given non-elementary probability measure $\mu$ on $G$ with finite (time-one) entropy, we show that there exists $ N= N(\mu)$ such that $\frac{1}{n} H(\mu^{\ast n})$ and $h(\mu)$ differ by at most $\epsilon$ for all $n \ge N$. 
\item We also show that there exists a small $\epsilon_{meas} >0$ (depending on the choice of $\mu$) such that, whenever \[
\|\mu_{i} - \mu\|_{0} < \epsilon_{meas}\,\,\textrm{and}\,\, |H(\mu_{i}) - H(\mu)| < \epsilon_{meas},
\]
the same estimate holds for $\mu_{i}$,  i.e.,  $\frac{1}{n} H(\mu_{i}^{\ast n})$ and  $h(\mu_{i})$ differ by at most $\epsilon$ for all $n \ge  N$.
\item Finally, we observe that $H(\mu_{i}^{\ast n}) \rightarrow H(\mu^{\ast n})$ for $n=N$. 
\end{enumerate}
This argument will show that $\liminf_{i} h(\mu_{i})$ and $\limsup_{i} h(\mu_{i})$ will differ from $h(\mu)$ by at most $2\epsilon$. Since $\epsilon$ is arbitrary, we conclude that $\lim_{i} h(\mu_{i}) = h(\mu)$.

The third step of the above is classical (cf. \cite[Lemma 3.5]{amir2013amenability}). Hence, the key is to pick $N(\mu)$ that works for not only $\mu$ but all $\mu_{i}$ that are sufficiently close to $\mu$.

\subsection{Schottky sets for WPD elements} \label{subsection:SchottkyWPD}

In this subsection, we construct a Schottky set $S$ that exhibits  the WPD property.

\begin{definition}\label{dfn:SchottkyWPD}
Let $(X, G)$ be as in Convention \ref{conv:main}. Let $S$ be a $K_{0}$-Schottky set associated with the auxiliary constants $D_{0}, E_{0}$. Let $F$ be a finite subset of $G$. We say that $S$ is $F$-WPD Schottky set if the following holds: if $w \in G$ and $s, s' \in S$ are such that the beginning and the ending points of $\Gamma(s)$ and $w \Gamma(s')$ are pairwise $100E_{0}$-close, then $w \in F$ and $\Pi(s)^{-1} w \Pi(s') \in F$.
\end{definition}

\begin{lem}\label{lem:SchottkyWPD}
Let $(X, G)$ be as in Convention \ref{conv:main} and let $g$ be a BGIP element in $G$ with WPD property. Let $A \subseteq G$ be a subset such that the semigroup $\llangle A \rrangle$ generated by $A$ contains two independent BGIP elements and $g$. Then for each $N$, there exists $m, K_{0}>0$, a finite set $F\subseteq G$ and a long enough $K_{0}$-Schottky set $S \subseteq A^{m}$ with cardinality $N$ and with $F$-WPD property.
\end{lem}

\begin{proof}
Given $N$, Proposition \ref{prop:SchottkyPA} tells us that there exists $K_{0}, m>0$ and a finite set $S'' \subseteq A^{m'}$ of cardinality $N$ such that \[
S_{k} := \{(s, \alpha^{(k)}, s) : s \in S''\}
\]
is a $K_{0}$-Schottky set for each $k$. We now take $M=  100E_{0} +2 \max_{s \in S''} \|s\|)$. The WPD property of $g$ tells us that \[
 Stab_{M} (o, g^{n} o) = \{h \in G : d(o, ho) \le M, d(g^{n} o, h g^{n} o) \le M\}
\]
is finite for all large enough $n$.   Now let \[\begin{aligned}
F_{1} &:= S' \cdot Stab_{M}(o, g^{n} o) \cdot S'^{-1} =  \{\Pi(s) t \Pi(u)^{-1} : s, u \in S', t \in Stab_{M}(o, g^{n} o)\}, \\
F_{2} &:= S'^{-1} \cdot g^{-n} \cdot Stab_{M}(o, g^{n} o) \cdot g^{n} S'.
\end{aligned}
\]
These are finite sets. We now claim that for a large enough $n$, $S_{n}$ is the  desired long enough $K_{0}$-Schottky set that has the $F$-WPD property   for $F = F_{1} \cup F_{2}$.

To check this, for $s, u\in S$ and $w \in G$ suppose that the beginning and the ending points of \[
\Gamma(s, \alpha^{(n)}, s), \,w \cdot \Gamma(u, \alpha^{(n)}, u)
\]
are each $100E_{0}$-close. In particular, we have \[
d(o, wo) < 100E_{0},\quad  d(\Pi(s) o, w \Pi(u) o) < 100E_{0} + 2 \max_{t \in S''} \|t\| = M.
\]
For a similar reason, we have $\Pi(s) g^{n} o$ and $w \Pi(u) g^{n}$ are $M$-close to each other. We now use the WPD property of $g^{n}$ to conclude that $\Pi(s)^{-1} w \Pi(u) \in Stab_{M}(o, g^{n} o)$, or in other words, $w \in F$.

Likewise, we have \[
\big( \Pi(s) g^{n} \Pi(s) \big)^{-1} \cdot w \cdot \big( \Pi(u) g^{n} \Pi(u)\big) = \Pi(s)^{-1} g^{-n} \cdot \Pi(s)^{-1} w \Pi(u) \cdot g^{n} \Pi(u) \in S'^{-1} \cdot g^{-n} \cdot Stab_{M}(o, g^{n} o) \cdot g^{n} \cdot S', 
\]
so $\big( \Pi(s) g^{n} \Pi(s) \big)^{-1} \cdot w \cdot \big( \Pi(u) g^{n} \Pi(u)\big)$ belongs to $F$ as well.
\end{proof}

\subsection{Sublinear growth of entropy of displacement} \label{subsection:sublinearEnt}

In this section, we prove the following.

\begin{prop}\label{prop:sublinearEnt}
Let $(G, X, o)$ be as in Convention \ref{conv:main}. Let $K_{0}, M_{0}>0$ and let $S$ be a large and long enough $K_{0}$-Schottky set $S \subseteq G^{M_{0}}$. Suppose that: \begin{enumerate} 
\item $H(\mu)< \infty$, 
\item $\mu^{M_{0}}(s)  > 0$ for each $s \in S$. 
\end{enumerate}
Then for each $\eta > 0$, there   exist $\epsilon' = \epsilon(\eta, \mu)$ and $N = N(\eta, \mu)$ such that the following holds: for    every probability measure $\mu'$ on $G$ with $\|\mu - \mu'\|_{0}< \epsilon'$ and $|H(\mu) - H(\mu')| < \epsilon'$ and for    every $g \in G$, 
\[\begin{aligned}
H_{g_{1} = g, (g_{2}, \ldots, g_{n}) \sim \mu'^{n-1}}  \big( \lfloor \|Z_{n}\| \rfloor =\lfloor  d(o, g_{1} \cdots g_{n} o)\rfloor \big) \le \eta n,\\
H_{(g_{1}, \ldots, g_{n-1}) \sim \mu'^{n-1}, g_{n} = g}  \big( \lfloor \|Z_{n}\| \rfloor =\lfloor  d(o, g_{1} \cdots g_{n} o)\rfloor \big) \le \eta n\\
\end{aligned}
\]
holds for each $n > N$.
\end{prop}

This proposition is proved in \cite{chawla2022the-poisson} when the underlying space is Gromov hyperbolic, by means of pivoting technique. We explain Chawla-Forghani-Frisch-Tiozzo's argument here for the readers' convenience.

\begin{proof}
We    prove the first claim only; the argument for the second claim is symmetric.

   Note that the Schottky set $S$ and the constants $K_{0}, E_{0}$ are fixed    throughout. We     recall  a lemma:

\begin{lem}[{\cite[Lemma 2.4]{chawla2022the-poisson}}]\label{lem:chawlaEnt}
For each $M \in \Z_{>0}$ and $\epsilon>0$, there exists $\delta = \delta(M, \epsilon)>0$ such that the following holds.

Let $A$ be a countable set    and let $Z : (\Omega, \Prob) \rightarrow A$ be a random variable.    Suppose that there exists a subset $U$ of $A$ with $\#U = A$ such that \[
\sum_{ x \notin U} \left( -\log (\Prob(Z=x)) \varphi(\Prob(Z = x))\right) < \epsilon/2.
\] Then for any measurable subset $E \subseteq \Omega$ with $\Prob(E)<\delta$, the modified RV \[
Y(\w):= \left\{\begin{array}{cc} Z(\w) & \textrm{ if $   \omega\in E$}\\ \ast & \textrm{otherwise} \end{array}\right.
\]
has entropy smaller than $\epsilon$.
\end{lem}

Let $\mu$ be a probability measure satisfying the assumption. Since $H(\mu) = \sum_{g \in G} -\mu(g) \log \mu(g) <+\infty$, there exists a finite set $U \subseteq G$ such that $\sum_{g \notin U} -\mu(g) \log \mu(g)$ is $\eta/10$-close to $H(\mu)$. Let $\delta = \delta(\#U, \eta/3)$ be as in Lemma \ref{lem:chawlaEnt}.

Let $m = \min_{s \in S} \mu^{M_{0}}(s) >0$. Let $\kappa = \kappa(\epsilon', M_{0})$ as in Proposition \ref{prop:gouezelRW1Ori} and let $\alpha \in \Z_{>1}$ be a large enough integer such that \[
\frac{1}{\kappa} e^{-\kappa(\alpha-1)} < \delta/10.
\]
We now take large enough finite subset $F$ of $G$ such that $\mu(F) \ge 1 - \frac{1}{100 (\alpha+1)} \delta$.

 We can take small enough $\epsilon' < \eta/10$ such that if a probability measure $\mu'$ satisfies \begin{equation}\label{eqn:muEnoughClose}
 \| \mu' - \mu\|_{0} < \epsilon'\,\,\textrm{and}\,\,|H(\mu') - H(\mu)| < \epsilon'
 \end{equation}
 then the following hold:
 \begin{enumerate}
\item $H(\mu') < H(\mu) + \eta/10$;
\item $\sum_{g \in U} -\mu'(g) \log \mu'(g) > H(\mu) - \eta/5$, and hence
\item $\sum_{g \notin U} -\mu'(g) \log \mu'(g) < \eta/3$;
\item $\mu'^{M_{0}}(s) > m/2$ for each $s \in S$, and
\item $\mu' (F) > 1 - \frac{1}{50 \alpha} \delta$.
\end{enumerate}

We will now take $N(\eta, \mu)$ that works for $\mu'$ satisfying Condition \ref{eqn:muEnoughClose}. But the actual value of $N$ will only depend on    $\eta$ and $\alpha$.

Let $n \in \Z_{>0}$ be given. We construct a probability space $\Omega_{n}$ and its partition $\mathcal{Q}_{n}$ into pivotal equivalence classes, aiming at RVs $g_{i}$'s with \[
(g_{1}, g_{2}, \ldots, g_{n}) \sim 1_{g_{1} = g} \times \mu' \times \cdots \times \mu'.
\]
   This is done using Proposition \ref{prop:gouezelRW1Ori}. Then the set of pivotal times $\mathcal{P}_{n}(\w) = \{j(1) < \ldots < j(\#\mathcal{P}_{n})\} \subseteq \{1, \ldots, n\}$ is a $2^{\{1, \ldots, n\}}$-valued RV on $\Omega_{n}$. We then divide the index set $\{1, \ldots, n\}$ into $\lceil n/\alpha \rceil$ intervals \[
I_{1}, I_{2}, \ldots, I_{\lceil n/\alpha \rceil}
\]
of lengths $\alpha$ or $\alpha - 1$. Now, for each $k=1, \ldots, \lceil n/\alpha \rceil$, we   say that $k$ is \emph{good} if: \begin{enumerate}
\item all elements in $\{ g_{i} : i \in I_{k-1} \cup I_{k} \cup I_{k+1}\}$ have norm less than $L$, i.e., $d(o, g_{i} o) < L$ for $i \in  I_{k-1} \cup I_{k} \cup I_{k+1}$ (by definition, $k=0$ and $k=\lceil n/\alpha \rceil$ are not good);
\item there exists at least one pivotal time in $I_{k-1}$ and at least one pivotal time $I_{k+1}$, i.e., $\mathcal{P}_{n}(\w) \cap I_{k-1} \neq \emptyset \neq \mathcal{P}_{n}(\w) \cap I_{k+1}$.
\end{enumerate}
For $1 \le k < \lceil n /\alpha \rceil$, $\{\w : \textrm{$k$ is not good}\}$ has probability less than $\delta$. To see this,   note that for each $i$, $\|g_{i}\|$ is greater than $L$ for probability at most $\frac{1}{50 \alpha} \delta$. Since there are at most   $3\alpha$ integers in 3 consecutive intervals, the probability that $\|g_{i}\| > L$ for at least one $i \in I_{k-1} \cup I_{k} \cup I_{k}$ is still less than $\delta/15$. Next, $I_{k-1}$ has no pivotal time for probability at most $\frac{1}{\kappa} e^{-\kappa (\alpha-1)} < \delta/10$, and similar estimate holds for $I_{k+1}$. Hence, the probability that both intervals do not have any pivotal time is less than $\delta/5$.

Now for each $i=1, \ldots, n$, we define \[
h_{i} := \left\{ \begin{array}{cc} g_{i} &  \textrm{  if the index $k$ for which $i \in I_{k}$ is not good} \\ \ast & \textrm{  if the index $k$ for which $i \in I_{k}$ is   good}.\end{array}\right.
\]
Since $\Prob(\textrm{the index $k$ for which $i \in I_{k}$ is not good})$ is less than $\delta$, Lemma \ref{lem:chawlaEnt} tells us that $H(h_{i}) < \eta/3$ for each $i$. Considering this, the entropy of the sequence $\mathbf{h} = (h_{1}, h_{2}, \ldots, h_{n})$ is at most $\eta n/3$.

Now we consider the conditional entropy of $\lfloor \|Z_{n}(\w)\|\rfloor $ conditioned on the values of $\mathbf{h}$. It is clear that $\mathbf{h} = (h_{1}, \ldots, h_{n})$ is an alternation of some sequences in $G$ and some sequences of $\ast$. Let $(h_{a(k)}, \ldots, h_{b(k)})$ be the $k$-th sequence   in $G$ that appears consecutively in $(h_{1}, \ldots, h_{n})$, and we define \[
S(\mathbf{h}) := \left\lfloor \sum_{k} \|h_{a(k)} \cdots h_{b(k)}\| \right\rfloor = \left\lfloor \sum_{k} d(Z_{a(k)-1} o, Z_{b(k)} o) \right\rfloor.
\]
We claim that: \[
S(\mathbf{h}) - \lfloor \|Z_{n} \|\rfloor \in \left\{ l \in \Z:  |l| \le    2L\alpha T + 2E_{0} T + Ln+2\right\}.
\]
If the entire $\mathbf{h}$ is recorded with elements of $G$ (not $\ast$), then the claim is clear. If not, let $T$ be the number of consecutive appearances of elements of $G$ (so that there are $T-1$ consecutive appearances of $\ast$). The indices for the $k$-th consecutive appearance of $\ast$'s are $(b(k)+1, \ldots, a(k+1) - 1)$. This is a union of intervals with good index, say $I_{c(k)},I_{c(k)+1}, \ldots, I_{d(k)}$. Because $c(k)$ is a good index, $\|g_{i} \| < L$ holds for all $i \in I_{c(k)-1}$. ($\ast$) Moreover, $I_{c(k)-1}$ contains a pivotal time; take the largest one and label it by $j''(k)$. By $(\ast)$ and by the fact that each    interval is no longer than $\alpha$, we conclude: \begin{equation}\label{eqn:goodPair1}
d(Z_{b(k)}, Z_{j''(k)}) < L\alpha.
\end{equation} For a similar reason, there exists a pivotal time $j'(k)$ in $I_{d(k)+1}$ (we take the smallest one), and we conclude \begin{equation}\label{eqn:goodPair2}
d(Z_{a(k+1)-1}, Z_{j'(k)}) < L\alpha.
\end{equation} Finally, since $c(k), c(k)+1, \ldots, d(k)$ are all good, $\|g_{i}\|< L$ holds for all $i \in I_{c(k)-1} \cup \ldots \cup I_{d(k)+1}$. This implies that \begin{equation}\label{eqn:goodPair3}
d(Z_{j''(k)}, Z_{j'(k)}) \le L(j'(k) - j''(k)).
\end{equation}
Now, note that $j'(k) \le j''(k+1)$; indeed, we have $d(k)+1 \le c(k)-1$ as the $k$-th and the $k+1$-th consecutive appearance of $\ast$'s are separated by some interval with bad index. This implies that \[
(o, \mathbf{Y}_{j''(1)}, \mathbf{Y}_{j'(1)}, \ldots, \mathbf{Y}_{j''(k)}, \mathbf{Y}_{j'(k)}, \ldots, Z_{n} o)
\]
is a $D_{0}$-semi-aligned sequence if we remove the duplicates. By Proposition \ref{prop:BGIPWitness}(1), we have \[
\left| d(o, Z_{j''(1)}) + \sum_{k=1}^{T-1} d(Z_{j''(k)}, Z_{j'(k)}) + \sum_{k=1}^{T-2} d(Z_{j'(k)}, Z_{j''(k+1)}) + d(Z_{j''(T-1)}, Z_{n}) - d(o, Z_{n} o) \right| \le 2E_{0} T.
\]
By plugging in Inequality \ref{eqn:goodPair1}, \ref{eqn:goodPair2} and \ref{eqn:goodPair3}, we conclude that 
\[
| \|Z_{n}\| - S(\mathbf{h}) | \le 2E_{0} T + \sum_{k} L ( j'(k) - j''(k)) +2L\alpha T \le 2E_{0} T + L n + 2L\alpha T.
\]
  We conclude the claim by taking the integer part.

Now, any random variable whose range is $k$ has entropy at most $\log k$. Hence, we have \[
H\big( \lfloor \| Z_{n} \| \rfloor \, \big| \, \mathbf{h}\big) \le \log n + \log (2E_{0} + 3L) + 2.
\]
This is smaller than $\eta n/3$ for all large enough $n$ (depending on $E_{0}$ and $L$).
Since $H(\mathbf{h}) \le \eta n/3$, we conclude \[
H\big( \lfloor \| Z_{n} \| \rfloor \big) \le H\big( \lfloor \| Z_{n} \| \rfloor \, \big| \, \mathbf{h}\big) + H(\mathbf{h}) \le \eta n. \qedhere
\]
\end{proof}

\begin{cor}\label{cor:sublinearEnt}
Let $(G, X, o)$ be as in Convention \ref{conv:main} and let $\mu$ be a non-elementary probability measure on $\mu$ with finite entropy. Let $\nu$ be a probability measure with finite entropy. Let $\eta>0$, and let $\epsilon' = \epsilon'(\eta, \mu)>0$ and $N = N(\eta, \mu)$ be as in Proposition \ref{prop:sublinearEnt}.

Then for any $g \in G$, for any probability measure $\mu'$ such that $\|\mu' - \mu\|_{0} < \epsilon'$ and $|H(\mu) - H(\mu')| < \epsilon'$, and for any $n > N$, we have \[
H_{ (g_{1}, \ldots, g_{n-1}) \sim \mu'^{n-1}, g_{n} \sim \nu} \big( \lfloor \|Z_{n} \| \rfloor \big) \le \eta n + H(\nu).
\]
\end{cor}

\begin{proof}
This is immediate from \[
H\big( \lfloor \|Z_{n} \| \rfloor \big) = H\big( \lfloor \|Z_{n} \| \rfloor  \, \big| \, g_{n} \big)+ H(g_{n})   \le \eta n + H(\nu).\qedhere
\]
\end{proof}

\begin{lem}\label{lem:almostDistbnEnt}
Let $A$ be a countable set, let $  0<\epsilon < 0.1$ and let $Z, Y : (\Omega, \Prob) \rightarrow A$ be RVs such that \[
(1-\epsilon) \Prob(Z(\w) = a) \le \Prob (Y(\w) = a) \le (1+\epsilon) \Prob(Z(\w) =a) \quad (\forall a \in A).
\]
Then we have $(1- 2\epsilon) H(Z) - 2\epsilon \le H(Y) \le (1+\epsilon) H(Z) + 2\epsilon$.
\end{lem}

\begin{proof}
By the triangle inequality, we have \[\begin{aligned}
H(Y) &= \sum_{a \in A} - \Prob(Y = a) \log \Prob(Y=a) \\
&\le \sum_{a \in A} \Prob(Y=a) | \log \Prob(Y = a) - \log \Prob (Z = a)| + \sum_{g \in G} -  |\Prob(Y = a) - \Prob( Z = a)| \log \Prob(Z = a) \\
&+ \sum_{g \in G} - \Prob( Z = a) \log \Prob(Z = a) \\
&\le \log (1 + \epsilon) \sum_{a \in A} \Prob(Y  = a) + \sum_{a \in A} - \epsilon \Prob(Z = a) \log \Prob(Z=a) + H(Z) \\
&\le 1.2\epsilon + (1+\epsilon) H(Z).
\end{aligned}
\]
Also, the assumption tells us that $\Prob(Z = a)/\Prob(Y=a)$ is bounded between $1-1.2\epsilon$ and $1+1.2\epsilon$ for any $a \in A$. The above argument then tells us that $H(Z) \le 1.44\epsilon + (1+1.2\epsilon)H(Y)$, which implies $H(Y) \ge (1-1.44\epsilon)H(Z) - 2 \epsilon$. 
\end{proof}

\subsection{More precise model   for pivoting} \label{subsection:deliPivot}

We next describe a more complicated version of pivotal time construction in \cite[Section 5C]{gouezel2022exponential}. Recall the notation $Z_{n} = g_{1} \cdots g_{n}$ in   Notation \ref{notat:location}. We also employ the notation $\axes_{i} := (Z_{i} o, \ldots, Z_{i + M_{0}} o)$.
\begin{definition}[{\cite[Definition 6.1]{choi2022random1}}]\label{dfn:pivotalEquivLDP}
Let $\mu$ and $\nu$ be non-elementary probability measures on $G$ and $(\Omega, \Prob)$ be a probability space for $\mu$. Let $0<\epsilon<1$, $K_{0}, N > 0$ and let $S \subseteq (\supp \mu)^{M_{0}}$ be a long enough $K_{0}$-Schottky set.

A subset $\mathcal{E}$ of $\Omega$ is called an \emph{$(n, N, \epsilon, \nu)$-pivotal equivalence class for $\mu$}, associated with the \emph{set of pivotal times} \[
\diffPivot^{(n, N, \epsilon, \nu)}(\mathcal{E}) = \{j(1) < j'(1) <  \ldots < j(\#\diffPivot/2) < j'(\#\diffPivot/2)\} \subseteq M_{0} \Z_{>0},
\]  if the following hold: \begin{enumerate}
\item for each $\w \in \mathcal{E}$ and $k \ge 1$, \[\begin{aligned}
s_{k}(\w) &:= \big(g_{j(k) - M_{0} + 1}(\w), \,\,g_{j(k) - M_{0} + 2}(\w), \,\, \ldots, \,\, g_{j(k)}(\w) \big),\\
s_{k}'(\w) &:=  \big(g_{j'(k) - M_{0} + 1}(\w), \,\,g_{j'(k) - M_{0} + 2}(\w), \,\, \ldots, \,\, g_{j'(k)}(\w) \big)
\end{aligned}
\]
are Schottky sequences;
\item for each $\w \in \mathcal{E}$, $\big(o, \axes_{j(1)}, \axes_{j'(1)}, \ldots, \axes_{j(\#\diffPivot/2)}, \axes_{j'(\#\diffPivot/2)}, Z_{n} o\big)$ is $D_{0}$-semi-aligned;
\item for the RV defined as \[
r_{k} := g_{j(k) + 1} g_{j(k) + 2} \cdots g_{j'(k) - M_{0}},
\]
$(s_{k}, s_{k}', r_{k})_{k > 0}$ on $\mathcal{E}$ are i.i.d.s and $r_{k}$'s are distributed almost according to $\mu^{\ast2 M_{0} N} \ast \nu^{\ast \frac{j'(k) - j(k)}{2M_{0}} - N - 0.5}$ in the sense that\[
(1-\epsilon) (\mu^{\ast 2M_{0} N} \ast \nu^{\ast \frac{j'(k) - j(k)}{2M_{0}} - N - 0.5})(g) \le \Prob(r_{k} = g) \le (1+\epsilon) (\mu^{\ast 2M_{0} N} \ast \nu^{\ast \frac{j'(k) - j(k)}{2M_{0}} - N-0.5})(g).
\]
for each $g \in G$.
\item On $\mathcal{E}$, for $i \notin \{ j(k) - M_{0} + 1 \le l \le   j'(k): k = 1, \ldots, \#\mathcal{P}/2\}$, $g_{i}$ is fixed.
\end{enumerate}
\end{definition}

\begin{prop}[{\cite[Proposition 6.2]{choi2022random1}}]\label{prop:gouezelRWLDP}
Let $\mu$ be a non-elementary probability measure on $G$, let $0<\epsilon   <1 $ and let $S \in G^{M_{0}}$ be a long enough Schottky set for $\mu$ with cardinality greater than $10/\epsilon$. Let $m := \min\{\mu^{M_{0}}(s) : s \in S\}$, let $N > 20/\epsilon m$ and let $\nu$ be a probability measure defined as \[
\nu:= \frac{1}{1 - 0.5m^{2}} \left( \mu^{\ast 2M_{0}} - 0.5 m^{2} \cdot (\textrm{uniform measure on}\, \{ \Pi(s) \Pi(s') : s, s' \in S\})\right).
\]
Then $\nu$ is a non-elementary probability measure. Moreover, there exists $K >0$ depending only on $S$, $m, N$ and $\epsilon$ (but not on $\mu$) such that, for each $n$, we have a probability space $(\Omega, \Prob)$ for $\mu$ and its measurable partition $\mathscr{P}_{n, N, \epsilon, \nu} = \{\mathcal{E}_{\alpha}\}_{\alpha}$ into $(n, N, \epsilon, \nu)$-pivotal equivalence classes, associated with the set of pivotal times $\mathcal{P}^{(n, N, \epsilon, \nu)}$, that satisfies\[
\Prob\left( \w :\frac{1}{2}\# \diffPivot^{(n, N, \epsilon, \nu)}(\w) \le (1-\epsilon) \frac{n}{2M_{0} N} \right) \le K e^{-n/K} \quad (\forall n \in \Z_{>0}).
\]
\end{prop}

Before proving Theorem \ref{thm:entropyContinuity}, we state two lemmata that are well-known to the experts.

\begin{lem}\label{lem:lengthSpec}
Let $g$ be a contracting isometry of a metric space $(X, d)$. Then for each $N$ there exists $K$ such that \[
\#\{ i > 0 : a \le \|g^{i}\| \le a + K\} < N
\]
for every $a \in \mathbb{R}_{>0}$.
\end{lem}

\begin{proof}
Let $g$ be a $K'$-contracting isometry. Since $|d(o, g^{i} o) - d(o, g^{j} o)| \le d(g^{i}o, g^{j} o) \le |i-j| \|g\|$, the set $\{ \|g^{ki}\| : i \in \Z\}$ is $k\|g\|$-coarsely dense in $\{\|g^{i}\| : i \in \Z\}$. Hence, it suffices to prove the statement for a power $g^{k}$ of $g$. We take $k$ large enough such that $\|g^{k}\| \ge 10K'$. Now, for each $i$, the contracting property of the set $\{o, go, \ldots, g^{k(i+1)} o\}$ forces that $[o, g^{k(i+1)} o]$ is $2K'$-close to $g^{ki} o$. Hence we have \[
\|g^{k(i+1)} \| = d(o, g^{k(i+1)} o) \ge d(o, g^{ki} o) + d(g^{ki} o, g^{k(i+1)} o) - 2K' \ge d(o, g^{ki} o) + 8K'.
\]
Now the conclusion follows.
\end{proof}

Given an isometry $g \in G$ of a metric space $(X, d)$, we define its \emph{elementary closure} \[
E(g) := \{ h \in G : d_{\rm Hauss} ( \{h g^{n} o\}_{n \in \Z}, \{g^{n} o\}_{n \in \Z}) < +\infty \}.
\]
  
\begin{lem}[{\cite[Corollary 4.4]{sisto2018contracting}}]\label{lem:eltClosure}
If $g \in G$ is a WPD contracting isometry of $X$, then $\langle g \rangle$ is a finite-index subgroup of $E(g)$. 
\end{lem}

We now prove Theorem \ref{thm:entropyContinuity}.

\begin{proof}
Let $g$ be a WPD contracting isometry of $X$, and let $\mu$ be a probability measure on $G$ whose suitable power has nonzero weight on $g$. Let $\{\mu_{i}\}_{i}$ be a sequence of probability measures that converges simply to $\mu$ with $H(\mu_{i}) \rightarrow H(\mu)$, and let $\epsilon>0$. We first note that $\limsup_{i \rightarrow \infty}   h(\mu_{i}) \le   h(\mu)$: this is the upper-semicontinuity of the asymptotic entropy that holds on any group. See \cite[Proposition 3.3]{amir2013amenability} and the proof of \cite[Theorem 2.9]{gouezel2018entropy}. Hence, the nontrivial part is to prove \begin{equation}\label{eqn:liminfEntropy}
\liminf_{i \rightarrow \infty}   h(\mu_{i}) \ge    h(\mu).
\end{equation} 

We   have  elementary and non-elementary cases. First, consider the case that $\supp \mu$ is contained in the elementary closure $E(g)$ of $g$.
Since $E(g)$ is    virtually cyclic, the $\mu$-random walk  becomes a random walk on an amenable group and its   asymptotic entropy must vanish. In this case, Inequality \ref{eqn:liminfEntropy} is immediate. 

In the case that there exists $h \in \supp \mu \setminus E(g)$, $g^{m}$ and $hg^{m}h^{-1}$ are independent contracting isometries and $\mu$ is non-elementary.

Let $\epsilon>0$ be arbitrary. We will   construct $\epsilon_{meas}>0$ such that Item (2) of the strategy in the beginning of this section holds. For this, let $\eta>0$ be another small positive such that \[
\big(( 1 - 5\eta) h(\mu) - (4H(\mu) +0.5) \eta \big) \cdot (1- \eta) \ge (1- \epsilon) h(\mu).
\]

  For $N=100/\eta$, we use Lemma \ref{lem:SchottkyWPD} to  obtain a finite set $F \subseteq G$, constants $K_{0}, M_{0} > 0$ and a large enough and long $K_{0}$-Schottky set $S \subseteq (\supp \mu)^{M_{0}}$ of cardinality $>100/\eta$ with the $F$-WPD property.
Let $\mathbf{m}$ be the minimum of $\mu^{M_{0}}(s)$ for $s \in S$. 

We now consider the condition \begin{equation}\label{eqn:muClosClose}
\|\mu' - \mu\|_{0} < \epsilon_{meas}\,\,\textrm{and}\,\, |H(\mu') - H(\mu) | < \epsilon_{meas} 
\end{equation} for some $\epsilon_{meas}> 0$. There exists a small enough choice of $\epsilon_{meas}$, accompanied with a threshold $N_{sub}$ such that for any $g \in G$, for any $n > N_{sub}$ and for any $\mu'$ satisfying Equation \ref{eqn:muClosClose}, \[
H_{(g_{1}, \ldots, g_{n-1}) \sim \mu'^{n-1}, g_{n} = g} \big(\lfloor d(o, g_{1} \cdots g_{n} o) \rfloor\big) < \frac{\eta}{10}n.
\]
Such $\epsilon_{meas}$ and $N_{sub}$   are  given by Proposition \ref{prop:sublinearEnt}. We now fix $N > 40/\eta \mathbf{m} + N_{sub} + (\log \#S)/\eta M_{0}$ that is also large enough such that: \[
H(\mu^{\ast 2M_{0} N}) \ge(1-\eta) \cdot  2M_{0} N  h(\mu) 
\]
and such   that \begin{equation}\label{eqn:logLargeClaim}
\frac{2\log \# S + 4 \log (6E_{0} + 6)}{M_{0} N} \le 0.1 \eta.
\end{equation}
Choosing some smaller $\epsilon_{meas}$ if necessary, we can also guarantee the following for the measures $\mu'$ satisfying Equation \ref{eqn:muClosClose}: \[ 
0.5 \mathbf{m} \le \min \mu'^{M_{0}}(s) \le 2\mathbf{m}, \quad H(\mu') < 2H(\mu), \quad H(\mu'^{\ast 2M_{0} N}) > (1-\eta) H(\mu^{\ast 2M_{0} N}).
\]
In the remaining, we will deal with an arbitrary probability measure $\mu'$ satisfying Equation \ref{eqn:muClosClose}.

  Let also $m = \min_{s \in S} \mu'^{M_{0}}(s)$ and \[
\nu':= \frac{1}{1 - 0.5m^{2}} \left( \mu'^{\ast 2M_{0}} - 0.5 m^{2} \cdot (\textrm{uniform measure on}\, \{ \Pi(s) \Pi(s') : s, s' \in S\})\right).
\]
For this choice, note that  \[
H(\nu') \le \frac{1}{1-0.5m^{2}} \cdot H(\mu'^{2M_{0}}) \le 2 \cdot 2M_{0}H(\mu') \le 8M_{0} H(\mu).
\]

By Proposition \ref{prop:gouezelRWLDP},  for each $n$ there exists a measurable partition $\mathscr{P}_{n, N, \eta, \nu'} = \{\mathcal{E}_{\alpha}\}_{\alpha}$ of $(G^{n}, \mu'^{n})$ into $(n, N, \epsilon, \nu)$-pivotal equivalence classes such that $\frac{1}{2} \#\mathcal{P}^{(n, N, \eta, \nu')}(\w) > (1-\eta) n/2M_{0} N$ holds for probability $1 - K e^{-n/K}$. Here, note that $K$ does not depend on the nature of $\mu$ but only on our choice of $S$, $\mathbf{m}$, $N$ and $\eta$. Hence, we can choose $\mathbf{n}$ only depending on $S, \mathbf{m}, N, \eta$ such that \begin{equation}\label{eqn:suffPivEquiv}
\Prob \Big( \frac{1}{2} \#\mathcal{P}^{(n, N, \eta, \nu')}(\w) > (1-\eta) n/2M_{0} N\Big) \ge 1-\eta
\end{equation}
holds for all $n > \mathbf{n}$.

We now condition on a $(n, N, \eta, \nu)$-pivotal equivalence class $\mathcal{E} \in \mathscr{P}_{n, N, \epsilon, \nu}$. For convenience, we use notation $T := \frac{1}{2} \#\mathcal{P}^{(n, N, \eta, \nu)}(\mathcal{E})$, i.e., \[
\mathcal{P}(\mathcal{E}) = \{j(1) < j'(1) < \ldots < j(T) < j'(T)\}.
\] Note that $T \le n/2M_{0} N$ always hold, as \[
n > \sum_{k=1}^{T} \big(j'(k) - j(k) \big) > 2M_{0} N \cdot T.
\]
Our first claim is 
\begin{claim}\label{claim:entropyContinuityC1}
For each $i =1, \ldots, T$, we have \[
H(r_{i} | \mathcal{E}) \ge (1-2\eta) H(\mu'^{\ast 2M_{0} N}) - 2\eta
\]
\end{claim}

\begin{proof}[Proof of Claim \ref{claim:entropyContinuityC1}]
Let $f_{i}$ be the distribution of $\mu'^{\ast 2M_{0} N} \ast \nu'^{\ast \frac{j'(k) - j(k)}{2M_{0}} - N-0.5}$, i.e., $f_{i}(g) = \mu'^{\ast 2M_{0} N} \ast \nu^{\ast \frac{j'(k) - j(k)}{2M_{0}} - N-0.5}(g)$. Note that the entropy of $f_{i}$ is at least $H(\mu^{\ast 2M_{0} N})$, because \[
H(Y_{1}Y_{2}) \ge \E[H(Y_{1} Y_{2} | Y_{2})] \ge H(Y_{1})
\]
holds for every pair of RVs $Y_{1}$ and $Y_{2}$ on $G$. Now, since $r_{i}$ is almost distributed according to $f_{i}$, Lemma \ref{lem:almostDistbnEnt} tells us the $H(r_{i} | \mathcal{E}) \ge (1-2\eta) H(\mu'^{\ast 2 M_{0} N}) - 2\eta$ as desired.
\end{proof}

Note that\[
Z_{j'(k)} = h_{0} \Pi(s_{1}) r_{1}  \Pi(s_{1})'h_{1} \cdots h_{k-1} \Pi(s_{k}) r_{k}  \quad (h_{0} := Z_{j(1) - M_{0}}, h_{i} := Z_{j'(i)}^{-1}   Z_{j(i+1) - M_{0}}\,\,\textrm{for $i\ge 1$})
\]
Since we are discussing everything inside $\mathcal{E}$, $h_{i}$'s are always fixed. 
Let $\mathscr{F}$ be the partition of $\mathcal{E}$ based on the values of \[
\big\{ \big\lfloor \| Z_{j(k) - M_{0}} \| \big\rfloor, \big\lfloor \| Z_{j(k)} \| \big\rfloor,   \big\lfloor \| Z_{j'(k)} \| \big\rfloor , \big\lfloor \| Z_{j'(k) +M_{0}} \| \big\rfloor  : k =1, \ldots, T \big\}.
\]

\begin{claim}\label{claim:entropyContinuityC2}
We have \[
H(\mathscr{F} | \mathcal{E}) \le \frac{\eta n}{4} + 4H(\mu) \cdot (n - 2M_{0} N T).
\]
\end{claim}

\begin{proof}[Proof of Claim \ref{claim:entropyContinuityC2}]
Let $k \in \{1, \ldots, T\}$. Let $Z$ be an RV distributed according to $\mu'^{\ast 2M_{0} N} \ast \nu'^{\ast \frac{j'(k) - j(k)}{2M_{0}} - N-0.5}$. Then the ratio of the  distributions of $\lfloor \| Z \| \rfloor$ and $\lfloor\|r_{k}\| \rfloor$ is uniformly between $(1-\eta)$ and $(1+\eta)$. By Lemma \ref{lem:almostDistbnEnt}, we have \[
H\big(\lfloor\|r_{k}\| \rfloor \, \big| \, \mathcal{E}\big) \le (1+ 2 \eta) H(\lfloor \| Z \| \rfloor) + 2\eta. 
\]
Meanwhile, since $2M_{0} N > N > N_{sub}$, Proposition \ref{prop:sublinearEnt} and Corollary \ref{cor:sublinearEnt} tells us that \[\begin{aligned}
H(\lfloor \| Z \| \rfloor) &\le \frac{\eta}{10} \cdot 2M_{0} N + H(\nu'^{\ast \frac{j'(k) - j(k)}{2M_{0}} - N-0.5}) \\
&\le \frac{\eta}{10} \cdot 2M_{0} N  + 8M_{0} H(\mu) \cdot \left(\frac{j'(k) - j(k)}{2M_{0}} - N-0.5\right).
\end{aligned}
\]
Combined together, we have \[
H\big( \lfloor \|r_{k} \| \rfloor  \, \big| \, \mathcal{E}\big) \le \frac{\eta}{4} M_{0} N + 5M_{0} H(\mu)\cdot (j'(k) - j(k) - 2M_{0} N - M_{0}) + 2 \eta.
\]

Let us now extract the information \[
\big\{ \big\lfloor \| Z_{j(k) - M_{0}} \| \big\rfloor, \big\lfloor \| Z_{j(k)} \| \big\rfloor,   \big\lfloor \| Z_{j'(k)} \| \big\rfloor , \big\lfloor \| Z_{j'(k) +M_{0}} \| \big\rfloor  : k =1, \ldots, T \big\}.
\] from the RVs $\lfloor \|r_{k} \| \rfloor, s_{k}, s_{k}'$'s. Proposition \ref{prop:BGIPWitness} asserts that  $[o, Z_{n} o]$ passes through the $E_{0}$-neighborhoods of $Z_{j(1)-M_{0}}, Z_{j(1)}, Z_{j'(1)}, Z_{j'(1) + M_{0}}, \ldots, Z_{j(T)-M_{0}}, Z_{j(T)}, Z_{j'(T)}, Z_{j'(T) + M_{0}}$ in order. Hence, setting $j'(-1) := -M_{0}$ for convenience, the following quantities: \[\begin{aligned}
 \|Z_{j'(k-1)+M_{0}}\| + \|Z_{j'(k-1) + M_{0}}^{-1} Z_{j(k) - M_{0}} \| -\|Z_{j(k)-M_{0}}\|,\\ 
\|Z_{j(k) - M_{0}}\| + \|s_{k}\| -\|Z_{j(k)} \|, \\
 \|Z_{j(k)}\| + \|r_{k}\|-\|Z_{j'(k)} \|, \\
\|Z_{j'(k)}\| + \|s_{k}'\|- \|Z_{j'(k) + M_{0}} \| 
\end{aligned} \quad (k=1, \ldots, T)
\]are all real numbers between $0$ and $6E_{0}$. Hence, if we define $\alpha_{k}, \beta_{k}, \gamma_{k}, \delta_{k}$ so that   \[\begin{aligned}
\left\lfloor \|Z_{j(k)-M_{0}}\| \right\rfloor &= \left\lfloor\|Z_{j'(k-1)+M_{0}}\|  \right\rfloor + \left\lfloor\|Z_{j'(k-1) + M_{0}}^{-1} Z_{j(k) - M_{0}} \|  \right\rfloor + \alpha_{k}, \\ 
\left\lfloor\|Z_{j(k)} \|\right\rfloor &=\left\lfloor \|Z_{j(k) - M_{0}}\|  \right\rfloor + \left\lfloor\|s_{k}\|  \right\rfloor + \beta_{k}, \\
\left\lfloor\|Z_{j'(k)} \|\right\rfloor &=\left\lfloor \|Z_{j(k)}\|  \|\right\rfloor+ \left\lfloor\|r_{k}\| \right\rfloor  + \gamma_{k}, \\
\left\lfloor\|Z_{j'(k) + M_{0}}  \|\right\rfloor &= \left\lfloor\|Z_{j'(k)}  \|\right\rfloor+ \left\lfloor\|s_{k}'\|   \|\right\rfloor + \delta_{k}
\end{aligned}
\]
then $\alpha_{k}, \beta_{k}, \gamma_{k}, \delta_{k}$'s are integers between $3$ and $-(6E_{0} + 3)$. Here, note that $Z_{j'(k-1) + M_{0}}^{-1} Z_{j(k) - M_{0}}$'s are fixed information across the pivotal equivalence class $\mathcal{E}$. Hence, as soon as we have information $(\lfloor \|r_{k}\|\rfloor, s_{k}, s_{k}', \alpha_{k}, \beta_{k}, \delta_{k}, \delta_{k})_{k=1}^{T}$, we determine all information for $\mathscr{F}$. 
We conclude that  \[\begin{aligned}
&H\big(\mathscr{F}\, \big|\,\mathcal{E}\big)\\
&\le H \Big( (\lfloor \|r_{k}\|\rfloor, s_{k}, s_{k}', \alpha_{k}, \beta_{k}, \delta_{k}, \delta_{k})_{k=1}^{T} \, \Big| \, \mathcal{E} \Big) \\
&\le \sum_{k=1}^{T} H( \lfloor \|r_{k} \| \rfloor)  + \sum_{k=1}^{T} H(s_{k}) + \sum_{k=1}^{T} H(s_{k}) \\
& \, + \sum_{k=1}^{T} H(\alpha_{k}) + \sum_{k=1}^{T} H(\beta_{k})+\sum_{k=1}^{T} H(\gamma_{k})+\sum_{k=1}^{T} H(\delta_{k}) \\
&\le T \cdot \frac{\eta M_{0}  N}{4} + 4H(\mu) \cdot \# \big(\{1, \ldots, n\} \setminus \{j(k)+ 1, \ldots, j(k) + 2M_{0} N : k=1, \ldots, T\}\big)\\
& \, + T ( 2\log \# S + 4 \log (6E_{0} + 6)) \\
&\le \frac{\eta n}{8} + 4H(\mu) \cdot (n - 2M_{0} N \cdot T) + \frac{\eta n}{8}.
\end{aligned}
\]
In the final step, we used the fact that $\{j(k) + 1, \ldots, j'(k)-M_{0}\}$'s for $k=1, \ldots, T$ are disjoint intervals containing at least $2M_{0}N$ integers, and that $N$ satisfies   Inequality \ref{eqn:logLargeClaim}.
\end{proof}

Now, we claim: \begin{claim}\label{claim:entropyContinuityC4}
Conditioned on each equivalence class $E$ of $\mathscr{F}$ (that partitions $\mathcal{E}$), we have \[
H( Z_{n} \, | \,E) \ge H\big( (r_{1}, \ldots, r_{T}) \, \big| \, E \big) - 2 T \log \# F.
\]
\end{claim}

\begin{proof}[Proof of Claim \ref{claim:entropyContinuityC4}]
Fix an equivalence class $E$, an element $g\in G$ and an element $\w \in E$ such that $Z_{n}(\w) = g$. We claim that if $\w' \in E$ satisfies $Z_{n}(\w') = g$, then $r_{i}(\w') \in F \cdot r_{i}(\w) \cdot  F$ for each $i=1, \ldots, T$. To see this, pick an arbitrary $k \in \{1, \ldots, T\}$. Recall that \[
\big(o, \,\mathbf{Y}_{j(k)}(\w),  \,\mathbf{Y}_{j(k)}(\w), \, g o\big),\quad 
\big(o, \,\mathbf{Y}_{j(k)}(\w'),  \,\mathbf{Y}_{j(k)}(\w'), \, g o\big)
\]
are each $D_{0}$-semi-aligned. Hence, there exists subsegments $\eta_{k}, \eta_{k}''$ of $[o, go]$ such that $\eta_{k}$ is $0.1E_{0}$-fellow traveling with $\mathbf{Y}_{j(k)}(\w)$ and    $\eta_{k}'$ is $0.1E_{0}$-fellow traveling with  $\mathbf{Y}_{j(k)}(\w')$.  Let us  denote the beginning point $\eta_{k}$ and $\eta_{k}'$ by $p$ and $p'$, respectively. Since    we chose $\w, \w'$ from the same equivalence class $E$ of $\mathcal{F}$, $\|Z_{j(k) - M_{0}}(\w)\|$ and $\|Z_{j(k) - M_{0}}(\w)\|$ have the same integer part and they differ by at most $1 < E_{0}$. This implies that $d(o, p)$ and $d(o, p')$ differ by at most $0.1E_{0} + E_{0} + 0.1E_{0} < 2E_{0}$. Meanwhile, $p$ and $p'$ are on the same geodesic starting from $o$. It follows that $d(p, p') < 2E_{0}$. Hence, $Z_{j(k) - M_{0}}(\w)$ and $Z_{j(k) - M_{0}}(\w')$ are $(2E_{0} + 0.1E_{0} + 0.1E_{0})$-close. For the same reason, the endpoints of $\eta_{k}$ and $\eta_{k}'$ are $2E_{0}$-close. This in turn implies that $Z_{j(k)}(\w)$ and $Z_{j(k)}(\w')$ are $3E_{0}$-close.

We are now in the situation that the beginning and the ending points of \[
\Gamma\big(s_{k}(\w)\big), \quad \big(Z_{j(k)-M_{0}}(\w) \big)^{-1} Z_{j(k)-M_{0}}(\w') \Gamma\big(s_{k}(\w') \big)
\]
are pairwise $3E_{0}$-close. Since $S$ is a $F$-WPD $K_{0}$-Schottky set, we conclude that \[
\big(Z_{j(k)-M_{0}}(\w) \big)^{-1} Z_{j(k)-M_{0}}(\w') \in F, \quad \big(Z_{j(k)}(\w) \big)^{-1} Z_{j(k)}(\w') \in F. \quad(k=1, \ldots, T) 
\]

For a similar reason, we have \[
\big(Z_{j'(k)}(\w) \big)^{-1} Z_{j'(k)}(\w') \in F, \quad \big(Z_{j(k)+M_{0}}(\w) \big)^{-1} Z_{j'(k)+M_{0}}(\w') \in F. \quad(k=1, \ldots, T) 
\]
We then conclude that \[\begin{aligned}
r_{k}(\w') &= \big( Z_{j(k) + M_{0}}(\w') \big)^{-1} Z_{j'(k)}(\w') \\
&=  \big( Z_{j(k) + M_{0}}(\w') \big)^{-1} Z_{j(k) + M_{0}}(\w) \cdot  \big( Z_{j(k) + M_{0}}(\w) \big)^{-1} Z_{j'(k)}(\w) \cdot  \big( Z_{j'(k)}(\w) \big)^{-1} Z_{j'(k)}(\w') \\
&\in F \cdot  \big( Z_{j(k) + M_{0}}(\w) \big)^{-1} Z_{j'(k)}(\w) \cdot F = F \cdot r_{k}(\w) \cdot F.
\end{aligned}
\]

This claim leads to the fact that there are at most $(\#F)^{2}$ choices of $r_{i}$ for each $i$  when conditioned with the value of $Z_{n}$. The entropy of such choices is at most $T \cdot \log (\#F)^{2}$. Now recall Bayes' rule: \[
H\big( Z_{n} \, \big| \, (r_{i})_{i=1}^{T}, E\big) = 
H\big((r_{i})_{i=1}^{T}\, \big | \, Z_{n}, E\big) - H\big( (r_{i})_{i=1}^{T}\, \big| \, E\big) + H(Z_{n} | E).
\]
By rearranging, we obtain \[\begin{aligned}
H(Z_{n} | E) &= 
H\big( Z_{n} \, \big| \, (r_{i})_{i=1}^{T}, E\big)  + H\big( (r_{i})_{i=1}^{T}\, \big| \, E\big) - H\big((r_{i})_{i=1}^{T}\, \big | \, Z_{n}, E\big) \\
&\ge H\big( (r_{i})_{i=1}^{T}\, \big| \,E\big) - 2T \log \# F. \qedhere
\end{aligned}
\]
\end{proof}

Let us now put the claims together. Let $\mathcal{E}$ be a pivotal equivalence class with $\#\frac{1}{2} \# \mathcal{P}^{(n, N, \eta, \nu')}(\mathcal{E}) = T$. Recall the partition    $\mathscr{F}$ made on $\mathcal{E}$. We have \[
\begin{aligned}
H(Z_{n} \, | \, \mathcal{E}) &\ge H(Z_{n} \, | \, \mathscr{F}) \\
&\ge H\big( (r_{1}, \ldots, r_{T}) \, \big| \, \mathscr{F}) - 2T \log \# F \\
&\ge H\big( (r_{1}, \ldots, r_{T}) \, \big| \, \mathcal{E}) - H(\mathscr{F} \, \big| \, \mathcal{E}) - 2T \log \# F \\
&= \sum_{i=1}^{T} H(r_{i} | \mathcal{E}) - 2T \log \# F - \frac{\eta n}{4} - 4H(\mu) (n-2M_{0} N T) \\
&\ge T(1- 2\eta) H(\mu'^{2M_{0} N}) - 2T\eta - \frac{n}{M_{0} N} \log \# F - \frac{\eta n}{4} - 4H(\mu) (n-2M_{0} N T)\\
&\ge T(1- 2\eta) (1-2\eta) 2M_{0} N \cdot h(\mu) -  \frac{\eta n}{2} - 4H(\mu) (n-2M_{0} N T) - 2\eta.
\end{aligned}
\]
Here we used the fact that $N >5 (\log \# F)/M_{0} \eta$.

Now consider the collection \[
   \mathcal{Q} := \big\{\mathcal{E} \in  \mathscr{P}_{n, N, \eta, \nu'} : T = \frac{1}{2} \# \mathcal{P}^{(n, N, \eta, \nu')}(\mathcal{E}) > (1-\eta) \frac{n}{2M_{0} N} \big\}
\]
When $\mathcal{E}$ belongs to $\mathcal{Q}$, the above estimate is greater than \[
(1 - 5 \eta) n h(\mu) - (4H(\mu) +0.5) \eta n.
\]
It remains to sum up the conditional entropies on each $\mathcal{E} \in \mathscr{P}_{n, N, \eta, \nu'}$. Equation \ref{eqn:suffPivEquiv} guarantees that equivalence classes in    $\mathcal{Q}$ take up probability at least $1-\eta$. This tells us that \[\begin{aligned}
H(Z_{n}) &\ge H(Z_{n} | \mathscr{P}_{n, N, \eta, \nu'}) \\
&= \sum_{\mathcal{E} \in \mathscr{P}_{n, N, \eta, \nu'}} H(Z_{n} | \mathcal{E}) \cdot \Prob(\mathcal{E}) \\
&\ge \sum_{\mathcal{E} \in \mathcal{Q}} H(Z_{n} | \mathcal{E}) \cdot \Prob(\mathcal{E}) \\
&\ge  \sum_{\mathcal{E} \in \mathcal{Q}} (( 1 - 5\eta) n h(\mu) - (4H(\mu) +0.5) \eta n) \cdot \Prob(\mathcal{E})\\
&\ge (( 1 - 5\eta) n h(\mu) - (4H(\mu) +0.5) \eta n) \cdot (1- \eta)    \ge (1-\epsilon) nh(\mu).
\end{aligned}
\]
This is as we wanted.
\end{proof}

\medskip
\bibliographystyle{alpha}
\bibliography{Continuity_rate}

\end{document}